\documentclass[11pt]{article}
\usepackage{fullpage}
\usepackage{amsmath}		
\usepackage{amsthm}
\usepackage{amssymb}
\usepackage{color,graphicx} 	
\usepackage{enumerate} 		
\usepackage{esint} 		
\usepackage{subfigure}		
\usepackage{url}		

\numberwithin{equation}{section}

\newtheorem{theorem}{Theorem} \newtheorem{proposition}{Proposition}[section] 
\newtheorem{corollary}{Corollary} \newtheorem{definition}{Definition} 
\newtheorem{lemma}[proposition]{Lemma} 
{\theoremstyle{remark} \newtheorem{remark}{Remark} 
\newtheorem*{remark*}{Remark} }

\newcommand{\dd}{\, \mathrm{d}}
\renewcommand{\div}{\operatorname{div}}
\newcommand{\ep}{\varepsilon}

\newcommand{\id}{\mathbf{id}}
\newcommand{\N}{\mathbb{N}}
\newcommand{\om}{\Omega}
\newcommand{\mr}{\mathbb{R}}
\newcommand{\p}{\partial}
\newcommand{\R}{\mathbb{R}}

\newcommand{\res}{\mathop{\hbox{\vrule height 7pt width .5pt depth 0pt \vrule height .5pt width 6pt depth 0pt}}\nolimits\,}
\newcommand{\ssubset}{\subset \! \subset}
\renewcommand{\S}{\mathbb{S}}
\newcommand{\vacio}{\varnothing} 
\renewcommand{\vec}{\mathbf}
\newcommand{\vecg}{\boldsymbol}
\newcommand{\vu}{\vec u}

\newcommand{\weakc}{\rightharpoonup}
\newcommand{\weakcs}{\overset{*}{\rightharpoonup}}

\DeclareMathOperator{\adj}{adj}   
\DeclareMathOperator{\cof}{cof}   
\DeclareMathOperator{\diam}{diam} \DeclareMathOperator{\dist}{dist} \DeclareMathOperator{\Det}{Det} 
\DeclareMathOperator{\Div}{Div}  \DeclareMathOperator{\imT}{im_T} \DeclareMathOperator{\Int}{Int} \DeclareMathOperator{\loc}{loc} \DeclareMathOperator{\Per}{Per} \DeclareMathOperator{\sgn}{sgn} \DeclareMathOperator{\spt}{spt} 

\def\XXint#1#2#3{{\setbox0=\hbox{$#1{#2#3}{\int}$}
\vcenter{\hbox{$#2#3$}}\kern-.5\wd0}}

\title{Energy estimates and cavity interaction
for a critical-exponent cavitation model}

\author{Duvan Henao \and Sylvia Serfaty}

\begin{document}
\maketitle

\begin{abstract}
 We consider the minimization of $\int_{\Omega_{\ep}} |D\vec u|^p \dd\vec x$
in a perforated domain $\Omega_{\ep}:= \Omega \setminus \bigcup_{i=1}^M B_{\ep}(\vec a_i)$ of $\R^n$,
among maps $\vec u \in W^{1,p}(\Omega_{\ep}, \R^n)$ that are incompressible ($\det D\vec u\equiv 1$),
invertible, and satisfy a Dirichlet boundary condition $\vec u= \vec g$ on $\partial \Omega$.
If the volume enclosed by $\vec g (\partial \Omega)$ is greater than $|\Omega|$,
any such deformation $\vec u$ is forced to map the small holes $B_{\ep}(\vec a_i)$ onto
macroscopically visible cavities (which do not disappear as $\ep\to 0$).
We restrict our attention to the critical exponent $p=n$,
where the energy required for cavitation is of the order of $\sum_{i=1}^M v_i |\log \ep|$
and the model is suited, therefore, for an asymptotic analysis
($v_1, \ldots, v_M$ denote the volumes of the cavities).
In the spirit of the analysis of vortices in Ginzburg-Landau theory, 
we obtain estimates for the ``renormalized" energy
$\frac{1}{n}\int_{\Omega_{\ep}} \left |\frac{D\vec u}{\sqrt{n-1}}\right |^p \dd\vec x - \sum_i v_i |\log \ep|$,
showing its dependence on the size and the shape of the cavities,
on the initial distance between the cavitation points $\vec a_1, \ldots, \vec a_M$,
and on the distance from these points to the outer boundary $\partial \Omega$.
Based on those estimates we conclude, for the case of two cavities, that either the cavities
prefer to be spherical in shape and well separated, 
or to be very close to each other and appear as a single equivalent round cavity.
This is in agreement with existing numerical simulations, and is reminiscent of the interaction between cavities
in the mechanism of ductile fracture by void growth and coalescence.
\end{abstract}

\section{Introduction}

\subsection{Motivation}
In nonlinear elasticity, cavitation is the name given to the sudden formation of cavities in
an initially perfect material, due to its incompressibility
(or near-incompressibility), in response to a sufficiently large and triaxial tension.
It plays a central role in the initiation of fracture
in metals \cite{Gurson77,RiTr69,GoBr79,Tvergaard90,PeCuSiEl06}
and in elastomers \cite{Gent91,WiSc65,GeWa91,Dorfmann03,CrMaLoHuStCr10} (especially in
reinforced elastomers \cite{ObBr65,GePa84,ChGeLa87,BaBeBa08,MiLoPoTr10}),
via the mechanism of void growth and coalescence.
It has important applications such as the indirect
measurement of mechanical properties \cite{KuCr09}
or the rubber-toughening of brittle polymers \cite{LaBu95,ChHiBaSoMy95,StVDGi99,LiLi00}.
Mathematically, it constitutes a realistic example of a regular variational problem with singular minimizers,
and corresponds to the case when the stored-energy function of the material is not $W^{1,p}$-quasiconvex
 \cite{Ball77,Ball84,BaMu84}, the Jacobian determinant is not weakly continuous \cite{BaMu84},
and important properties such as the invertibility of the deformation may not pass to the weak limit
\cite[Sect. 11]{MuSp95}.
The problem has been studied by many authors, beginning with Gent-Lindley \cite{GeLi59}
and Ball \cite{Ball82}; see the review papers \cite{Gent91,HoPo95,Fond01},
the variational models of M\"uller-Spector \cite{MuSp95} and Sivaloganathan-Spector \cite{SiSp00},
and the recent works \cite{HeMo10,LoIdNa11} for further motivation and references.

The standard model in the variational approach to cavitation considers functionals of the form
\begin{equation}\label{standardpmod}
\int_\om |D \vu|^p\, \dd\vec x,
\end{equation}
where the deformation $\vu: \Omega \subset \R^n \to \R^n$ is constrained to be incompressible (i.e. $\det D\vu=1$) and globally invertible,
and either a Dirichlet condition $\vec u=\vec g$ or a force boundary condition is applied.
Unless the boundary condition is exactly compatible with the volume, cavities have to be formed.
If $p<n$  this can happen while still keeping a finite energy.
 A typical deformation creating a cavity
of volume $\omega_n A^n$
 at the origin
($\omega_n$ being the volume of the unit ball in $\R^n$) is given by
\begin{equation}\label{cavitybehavior}
\vu(\vec x)= \sqrt[n]{ A^n + |\vec x|^n} \frac{\vec x}{|\vec x|}.
\end{equation}
We can easily compute that \begin{equation}\label{cavbehaviorDu}
|D\vu|^2\mathop{\sim}_{\vec x=\vec 0} \frac{(n-1) A^2}{|\vec x|^2}  .\end{equation}
In that situation the origin  is called a {\it cavitation point}, which belongs to the domain space, and its image by $\vu$ is
the {\it cavity}, belonging to the target space.
Contrarily to the usual, we study the critical  case $p=n$  where the cavity behaviour \eqref{cavitybehavior} just fails to be of finite energy.

This fact is analogous to what happens for $\mathbb{S}^1$-valued harmonic maps in dimension 2,
which were  particularly studied in the context of the Ginzburg-Landau model, see Bethuel-Brezis-H\'elein \cite{BrBeHe94}.
For $\mathbb{S}^1$-valued maps $\vec u$ from $\om \subset \mathbb{R}^2$, the topological degree of $\vec u$
around a point $\vec a$ is defined by the following integer
$$d= \frac{1}{2\pi} \int_{\partial B(\vec a, r)}  \frac{\p \vec u}{\p \tau}\times \vec u .$$
Points around which  this is not zero are called vortices.
 Typical vortices   of degree $d$ look like  $\vec u=e^{id\theta}$ (in polar coordinates).
If $d\neq 0$ again $|D\vec u|^2$ just fails to be integrable since
for the typical vortex $|D\vec u|^2\sim_{\vec x=\vec 0}\frac{|d|^2}{|\vec x|^2}$, just as above \eqref{cavbehaviorDu},
up to a constant factor.
So there is an analogy in that sense between maps from $\om $ to $\mathbb{C}$ which are constrained to satisfy $|\vec u|=1$,
 and maps from $\om $ to $\mr^2$  which satisfy the incompressibility constraint $\det D\vu=1$.
We see that in this analogy (in dimension 2) the volume of the cavity divided by $\pi$ plays the role of
the absolute value of the degree for $\mathbb{S}^1$-valued maps.
In this correspondence two important differences appear:
the degree is
 quantized while the cavity volume is not; on the other hand the degree has a sign, which can lead to  ``cancellations" between vortices, while the cavity volume is always positive.

In the context of $\mathbb{S}^1$-valued maps, two possible ways of giving a meaning to
$\int_{\om} |D\vec u|^2$ are the following.
The first is to  relax the constraint $|\vec u|=1$ and replace it by a penalization, and study instead
\begin{equation}
\label{gl}
\int_{\om} |D\vec u|^2 + \frac{1}{\ep^2} (1-|\vec u|^2)^2\end{equation}
in the limit $\ep \to 0 $; this is the Ginzburg-Landau approximation.
The second is to study the energy with the constraint $|\vec u|=1$ but in  a punctured domain $\om_\ep:=\om\backslash \cup_i B({\vec a}_i, \ep)$ where ${\vec a}_i$'s stand for the vortex locations:
\begin{equation}\label{harmmap}
\min_{|\vec u|=1}  \int_{\om_\ep} |D\vec u|^2 \end{equation}
again in the limit $\ep\to 0$; this can be called the ``renormalized energy approach".
 Both of these approaches were followed in
\cite{BrBeHe94}, where it is proven  that the Ginzburg-Landau approach essentially reduces to the renormalized energy approach.
 More specifically, when there are vortices at ${\vec a}_i$,
 $|D\vec u|$ will behave like $|d_i|/|\vec x-{\vec a}_i|$
near each vortex (where $d_i$ is the degree of the vortex)  and both energies \eqref{gl} and \eqref{harmmap} will blow up
like $\pi \sum_i d_i^2 \log \frac{1}{\ep}$ as $\ep \to 0$. It is shown in
\cite{BrBeHe94} that when this divergent term is subtracted off
 (this is the ``renormalization" procedure), what remains is a nondivergent term
depending on the positions of the vortices ${\vec a}_i$ and their degrees $d_i$ (and the domain), called precisely
 the renormalized energy. That energy is essentially a Coulombian interaction between the points ${\vec a}_i$ behaving like
 charged particles (vortices of same degree repel, those of opposite degrees attract) and it can be
written down quite explicitly.

  Our goal here is to study cavitation in the same spirit.  A first attempt, which would be the analogue of \eqref{gl},
 would be to relax the incompressibility constraint and  study for example
  \begin{equation}\label{glcav}
  \int_{\om} |D\vu|^2 + \frac{(1-\det D\vu)^2}{\ep}.\end{equation}
  We do not however follow this route which seems to present many difficulties
(one of them is that this energy in two dimensions is scale invariant, and that contrarily to \eqref{gl}
the nonlinearity contains as many order of derivatives as the other term),
but it remains a seemingly interesting open problem, which would have good physical sense. Rather we follow the second approach, i.e.  that of working  in punctured domains while  keeping the incompressibility constraint.

  For the sake of generality we consider holes which can be of different radii $\ep_1, \cdots, \ep_m$,
define $\om_\ep:=\om \backslash \cup_{i=1}^m \overline{B} (\vec a_i, \ep_i)$ and look at
  \begin{equation}\label{elas}
  \min_{\det D\vu=1} \int_{\om_\ep} |D\vu|^2
\end{equation}
(or $\min_{\det D\vec u =1} \int_{\Omega_{\ep}} |D\vec u|^n$ in dimension $n$),
in the limit $\ep \to 0$.
  This also has a reasonable physical interpretation: it corresponds to studying the incompressible deformation
of a body that contains micro-voids which expand under the applied boundary deformation.
One may think of the points ${\vec a}_i$ as fixed, then they correspond to defects that pre-exist, just as above.
Or the model can be seen as a fracture model where we postulate that the body will first break around
the most energetically favorable  points ${\vec a}_i$
(see, e.g., the discussion in
\cite{Ball82,HoAb86,Sivaloganathan86,Gent91,Horgan92,MuSp95,Ball96,SiSp02,SiSp06,LoIdNa11,LoNaId11}).
It can also be compared to the core-radius approach in dislocation models \cite{CeLe05,Ponsiglione07,GaLePo10}.

Following the analogy above, we would like to  be able to subtract from \eqref{elas} a   leading order  term proportional to $\log \frac{1}{\ep}$, in order to extract at the next order a ``renormalized" term which will tell us how cavities ``interact" (attract or repel each other), according to their positions and shapes.
  This is more difficult than the problem  \eqref{harmmap} because the condition
   $\det D\vu=1$ is much less constraining than $|\vec u|=1$. While
the maps with $|\vec u|=1$ can be parametrized by lifting in the form  $\vec u=e^{i\varphi}$,
to our knowledge no parametrization of that sort exists for incompressible maps.
In addition while the only characteristic of a vortex is an integer --its degree--, for incompressible maps,
the characteristics of a cavity are more complex --they comprise the volume of the cavity and its shape--,
and there is no quantization.
For these reasons we cannot really hope for something as nice and explicit as a complete ``renormalized energy" for this toy cavitation model. However we will show that we can obtain, in particular in the case of two cavities, some quantitative information about the cavities interaction that is reminiscent of the renormalized energy.

\subsection{Method and main results : energy lower bounds}\label{sec:1.2}

Our method relies on obtaining general and ansatz-free lower bounds for the energy on the one hand, and on the other hand upper bounds via explicit constructions, which match as much as possible the lower bounds. This is in the spirit of $\Gamma$-convergence (however we will not prove a complete $\Gamma$-convergence result).
For simplicity in this section we present the results in dimension 2, but they carry over in higher dimension.

To obtain lower bounds we use the ``ball construction method",
which was introduced in the context of Ginzburg-Landau by Jerrard \cite{Jerrard99} and Sandier \cite{Sandier98,Sandier00}.
 The crucial estimate for Ginzburg-Landau, or more simply $\mathbb{S}^1$-valued harmonic maps,
 is the following simple relation, corollary of Cauchy-Schwarz:
\begin{equation}\label{csgl}
\int_{\p B(\vec a, r)} |D\vec u|^2
\ge \frac{1}{2\pi r} \left(\int_{\p B(\vec a,r)} \frac{\p \vec u}{\p \tau} \times \vec u\right)^2= 2 \pi\frac{ d^2}{r}
\end{equation}
if $d$ is the degree of the map on $\p B(\vec a,r)$. Integrating this relation over $r$  ranging from $\ep$ to $1$ yields  a lower bound for the energy on annuli, with the logarithmic behavior stated above.
One sees that the equality case in \eqref{csgl} is achieved when $\vec u$ is exactly radial
(which corresponds to $\vec u=e^{id\theta}$ in polar coordinates), so the least energetically costly vortices are the radial ones. For an arbitrary number of vortices
 the ``ball construction" \`a la Jerrard and Sandier allows to paste together the lower bounds obtained on disjoint annuli.
Previous constructions for bounded numbers of vortices include those of Bethuel-Brezis-H\'elein \cite{BrBeHe94}
and Han-Shafrir \cite{HaSh95}.
The ball construction method will be further described in Section \ref{ballconstruction}.

For the  cavitation model,  there is an analogue to the above  calculation, which is also our starting point.
Assume that $\vu$ develops  a cavity of volume $v$ around a cavitation point $\vec a $ in the domain space.
 By $v$ we really denote the excess of volume created by the cavity (we still refer to it as cavity volume),
this way  the image of the ball $B(\vec a, \ep)$ contains a volume $\pi\ep^2 +v$.
 Using  the Cauchy-Schwarz inequality,  we may write
\begin{equation}\label{cs1}\int_{\p B(\vec a,r)} |D\vu|^2 \ge \frac{1}{2\pi r} \left( \int_{\p B(\vec a,r)} |D\vu\cdot \tau |\right)^2
.\end{equation}
But then one can observe that $\int_{\p B(\vec a,r)}|D\vu \cdot \tau|$ is exactly  the length of the image curve of the circle  $\partial  B(\vec a,r)$.
 We may then use the classical isoperimetric inequality
 \begin{equation}\label{isoper}
 (\Per E(\vec a,r))^2 \ge 4\pi |E(\vec a,r)|
 \end{equation}
 where $|\cdot |$ denotes the volume, and $E(\vec a, r)$ is the region enclosed by this image curve, which contains the cavity, and has volume $\pi r^2 + v$ by incompressibility.
 Inserting this into \eqref{cs1}, we are led to
 \begin{equation}\label{cs2}
 \int_{\p B(\vec a,r)} \frac{|D\vu|^2}{2} \ge   \frac{\Per^2 (E(\vec a, r)  )}{4\pi r} \ge
\frac{|E(\vec a,r)|}{r}  \ge \frac{v}{r} + \pi r .\end{equation}
 This is the building block that we will integrate over $r$ and insert into the ball construction,
to obtain  our  first lower bound, which is proved in Section \ref{ballconstruction}. To state it, we will use the notion of weak determinant:
\begin{equation*}
\langle \Det D\vec u, \phi \rangle :=
  -\frac{1}{n} \int_\Omega \vec u(\vec x)
\cdot
  (\cof D\vec u(\vec x)) D\phi(\vec x) \dd\vec x, \qquad \forall\, \phi \in C_c^\infty (\Omega)
\end{equation*}
whose essential features we recall in Section \ref{se:DetDu};
as well as M\"uller and Spector's invertibility ``condition INV" \cite{MuSp95}
which is defined  in Section \ref{sec:degree} (Definition \ref{df:INV}) and which essentially means that the deformations of the material, in addition to being one-to-one, cannot create cavities which would at the same time be filled by material coming from elsewhere. Even though we have discussed dimension 2, we directly state the result in dimension $n$.
\begin{proposition}\label{pro1}
 Let $\om$ be an open and bounded set in $\R^n$, and $\om_\ep= \om \backslash \cup_{i=1}^m \overline{B}({\vec a}_i, \ep_i)$
where ${\vec a}_1,  \cdots , {\vec a}_m\in \om$ and the $ \overline{B}({\vec a}_i, \ep_i)$ are disjoint.
  Suppose that $\vec u \in W^{1,n}(\Omega_{\ep}, \R^n)$ 
and that
 condition INV  is satisfied. Suppose,
further, that $\Det D\vec u = \mathcal L^n$ in $\Omega_{\ep}$
(where $\mathcal L^n$ is the Lebesgue measure), and let $v_i:=|E({\vec a}_i, \ep_i)|- \omega_n \ep_i^n$
(with $E(\vec a_i, \ep_i)$ as in \eqref{isoper}). Then
for any $R>0$
$$\frac{1}{n} \int_{\om_\ep} \left (\left |\frac{D\vu}{\sqrt{n-1}}\right |^n -1\right ) \dd\vec x\ge  \left( \sum_{i, B(\vec a_i, R) \subset \om}^m v_i \right) \log \frac{R}{ 2\sum_{i=1}^m \ep_i}
.$$
\end{proposition}
Note that  $\sum_i v_i =V$ is the total cavity volume, which due to incompressibility
is completely determined by the Dirichlet data, in the case of a displacement boundary value problem.

Examining the equality cases in the chain of inequalities  \eqref{cs1}--\eqref{cs2} already tells us
that the minimal energy is obtained when ``during the ball construction" all circles (at least for $r$ small)
are mapped into circles  and  the cavities  are spherical.
A more careful examination of \eqref{cs1} indicates that the map should at least locally
follow the model cavity map \eqref{cavitybehavior}.
It is the same argument that has been used by Sigalovanathan and Spector \cite{SiSp10a,SiSp10b}
to prove the radial symmetry of minimizers for the model with power $p<n$.

When there is more than one cavity,  and two cavities are close together,
we can observe that there is a geometric obstruction  to all circles ``of the ball construction" being mapped into circles.
 This is true for any number of cavities larger than 1;
to quantify it is in principle possible but a bit inextricable for more than 2.
 For that reason and for simplicity, we restrict to the case of two cavities, and now explain the quantitative point.

\begin{figure}[hbt!]
\centering
  \subfigure{\quad \input{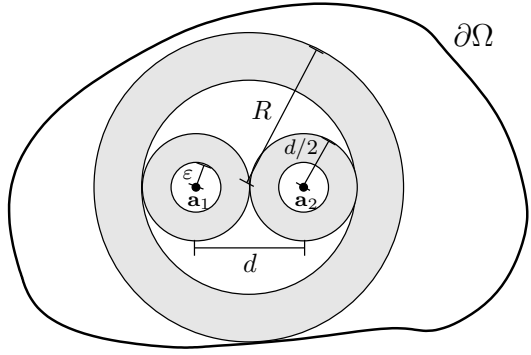}\quad} \qquad
  \caption{\label{fig:bcnst}
Ball construction in the reference configuration
}
\end{figure}

Let  $\vec a_1$ and $\vec a_2$ be the two cavitation points
  with $|\vec a_1-\vec a_2|=d$, small compared to $1$. For simplicity of the presentation let us also assume that $\ep_1=\ep_2=\ep$.
  The ball construction is very simple in such a situation:
three disjoint annuli are constructed, $B({\vec a}_1, d/2)\backslash B({\vec a}_1, \ep) $,
$ B({\vec a}_2, d/2)\backslash B({\vec a}_2,\ep)$ and $ B(\vec a,  R)\backslash B(\vec a, d)$,
 where $\vec a$ is the midpoint of ${\vec a}_1$ and ${\vec a}_2$
 (see Figure \ref{fig:bcnst}). These annuli can be seen as a union of concentric circles centred at ${\vec a}_1$, ${\vec a}_2$,
 $\vec a$ respectively.  To achieve the optimality condition above, each of these circles would have to be mapped  by $\vu$ into a circle.
  If this were true,  the images of $ B({\vec a}_1, d/2)$ and $ B({\vec a}_2, d/2)$ would be two disjoint balls containing the two cavities, call them $E_1$ and $E_2$. By incompressibility, $|E_1|= v_1 + \pi (d/2)^2 $ and $|E_2|= v_2 + \pi (d/2)^2 $. Then the image of $B (\vec a, d)$ would also have to be a ball, call it $E$, which contains the disjoint union $E_1\cup E_2$, and by incompressibility
\begin{equation}\label{vole}
|E|=v_1+ v_2 + \pi d^2.\end{equation}
    If $d$ is small compared to $v_1 $ and $v_2$  it is easy to check this is geometrically impossible:
the radius  of the ball $E_1$ is certainly bigger than $\sqrt{v_1/\pi}$,
that of $E_2$ than  $\sqrt{v_2/\pi}$ and since $E$ is a ball that contains their disjoint union, its radius
is at least the sum of the two, hence $|E|\ge (\sqrt{v_1}+ \sqrt{v_2})^2$. This is incompatible with \eqref{vole}
unless $\pi d^2 \ge 2\sqrt{v_1v_2}$.\\

\begin{figure}[hbt!]
\centering
  \subfigure[$\mu=0.5$] {\quad\ \;\includegraphics[width=0.25\textwidth]{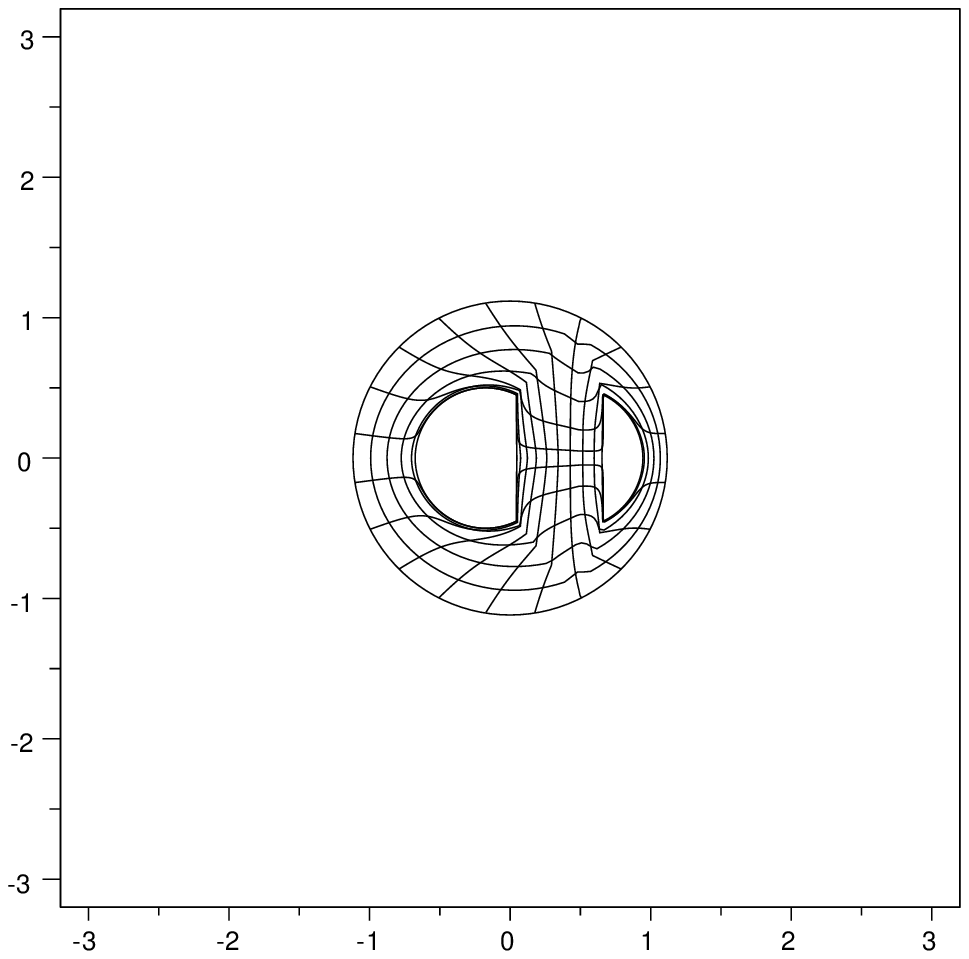}\quad\ \;}
  \subfigure[$\mu=1.5$] {\quad\ \;\includegraphics[width=0.25\textwidth]{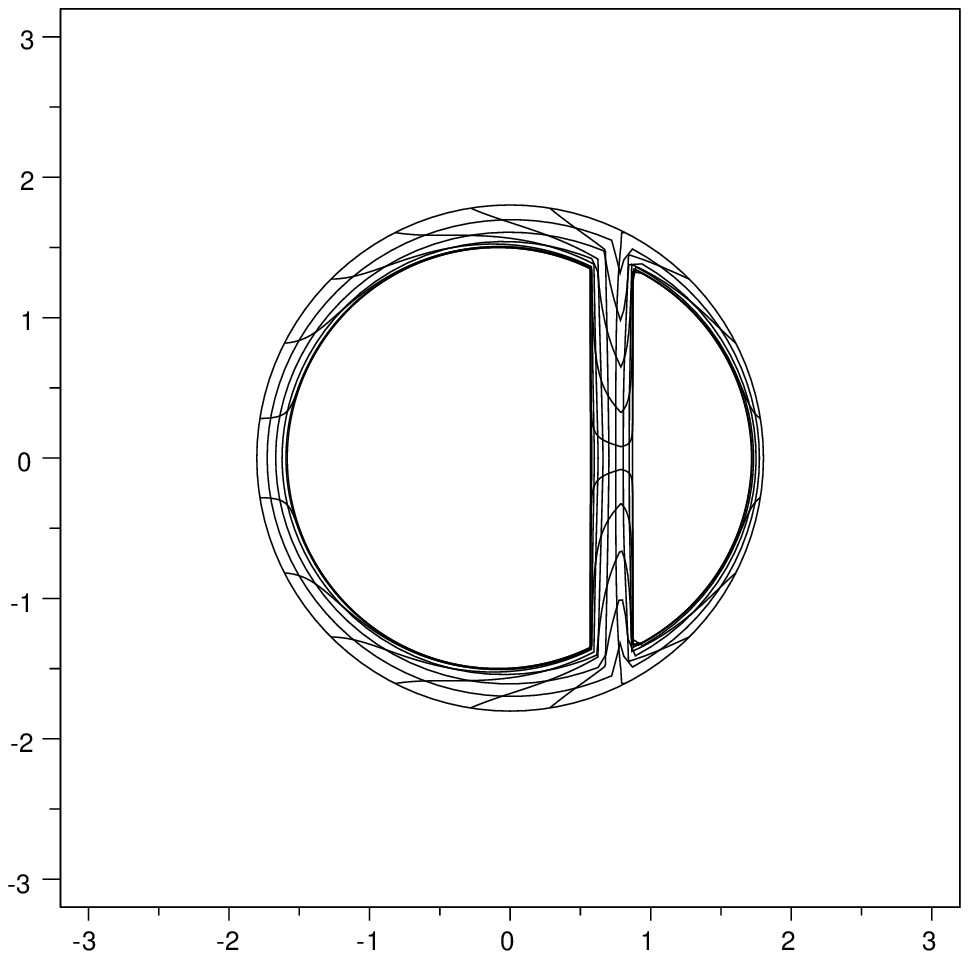}\quad\ \;}
  \subfigure[$\mu=3$] {\quad\ \;\includegraphics[width=0.25\textwidth]{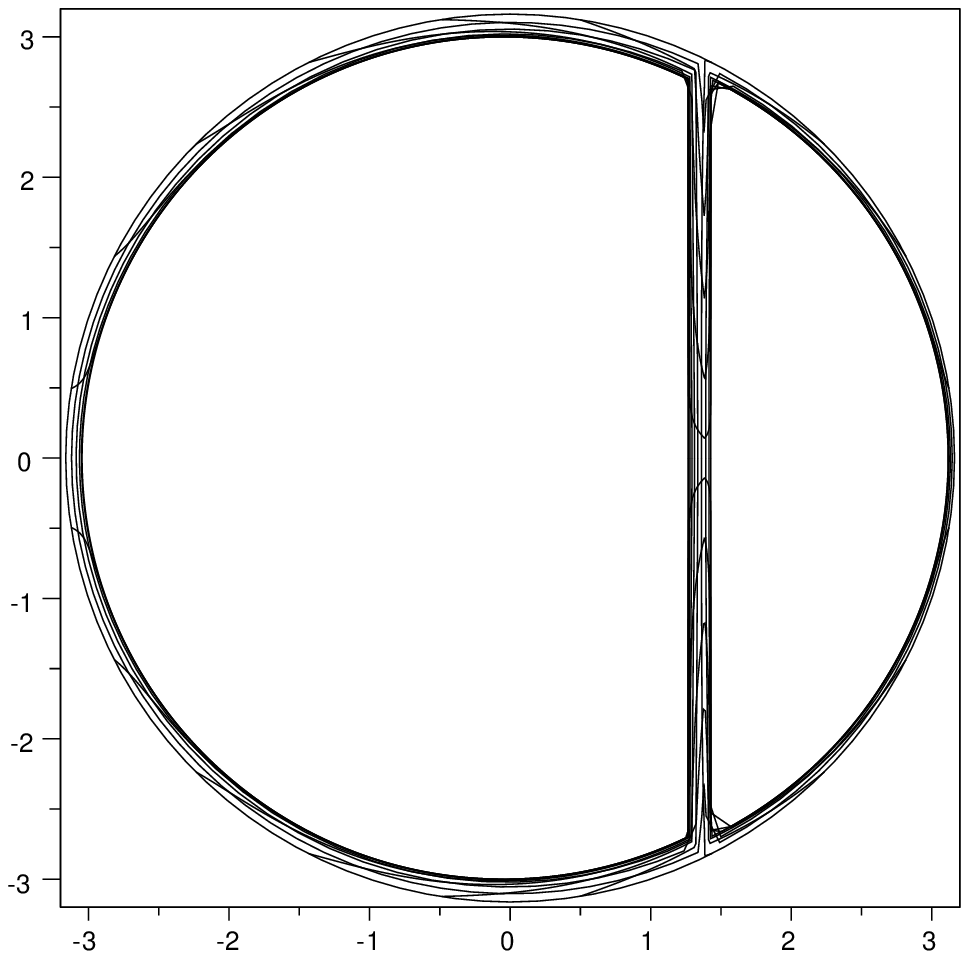}\quad\ \;}
  \caption{\label{fig:distmu}Incompressible deformation $\vec u:
B(\vec 0, d)\setminus \{\vec{a}_1,\vec a_2\}
\to \R^2$, $d=|\vec a_2 - \vec a_1|$,
opening distorted
cavities of
volumes $v_1+\pi \ep_1^2$,
 $v_2+\pi\ep_2^2$; deformed configuration for increasing values of
the displacement load ($\mu:=\sqrt{\frac{v_1+v_2}{\pi d^2}}$). Choice of parameters: $d=1$, $\frac{v_2}{v_1}=0.3$.}
\end{figure}

So in practice, if $d$ is small compared to the volumes,
the circles are not all mapped to exact circles, the inclusion and disjointness are preserved,
but some distortion in the shape of the images has to be created
{\it either} for the ``balls  before merging" i.e. $E_1$ and $E_2$ -- this corresponds to what is
sketched on Figure \ref{fig:distmu} --
{\it or} for the ``balls after merging" i.e.\ $E$ -- this corresponds
to what is sketched in Figure \ref{fig:sphmu}
(the situations of Figures \ref{fig:distmu} and \ref{fig:sphmu} correspond to the test-maps we will use
to get energy upper bounds, see below Section \ref{sec:upperbd}).

\begin{figure}[hbt!] \vspace{0.11cm}
  \begin{minipage}{0.20\textwidth} \includegraphics[width=1\textwidth]{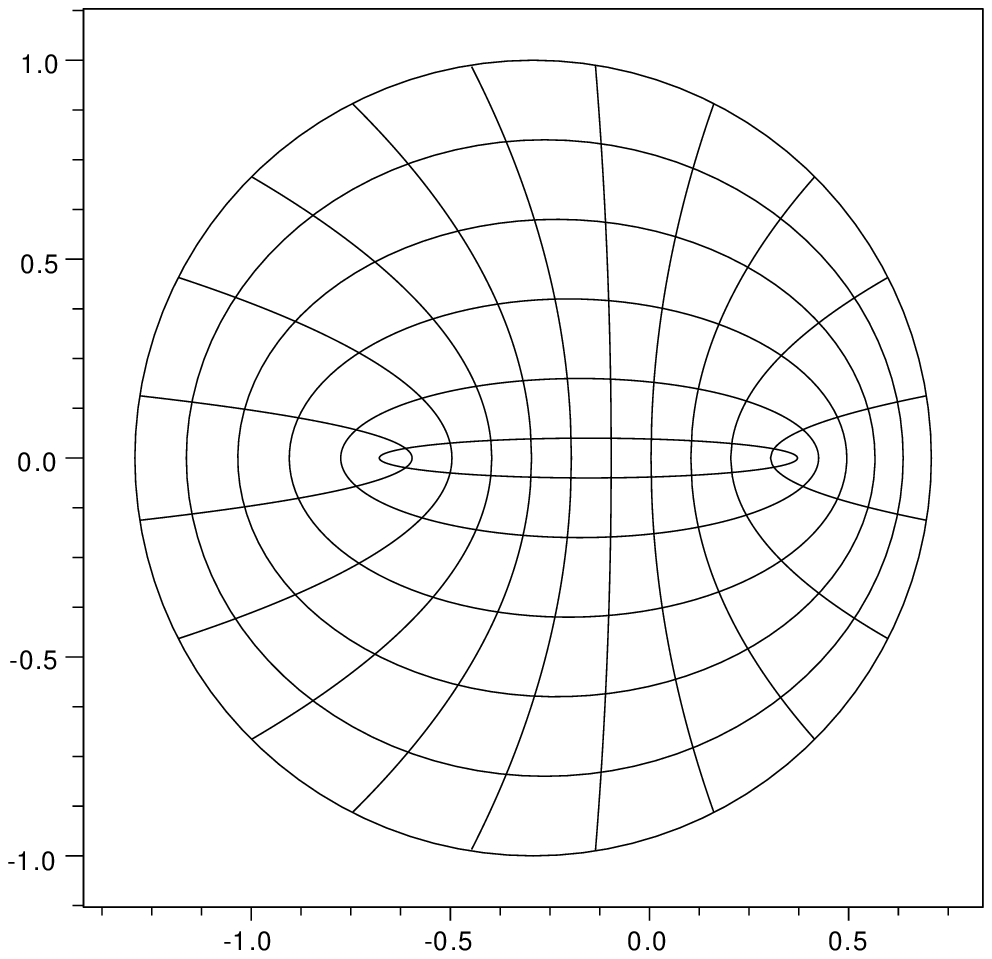} \end{minipage} \qquad\
  \begin{minipage}{0.20\textwidth} \includegraphics[width=1\textwidth]{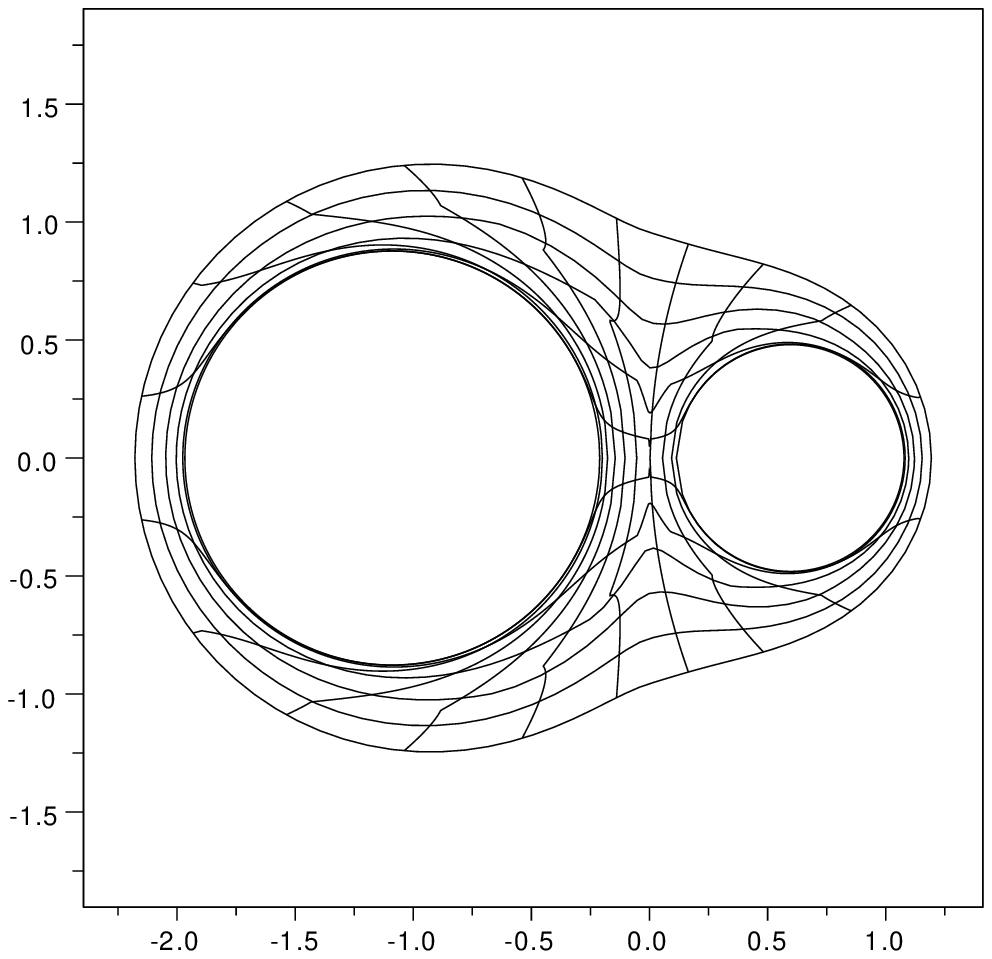} \end{minipage} \qquad\
  \begin{minipage}{0.20\textwidth} \includegraphics[width=1\textwidth]{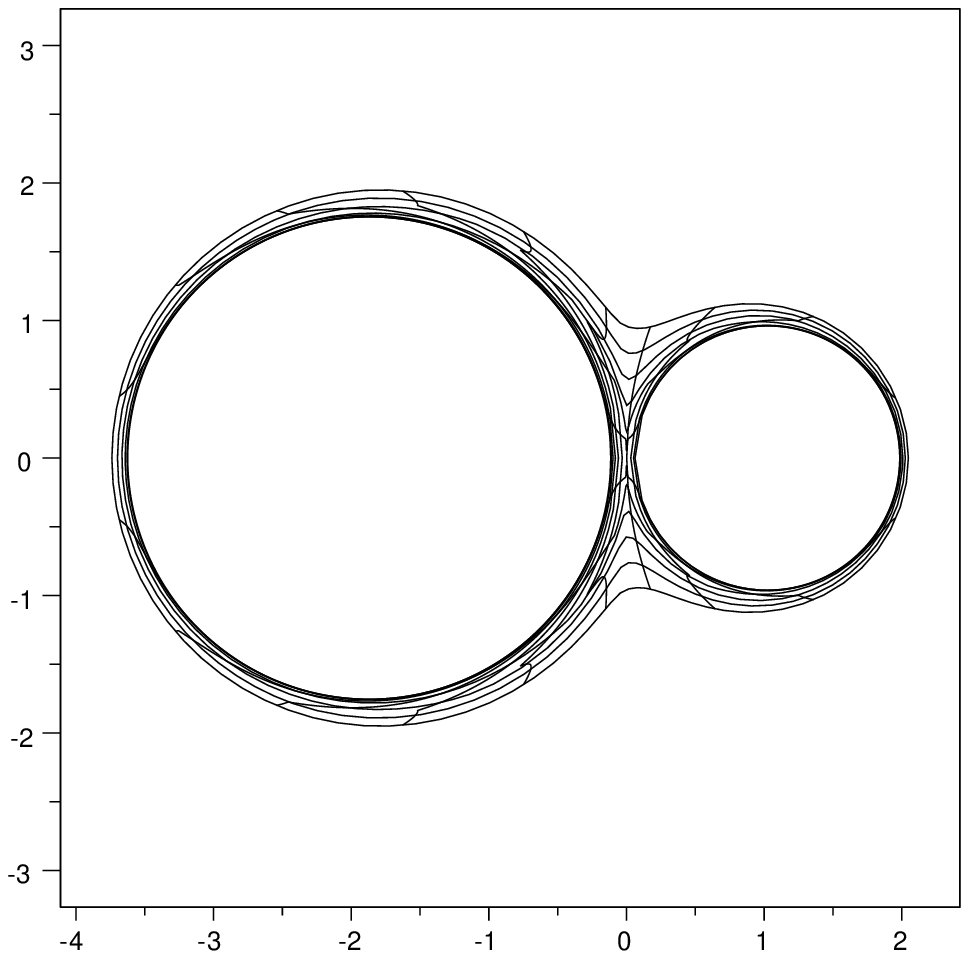} \end{minipage} \qquad\
  \begin{minipage}{0.20\textwidth} \includegraphics[width=1\textwidth]{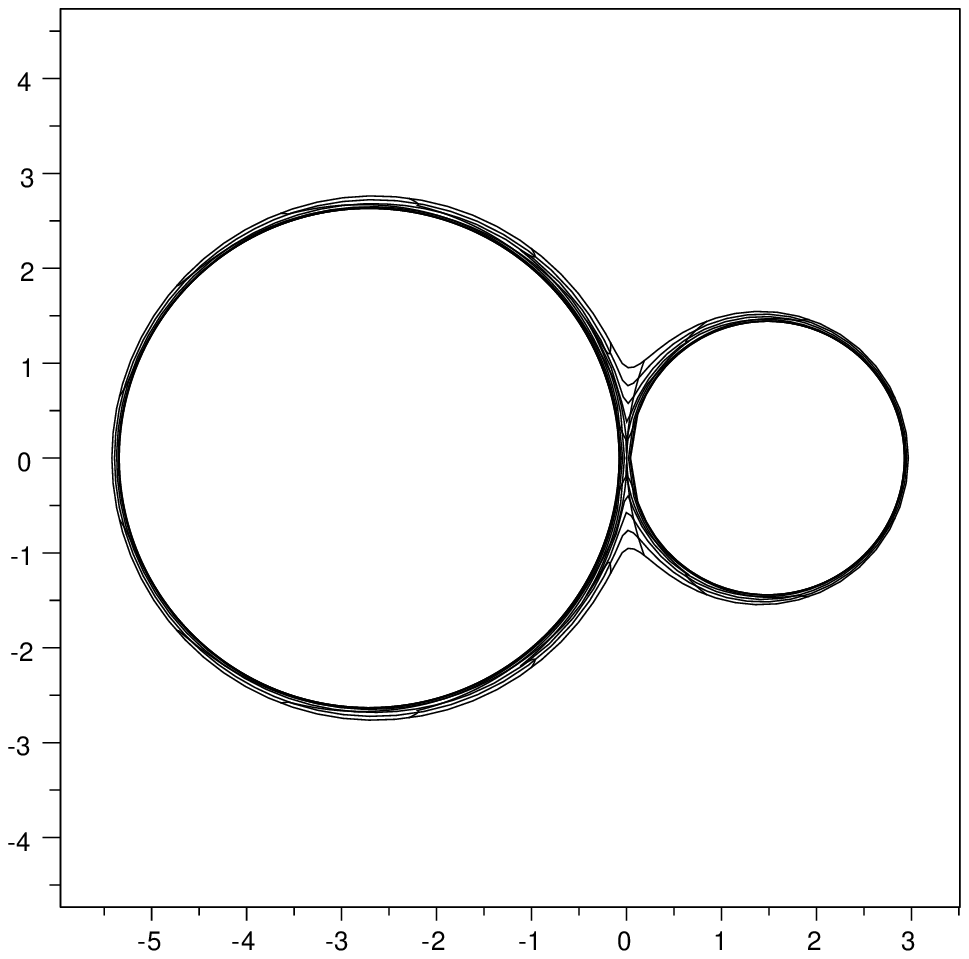} \end{minipage} \vspace{0.1cm} \\

 \begin{minipage}{0.20\textwidth} \footnotesize \centering  (a)\ Reference configuration  \end{minipage} \qquad\
  \begin{minipage}{0.20\textwidth} \footnotesize \centering  (b)\  Deformed configuration, $\mu=1$ \end{minipage} \qquad\
  \begin{minipage}{0.20\textwidth} \footnotesize \centering  (c)\ Deformed configuration, $\mu=2$ \end{minipage} \qquad\
  \begin{minipage}{0.20\textwidth} \footnotesize \centering  (d)\ Deformed configuration, $\mu=3$ \end{minipage}
  \caption{\label{fig:sphmu}Incompressible deformation of $B(\mathbf 0, d)$, $d:=|\vec a_2-\vec a_1|$,
 for increasing values of $\mu:=\sqrt\frac{v_1+v_2}{\pi d^2}$.
Final cavity volumes $v_1$ and $v_2$ given by $d=1$, $\frac{v_2}{v_1}=0.3$.}
\end{figure}

A convenient tool to quantify how much these sets differ from balls, which is what we exactly mean by ``distortion",  is the following
\begin{definition}
 The Fraenkel asymmetry of a measurable set $E\subset \R^n$
is defined as
\begin{align*}
D(E):=\min_{\vec x\in \R^n} \frac{|E\triangle B(\vec x, r_{E})|}{|E|},
\quad \text{with}\ r_E\  \text{such that}\ |B(\vec x, r_{E})|=|E|
\end{align*} where $\triangle$ denotes the symmetric difference between sets.
\end{definition}
Note that $D(E)$ is a scale-free quantity which depends 
not on the size of $E$,
but on its shape.

The following proposition, which we shall prove in Section \ref{se:distortions},
allows to make the observations above quantitative in terms of the distortions.
\begin{proposition} \label{pr:distortions}
Let $E$, $E_1$, and $E_2$ be sets of positive measure in $\R^n$, $n\geq 2$ such that
$E\supset E_1 \cup E_2$ and $E_1\cap E_2 =\varnothing$, and assume without loss of generality that $|E_1|\ge |E_2|$. Then
\begin{multline*} 
 \frac{|E|D(E)^\frac{n}{n-1} + |E_1|D(E_1)^\frac{n}{n-1} + |E_2|D(E_2)^\frac{n}{n-1}}{|E| + |E_1 \cup E_2|} \\
\geq
C_n \left(  \frac{|E_2|}{|E_1|+|E_2|}\right)^{\frac{n}{n-1}}
\left (\frac{(|E_1|^\frac{1}{n}+|E_2|^\frac{1}{n})^n
 -|E|}{(|E_1|^\frac{1}{n}+|E_2|^\frac{1}{n})^n
 - |E_1\cup E_2|}\right )^\frac{n(n+1)}{2(n-1)}\end{multline*}
for some constant $C_n>0$ depending only on $n$.
\end{proposition}

The fact that $E_1 , E_2, E$ cannot simultaneously be balls is made explicit by the fact that  $D(E_1)$, $D(E_2)$, $D(E)$ cannot all vanish unless the right-hand side is negative, which can happen only if $|E|$ is large relative to $|E_1|$ and $|E_2|$.
The first factor in the estimate degenerates only when one of the sets is very small compared to the other.

Note that such a  geometric constraint  is also true for more than two merging balls, so in principle we could treat (with more effort) the case of more than two cavities, however the estimates would degenerate as the number of cavities gets large.

These estimates on the distortions are useful for us thanks to the
following improved isoperimetric inequality, precisely expressed in terms of the Fraenkel asymmetry:
\begin{proposition}[Fusco-Maggi-Pratelli \cite{FuMaPr08}] \label{pr:iso}
For every Borel set $E\subset \R^n$
\begin{equation*}
\Per E \geq n\omega_n^\frac{1}{n}|E|^\frac{n-1}{n}(1+C D(E)),
\end{equation*}
where $C$ is a universal constant. 
\end{proposition}
In dimension 2,  we thus have the improved isoperimetric inequality
\begin{equation}\label{improvedisoper}
(\Per E)^2 \ge 4\pi   |E| + C |E| D^2(E),
\end{equation} for some universal $C>0$.
 Inserting \eqref{improvedisoper} instead of \eqref{isoper} into the basic estimate \eqref{cs2} gives us
 \begin{equation}\label{isoperbis}
  \int_{\p B(\vec a,r)} \frac{|D\vu|^2}{2} \ge
\frac{|E(\vec a,r)|}{r}  + \frac{C}{r}   |E(\vec a, r)| D^2(E(\vec a, r))
\ge \frac{v}{r} + \pi r + \frac{C}{r} |E(\vec a, r)| D^2(E(\vec a, r)).
  \end{equation}
 This   then allows us to get improved estimates when integrating over $r$ (in a ball construction procedure),
 keeping track of the fact that to achieve equality, all level curves $E(\vec a,r)$
   which are images of circles during the ball construction would have  to be circles.
This way, after subtracting off  the leading order term $\sum_i v_i \log \frac{1}{\sum_{i}\ep_i}$ we can retrieve a next order ``renormalized" term that will account for the cavity interaction.
 This is expressed in the following main result.

\begin{theorem}[Lower bound] \label{th:LB} \label{th:LBn}
Given $\Omega\subset \R^n$ a bounded open set, let
$\Omega_\ep:=\Omega \setminus (\overline B_{\ep_1}(\vec a_1) \cup \overline B_{\ep_2} (\vec a_2))$,
where $\vec a_1, \vec a_2 \in \Omega$, $\ep_1, \ep_2>0$,
and assume that $B_{\ep_1}(\vec a_1)$ and $B_{\ep_2}(\vec a_2)$ are disjoint and contained in $\Omega$.
Suppose that $\vec u \in W^{1,n}(\Omega_{\ep}, \R^n)$ 
satisfies condition INV and $\Det D\vec u = \mathcal L^n$ in $\Omega_{\ep}$.
Set
\begin{align*}
  \vec a:= \frac{\vec a_1+\vec a_2}{2},
\qquad
  d:=|\vec a_1-\vec a_2|,
\qquad
  v_1:= |E(\vec a_1, \ep_1)| - \omega_n \ep_1^n,
\qquad
  v_2:= |E(\vec a_2, \ep_2)| - \omega_n \ep_2^n.
\end{align*}
Then, for all $R$ such that $B(\vec a, R)\subset \Omega$,
\begin{multline*} 
\frac{1}{n} \int_{\Omega_{\ep}\cap B(\vec a, R)}
   \left ( \left |\frac{D\vec u(\vec x)}{\sqrt{n-1}}\right |^n - 1 \right )\dd \vec x
\geq
  v_1 \log \frac{R}{2\ep_1} + v_2 \log \frac{R}{2\ep_2}
\\   \quad \nonumber
  + C (v_1+v_2) \left ( \left ( \frac{\min\{v_1,v_2\}}{v_1+v_2} \right )^\frac{n}{n-1}
      - \frac{\omega_n d^n}{v_1+v_2} \right )_+
  \log \min \left \{ \left ( \frac{v_1+v_2}{2^n \omega_n d^n} \right )^\frac{1}{n^2} ,
	\frac{R}{d},
	\frac{d}{\max\{\ep_1, \ep_2\}} \right \}
\end{multline*}
for some constant $C$ independent of $\Omega$, $\vec a_1$, $\vec a_2$, $d$, $v_1$, $v_2$, $\ep_1$, and $\ep_2$
($t_+$ stands for $\max\{0, t\}$).
\end{theorem}

Two main differences appear in this lower bound compared to Proposition \ref{pro1}.
First, the leading order term $(v_1+v_2)\log \frac{1}{\ep_1+\ep_2}$
has been improved to $v_1\log \frac{1}{\ep_1} + v_2 \log \frac{1}{\ep_2}$,
which shows that the energy goes to infinity as $\ep_1\to 0$ or $\ep_2\to 0$,
even if $\ep_1 + \ep_2 \not \to 0$. This term is optimal since it coincides
with the leading order term in the upper bound of Theorem \ref{th:UB} below, and in fact
it should be possible to replace $\sum_i v_i \log \frac{1}{\sum_i \ep_i}$ with
$\sum_i (v_i \log \frac{1}{\ep_i})$ in Proposition \ref{pro1} (however, this would require
a more sophisticated ball construction, and it is not immediately clear how to obtain a general result
for the case of more than two cavities).
Second,
and returning to the discusion in dimension two and choosing $\ep_1=\ep_2=\ep$,
compared to Proposition \ref{pro1} we have gained the new term
\begin{equation*}
C(v_1+v_2) \left (  \left (\frac{\min\{v_1,v_2\}}{v_1+v_2}\right)^2 - \frac{\pi d^2}{v_1+v_2} \right )
  \min\left \{\log \sqrt[4]{\frac{v_1+v_2}{4\pi d^2}}, \log \frac{R}{d}, \log \frac{d}{\ep} \right\},
\end{equation*}
This term is  of course worthless  unless $\frac{\pi d^2}{v_1+v_2} < \left (\frac{\min\{v_1,v_2\}}{v_1+v_2}\right)^2$ i.e.
$ \pi d^2 \le \frac{\min\{v_1^2,v_2^2\}}{ v_1+v_2}$.
Under that condition, it expresses an interaction between  the two cavities in terms of the distance of the cavitation points
relative to the data of $v_1, v_2 $ and $\ep$.
  As $\frac{\pi d^2}{v_1+v_2}\to 0$ the interaction  tends logarithmically to $+\infty$;
 this expresses a {\it logarithmic repulsion} between the cavities, unless the term
$\log \frac{d}{\ep}$ is the one that achieves the min above, which can only happen if $\log d$ is comparable to $\log \ep$. This expresses an {\it attraction of the cavities when they are close compared to the puncture scale}, which we believe means that two cavities thus close would energetically prefer to be merged into one.
%
This  suggests that three scenarii are energetically possible:
  \begin{description}
  \item[Scenario (i)] \label{it:scn1}
  the cavities are spherical and the cavitation points are well separated (but not necessarily the cavities themselves),
this is  the situation of
 Figure \ref{fig:sphmu}
  \item[Scenario (ii)]
 the cavitation points
are at distance $\ll 1$  but all but  one cavity are of very small volume and  hence  ``close up" in the limit $\ep \to 0$
 \item[Scenario (iii)]
 ``outer circles" (in the ball construction) are mapped into circles and cavities  (as well as cavitation points)
are  pushed together to form one equivalent round cavity, this is the situation of Figure \ref{fig:distmu}.
This seems to correspond to void coalescence (c.f.\ \cite{XuHe11,LiLi11}).
\end{description}

\subsection{Method and main results: upper bound }\label{sec:upperbd}
After obtaining this lower bound,
we show that it is close to being optimal
(at least in scale).
To do so we need to construct explicit test maps and evaluate their energy
(in terms of the parameters of the problem). The main difficulty is that these test maps have to
satisfy the incompressibility condition outside of the cavitation points, and as we mentioned previously,
there is no simple parametrization of such incompressible maps.
The main known result in that area is the celebrated result of Dacorogna and Moser \cite{DaMo90}
which provides an existence result for incompressible maps with compatible boundary conditions.
Two methods are proposed in their work, one of them constructive,
however they are not explicit enough to
evaluate the Dirichlet energy of the map.


The question we address can be phrased in the following way: given a domain with a certain  number of ``round holes" at certain distances from each other,  and another domain of same volume, with the same number of holes whose volumes are prescribed but whose positions and shapes are free; can we find an incompressible map that maps one to the other, and can we estimate its energy $\int |D\vu|^n$ in terms of the distance of the holes and the cavity volumes?

We  answer positively this question, still in the case of two holes, by using two tools:
\begin{itemize}
\item[(a)]
 a family of  explicitly defined incompressible deformations preserving angles, that we introduce
 \item[(b)] the construction of incompressible maps of Rivi\`ere and Ye \cite{RiYe94,RiYe96},
which is more tractable than Dacorogna and Moser to obtain energy estimates.\end{itemize}

We believe it would be of interest to tackle that question in a more general setting: compute the minimal Dirichlet
energy of an incompressible map between two domains with same volume, and  the same number of holes,
the holes having arbitrary shapes and sizes; and find appropriate geometric parameters to evaluate it as a function
of the domains. This question is beyond the scope of our paper however and we do not attempt to treat it in that
much generality.\\

\begin{figure}[hbtp!]
 \centering
  \subfigure[$\delta=0.1$]{\includegraphics[height=0.24\textwidth]{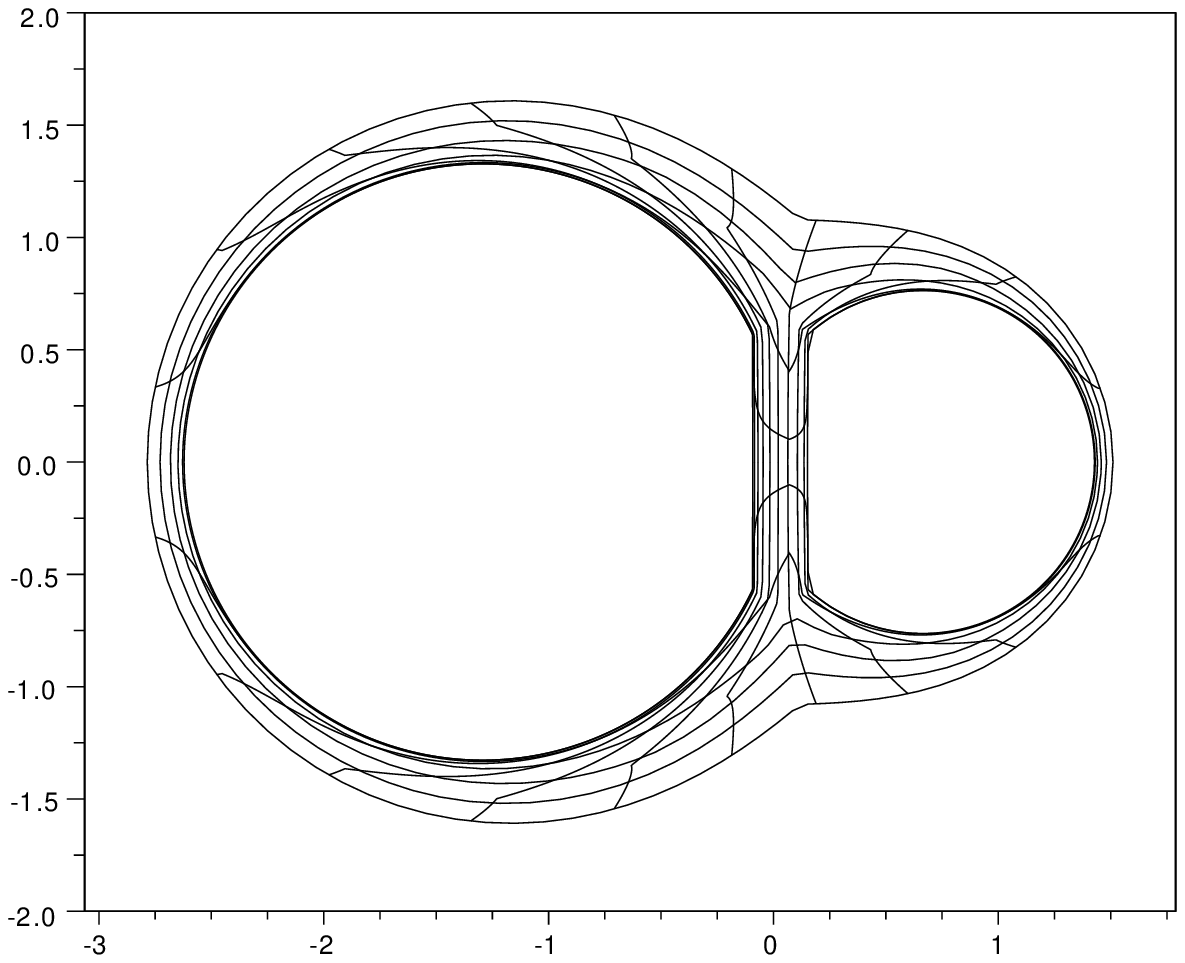}}
\qquad\quad
  \subfigure[$\delta=0.4$]{\includegraphics[height=0.24\textwidth]{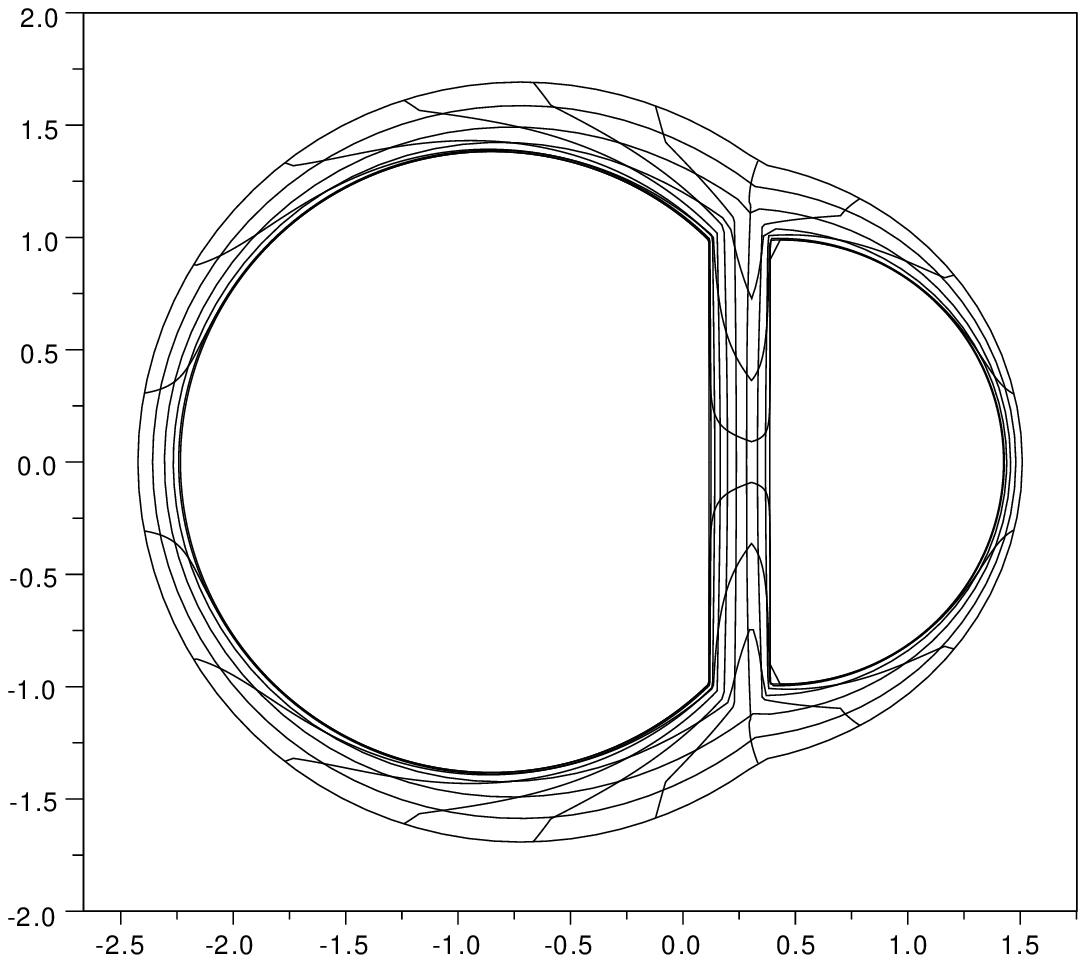}}
\qquad\quad
  \subfigure[$\delta=0.9$]{\includegraphics[height=0.24\textwidth]{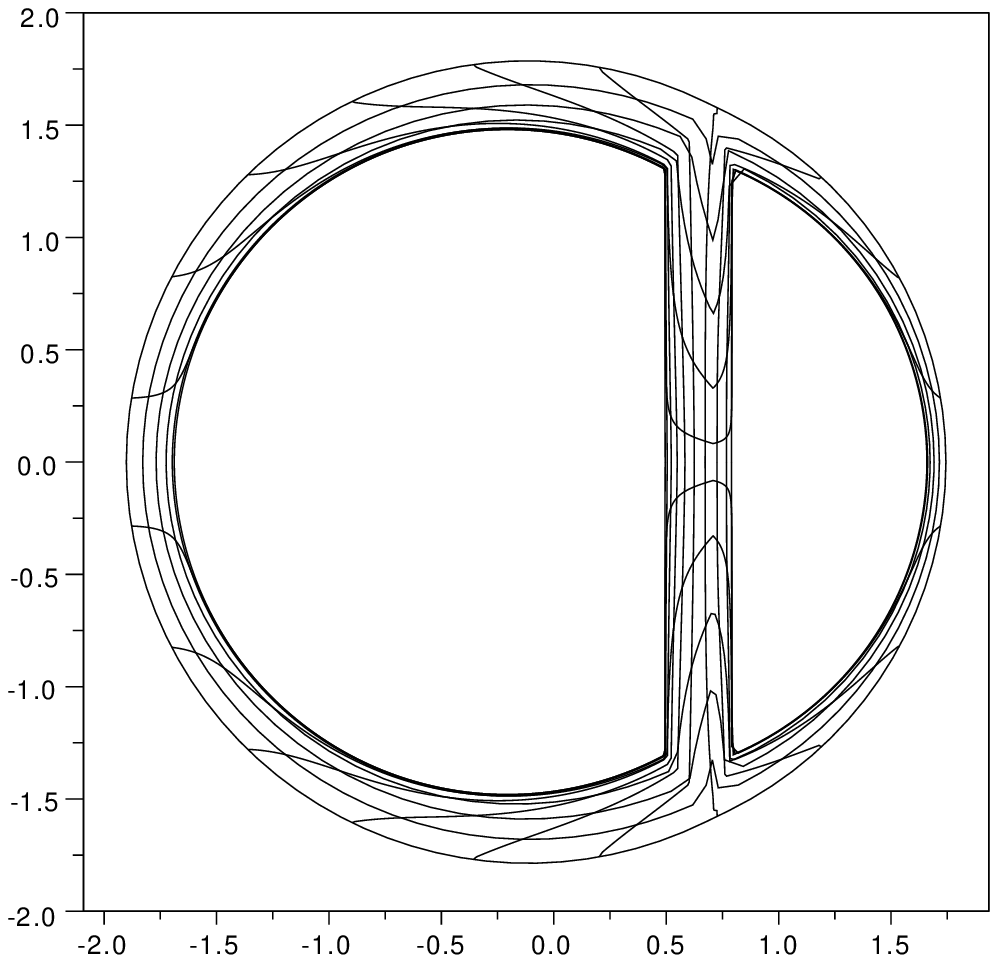}}
  \caption{\label{fig:intdelta}Transition from round to distorted cavities:
$d=1$, $\sqrt{\frac{v_1+v_2}{\pi d^2}}=1.5$, $\frac{v_2}{v_1}=0.3$.}
\end{figure}

Our main result (proved in Section \ref{se:ub1}) is the following.
\begin{theorem} \label{th:ub1} \label{th:UB}
 Let $\vec a_1, \vec a_2 \in \R^n$, $v_1\geq v_2\geq 0$, and suppose that $d:=|\vec a_1-\vec a_2|>\ep_1+\ep_2$.
Then, for every $\delta \in [0,1]$ there exists $\vec{a^*}$ in the line segment joining $\vec a_1$ and $\vec a_2$,
and a piecewise smooth map $\vec u \in C(\R^n \setminus
\{\vec a_1,\vec a_2\}, \R^n)$ satisfying condition INV,
such that
$\Det D\vec u = \mathcal L^n + v_1\delta_{\vec a_1} + v_2 \delta_{\vec a_2}$ in $\R^n$
 and
for all $R>0$
\begin{multline*}
 \underset{B(\vec a^*, R) \setminus (B_{\varepsilon_1}(\vec a_1)\cup  B_{\varepsilon_2}(\vec a_2))}{\int}
 \frac{1}{n} \left | \frac{D\vec u}{\sqrt{n-1}} \right |^n \dd\vec x
 \leq
C_1(v_1+v_2+ \omega_n R^n)
+v_1\left (\log \frac{R}{\ep_1}\right )_+ + v_2\left (\log \frac{R}{\ep_2}\right)_+
\\
+C_2(v_1+v_2)\left ( (1-\delta) \left (\log \frac{R}{d} \right)_+
+ \delta
\left(\sqrt[n]{\frac{v_2}{v_1}}  \log \frac{d}{\ep_1}
+ \sqrt[2n]{\frac{v_2}{v_1}} \log \frac{d}{\ep_2}
\right ) \right)
\end{multline*}
($C_1$ and $C_2$ are universal constants depending only on $n$).
\end{theorem}

If we are not preoccupied with boundary conditions but just wish to build a test configuration with cavities
of prescribed volumes and cavitation points at distance $d$,
then the above result suffices.
This is obtained by our  construction of an explicit family of incompressible maps,
which contains parameters allowing for all possible cavitation points distances $d$ and cavity volumes $v_1, v_2$.
The feature of this construction is that it allows for our almost optimal estimates,
as the shapes of the cavities are automatically adjusted to the optimal scenario according
to the ratio between $d, \ep, \sqrt{v_1}, \sqrt{v_2}$, their logs, etc,
as in the three scenarii  of the end of the previous subsection.
In other words, the construction builds cavities which, when $d$ is comparable to $\ep$, are distorted
and form one equivalent round  cavity while the deformation rapidly becomes radially symmetric
(as in Scenario (iii));
and cavities which are more and more round as $d$ gets large compared to $\ep$
(as in Scenario (i)).
For the extreme cases $\delta=1$ and $\delta=0$,
the maps are those that were presented in Figures \ref{fig:distmu} and \ref{fig:sphmu} respectively.
The result for intermediate values of $\delta$ is shown in Figure \ref{fig:intdelta}.\\

\begin{figure}[htbp!]
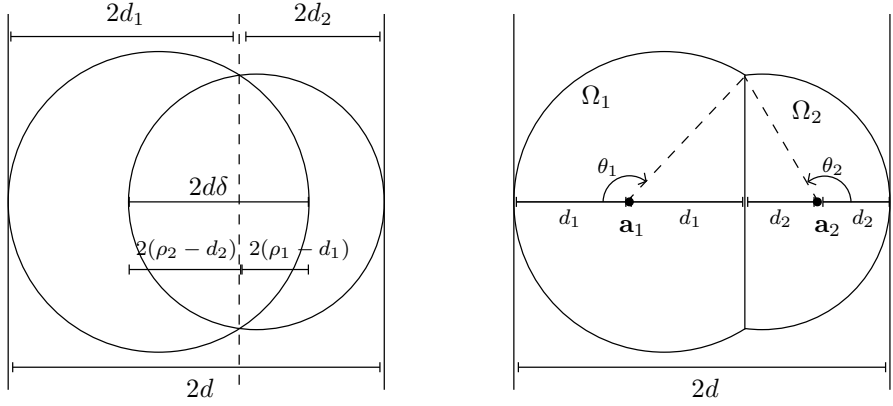

 \centering
  \subfigure{\input{upbnd.pstex_t}} \qquad \qquad
  \subfigure{\input{upbnd2.pstex_t}}
\caption{\label{fig:upbnd}Geometric construction of domains $\Omega_1$, $\Omega_2$ satisfying
$\frac{|\Omega_1|}{|\Omega_2|}= \frac{v_1}{v_2}$.}
\end{figure}

The idea of the construction is the following.
Take two intersecting balls $B(\vec{\tilde a}_1, \rho_1)$
and $B(\vec{\tilde a}_2, \rho_2)$ such that the width of their union is exactly $2d$ and the
width of their intersection is $2d\delta$,
and let $\Omega_1$ and $\Omega_2$ be as in Figure \ref{fig:upbnd} (the precise definition is given in \eqref{eq:Om1Om2}).
As will be proved in Section \ref{se:ubI},
for every $\delta \in [0,1]$ there are unique $\rho_1$ and $\rho_2$
such that $\frac{|\Omega_1|}{|\Omega_2|} = \frac{v_1}{v_2}$.
 The cavitation points $\vec a_1$ and $\vec a_2$ are suitably
placed in $\Omega_1$ and $\Omega_2$, respectively, in such a way that $|\vec a_1-\vec a_2|=d$.
It is always possible to choose $\vec {a^*}$ between $\vec a_1$ and $\vec a_2$
such that $\overline{\Omega_1\cup \Omega_2}$ is \emph{star-shaped} with respect to $\vec a^*$.
In order to define $\vec u$ in $\R^n \setminus \overline{\Omega_1\cup \Omega_2}$
we choose $\vec a^*$ as the origin and look for an \label{p:anglepreserving}\emph{angle-preserving} map
\begin{align*}
 \vec u(\vec x) = \lambda \vec a^* + f(\vec x) \frac{\vec x-\vec a^*}{|\vec x -\vec a^*|},
\qquad \lambda^n-1:= \frac{v_1+v_2}{|\Omega_1\cup \Omega_2|} = \frac{v_1}{|\Omega_1|} = \frac{v_2}{|\Omega_2|}.
\end{align*}
By so doing, we can solve the incompressibility equation $\det D\vec u=1$ explicitly, since
for angle-preserving maps the equation has the same form as in the radial case,
\begin{align*}
 \det D\vec u(\vec x) = \frac{f^{n-1}(\vec x) \frac{\partial f}{\partial r}(\vec x)}{r^{n-1}}\equiv 1,
\quad r=|\vec x-\vec {a^*}|,
\end{align*}
which we will see can be solved as
\begin{align*}
f^n (\vec x)= |\vec x -\vec a^*|^n + A\left (\frac{\vec x -\vec {a^*}}{|\vec x -\vec {a^*}|}\right) ^ n,
\end{align*}
where the function $A: \S^{n-1}\to \R$ is completely determined
if we prescribe $\vec u$ on $\partial \Omega_1 \cup \partial \Omega_2$.
Inside $\Omega_1$ and $\Omega_2$ the deformation $\vec u$ is defined analogously, taking $\vec a_1$ and $\vec a_2$ as the
corresponding origins. The resulting map creates cavities at $\vec a_1$ and $\vec a_2$ with the desired volumes, and with
exactly the same shape as $\partial \Omega_1$ and $\partial \Omega_2$.
For compatibility we impose $\vec u(\vec x) = \lambda \vec x$ on $\partial \Omega_1 \cup \partial \Omega_2$.

In the energy estimate, $(1-\delta)\log \frac{R}{d}$ is the excess energy due to the distortion of the `outer'
curves $\vec u(\partial B(\vec {a^*}, r))$, $r\in (d, R)$, and $\delta \left ( \sqrt[n]{\frac{v_2}{v_1}} \log \frac{d}{\ep_1}
+ \sqrt[2n]{\frac{v_2}{v_1}}\log \frac{d}{\ep_2}\right )$ is that due to the distortion of the curves
$\vec u(\partial B(\vec{a_i}, r))$, $r\in (\ep_i, d)$, $i=1,2$ near the cavities.
When $\delta=0$, $\overline{\Omega_1}$ and $\overline{\Omega_2}$
are tangent balls, the cavities are spherical, and the second term in the estimate vanishes.
The outer curves are distorted because their shape depends on that of $\partial (\Omega_1 \cup \Omega_2)$,
hence a price of the order of $(v_1+v_2)\log \frac{R}{d}$ is felt in the energy.
When $\delta=1$, at the opposite end,
$\Omega_1 \cup \Omega_2$ is a ball of radius $d$, the deformation is radially symmetric outside $\Omega_1 \cup \Omega_2$,
and no extra price for the outer curves is paid. In contrast, the cavities are ``D-shaped''
(they are copies of $\partial \Omega_1$ and $\partial \Omega_2$), and a price of order
$(v_1+v_2) \sqrt[2n]{\frac{v_2}{v_1}} \log \frac{d}{\ep}$ is obtained as a consequence
(in this case the excess energy vanishes as $\frac{v_2}{v_1}\to 0$, in agreement with the prediction of Theorem \ref{th:LB}).

Since the last term of the energy estimate is linear in $\delta$, by
taking\footnote{When considering boundary conditions, not all values of $\delta$ can be chosen,
see the discussion below.}
either $\delta=0$ or
$\delta=1$ (and assuming $R>d$) the estimate becomes
\begin{align*}
 C(v_1+v_2)\min\left \{ \log \frac{R}{d}, \sqrt[n]{\frac{v_2}{v_1}}  \log \frac{d}{\ep_1}
+ \sqrt[2n]{\frac{v_2}{v_1}} \log \frac{d}{\ep_2} \right \}.
\end{align*}
Comparing it against the corresponding term for the lower bound,
namely\footnote{we assume, e.g., that $v_1+v_2 < 4\pi R^2$, in order to illustrate the main point},
\begin{align*}
 C(v_1+v_2) \min \left \{ \left (\frac{v_2}{v_1}\right)^{\frac{n}{n-1}} \log \frac{R}{d},
\left (\frac{v_2}{v_1}\right)^{\frac{n}{n-1}} \log\frac{d}{\max\{\ep_1, \ep_2\}} \right \},
\end{align*}
we observe that there are still some qualitative differences. First of all, in the case when $\ep_1 \ll \ep_2$,
a term of the form $\log \frac{d}{\ep_1} + \log \frac{d}{\ep_2}$ is much larger than $\log \frac{d}{\max\{\ep_1, \ep_2\}}$.
We believe that the expression in the lower bound quantifies more accurately the effect of the distortion of the cavities,
 and that the obstacle for obtaining a comparable expression in the upper bound is
that the domains $\Omega_1$ and $\Omega_2$ in our explicit constructions are required to be star-shaped.
For example, in the case $d\sim \ep_2$, an energy minimizing deformation $\vec u$ would
try to create a spherical cavity at $\vec a_1$ (so as to prevent a term of order $\log \frac{d}{\ep_1}$ from
appearing in the energy due to the distortion of the first cavity),
and, at the same time, to rapidly
become radially symmetric (because of the price of order $\log \frac{R}{d}$ due to the distortion of the `outer' circles).
Therefore, for values of $\pi \ep_2^2\ll v_1+v_2$, the second cavity would be of the form
$B \setminus B_1$ for some balls $B_1$ and $B$ such that $B_1\subset B$, $|B_1|=v_1$, and $|B|=v_1+v_2$.
In other words, $\vec u$ must create  ``moon-shaped'' cavities, which cannot be obtained if $\vec u$ is angle-preserving.

In the second place, the interaction term in the lower bound vanishes as $\frac{v_2}{v_1}\to 0$ regardless of whether
the minimum is achieved at $\log \frac{R}{d}$ or at $\log\frac{d}{\ep}$, whereas in the upper bound this
vanishing effect is obtained only for the case of distorted cavities (when $\log \frac{d}{\ep}$ is the smallest).
This is because when $\delta=0$ and $v_1\gg v_2$, the circular
sector\footnote{we state this in two dimensions for simplicity} $\{\vec a^* + de^{i\theta}, \theta \in (\frac{\pi}{2},
\frac{3\pi}{2})\}$ is mapped to a curve $\lambda \vec a^* + f(\varphi) e^{i\varphi}$ with polar angles $\varphi$ ranging almost from $0$ to
$2\pi$.
This ``angular distortion'' necessarily produces a strict inequality in \eqref{cs1}, so in principle it could
be possible to quantify its effect in the lower bound. It is not clear, however, whether
for a minimizer an interaction term of the form $(v_1+v_2)\log \frac{R}{d}$ will always be present (in the case
when $\frac{v_2}{v_1}\to 0$), or if the fact that such a term appears
in the upper bound is a limitation of the method used for the explicit constructions.

Finally, the factor $\frac{v_2}{v_1}$ in front of $\log \frac{d}{\ep_1}$ and $\log \frac{d}{\ep_2}$
is raised to a different exponent in each term, the reason being that $\Omega_1$ and $\Omega_2$
play different roles in the upper bound construction. Provided $\delta >0$, when $\frac{v_2}{v_1}\to 0$
the first subdomain is becoming more and more like a circle (its height and its width
tend to be equal, and the distortion of the first cavity tends to vanish)
whereas $\Omega_2$ becomes increasingly distorted (the ratio between its height and its width tends to infinity).
The factor $\sqrt[2n]{\frac{v_2}{v_1}}$ in front of $\log \frac{d}{\ep_2}$ is only due to the fact
that the effect in the energy of the distortion of the cavities also depends on the size of the cavity.

\subsubsection*{Dirichlet boundary conditions}

If we want our maps to satisfy specific Dirichlet boundary conditions, then they need to be
``completed" outside of the ball $B(\vec {a^*}, R)$ of the previous theorem.
For that we use the method of Rivi\`ere and Ye, and show how to obtain explicit Dirichlet energy estimates from it.
We consider the radially symmetric loading of a ball,
but other boundary conditions could also be handled.
Let $\vec{a^*}$, $\delta$, $\rho_1$, $\rho_2$, $\Omega_1$, $\Omega_2$ be as before.
We are to find $R_1$, $R_2$,
and an incompressible diffeomorphism $\vec u:\{R_1<|\vec x -\vec {a^*}| < R_2\} \to \R^n$
such that
\begin{enumerate}[i)]
 \item $\Omega_1\cup \Omega_2 \subset B(\vec{a^*}, R_1)$ and
$\vec u|_{\partial B(\vec {a^*}, R_1)}$ coincides with the map of Theorem \ref{th:ub1}
 \item $\vec u|_{\partial B(\vec a^*, R_2)}$ is radially symmetric.
\end{enumerate}

\begin{figure}[hbt!]
\centering
\begin{minipage}{0.31\textwidth}
\begin{center}
\includegraphics[height=0.8\textwidth]{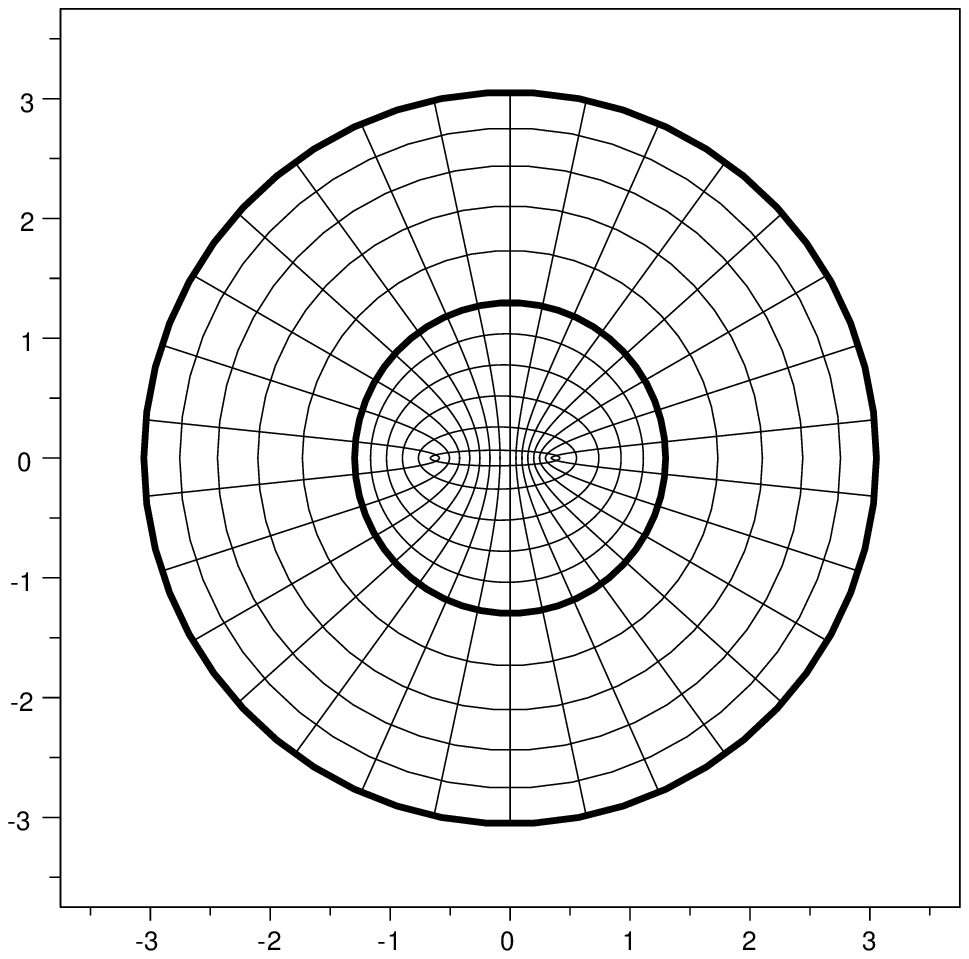}\vspace{0.2cm}\\
{\footnotesize (a) $\delta=0.1$, reference configuration,\\
 $\pi(R_2^2-R_1^2)=3.06(v_1+v_2)(1-\delta)$}\vspace{0.6cm} \\
\includegraphics[height=0.8\textwidth]{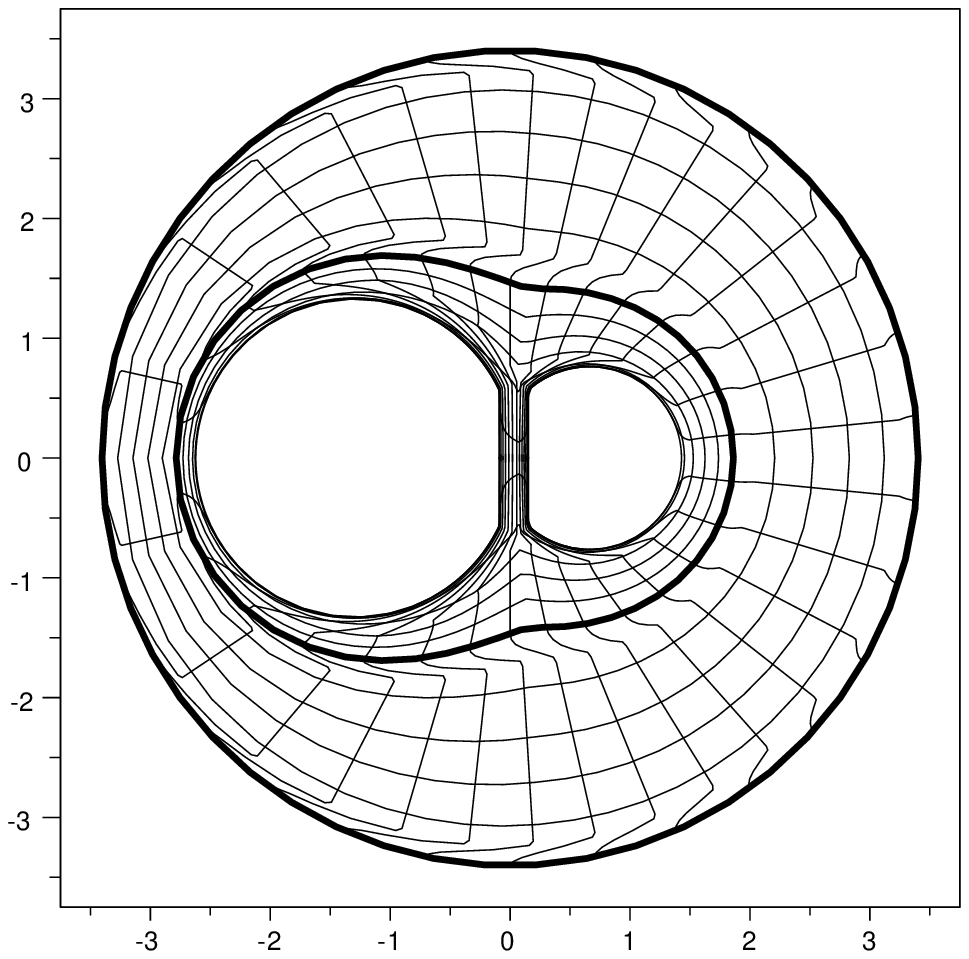}\vspace{0.2cm}\\
{\footnotesize (d) $\delta=0.1$, deformed configuration.\\ Thick line at $\vec u(\partial B(\vec a^*, R_1))$.}
\end{center}
\end{minipage}
\quad
\begin{minipage}{0.31\textwidth}
\begin{center}
\includegraphics[height=0.8\textwidth]{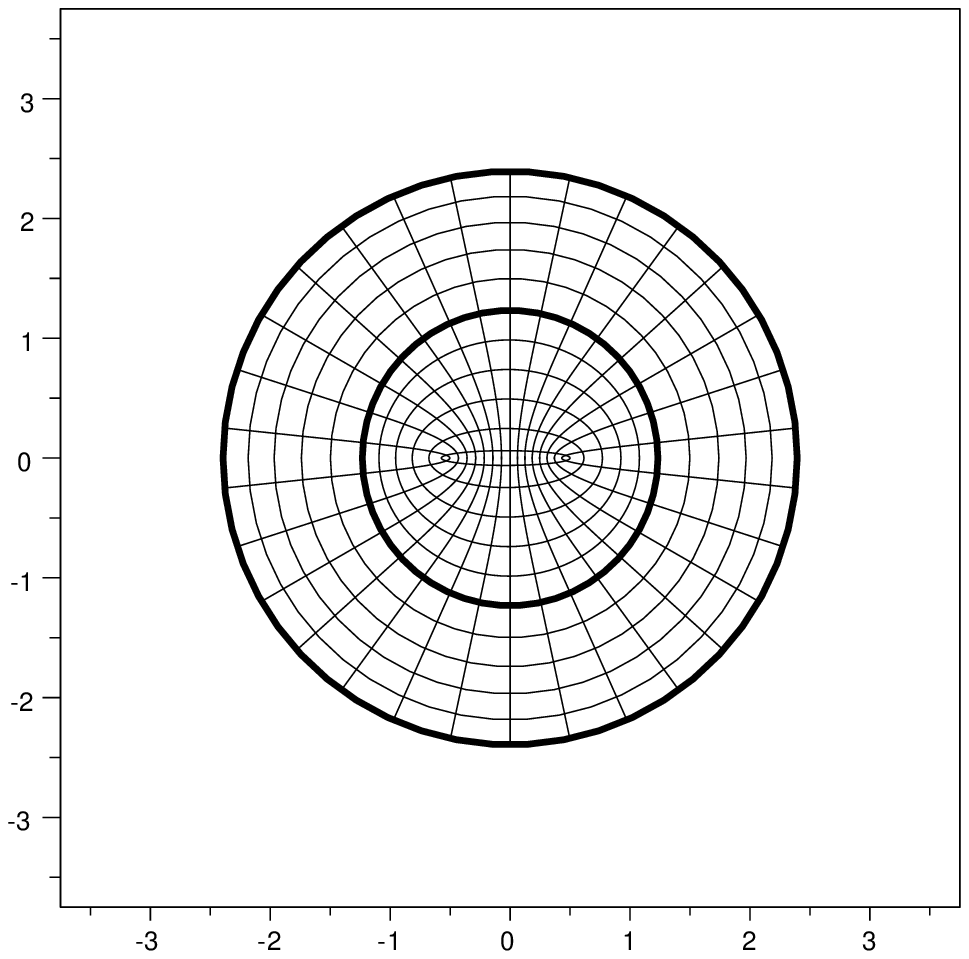}\vspace{0.2cm}\\
{\footnotesize (b) $\delta=0.4$, reference configuration,\\
 $\pi(R_2^2-R_1^2)=3.12(v_1+v_2)(1-\delta)$}\vspace{0.6cm} \\
\includegraphics[height=0.8\textwidth]{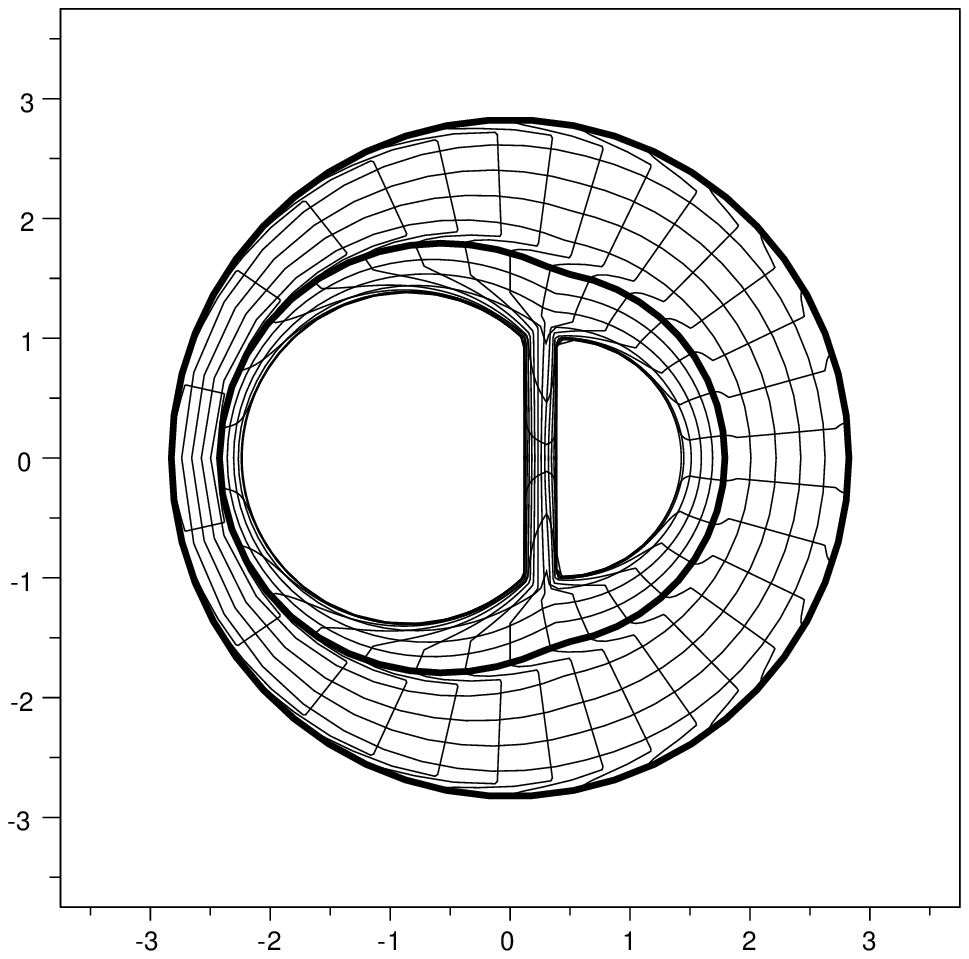}\vspace{0.2cm}\\
{\footnotesize (e) $\delta=0.4$, deformed configuration.\\ Thick line at $\vec u(\partial B(\vec a^*, R_1))$}
\end{center}
\end{minipage}
\quad
\begin{minipage}{0.31\textwidth}
\begin{center}
\quad\includegraphics[height=0.8\textwidth]{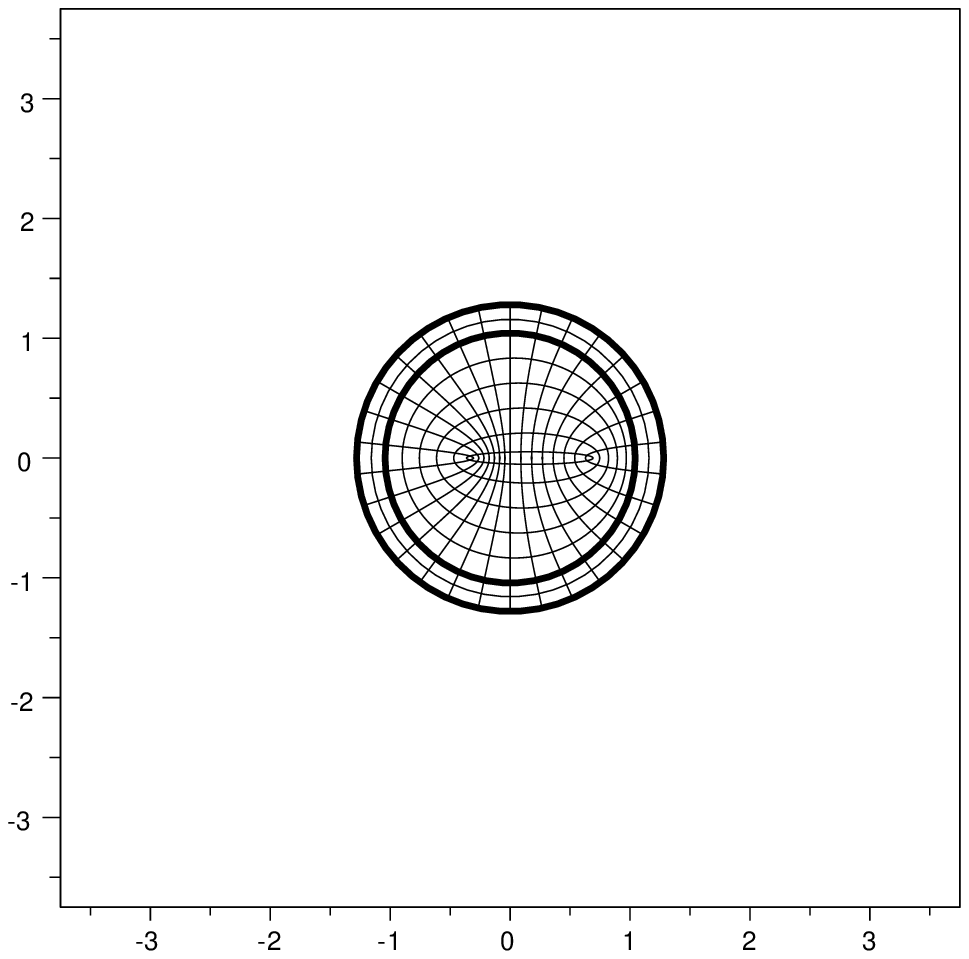}\vspace{0.2cm}\\
{\footnotesize (c) $\delta=0.9$, reference configuration,\\
 $\pi(R_2^2-R_1^2)=2.46(v_1+v_2)(1-\delta)$}\vspace{0.6cm} \\
\includegraphics[height=0.8\textwidth]{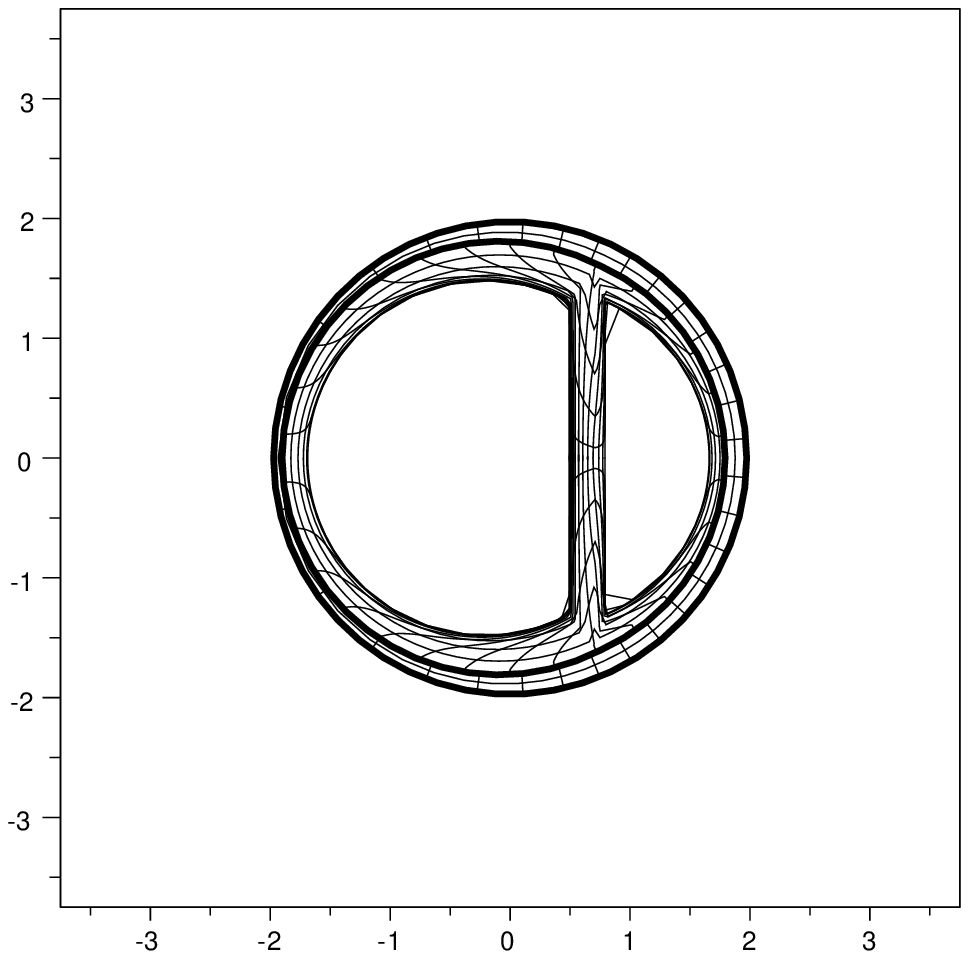}\vspace{0.2cm}\\
{\footnotesize (f) $\delta=0.9$, deformed configuration.\\ Thick line at $\vec u(\partial B(\vec a^*, R_1))$}
\end{center}
\end{minipage}
  \caption{\label{fig:transition}Transition to a radially symmetric map. A larger initial domain is necessary
in order to create spherical cavities.
Parameters: $\Omega=B(\vec 0, R_2)$, $\sqrt{\frac{v_1+v_2}{\pi d^2}}=1.5$, $\frac{v_2}{v_1}=0.3$, $d=1$, $R_1\approx d$.}
\end{figure}

Not all values of $R_1$ and $R_2$ are suitable for the existence of a solution, since
the reference configuration
$\{R_1\leq |\vec x -\vec {a^*}|\leq R_2\}$ must contain enough material to fill the space between
$\vec u(\partial B(\vec {a^*}, R_2))$
(with shape prescribed by the Dirichlet data)
and $\vec u(\partial B(\vec {\vec a^*}, R_1))$ (whose shape is determined by Theorem \ref{th:ub1}, see Figure \ref{fig:transition}).
In the case of a radially symmetric loading,
the farther $\Omega_1\cup \Omega_2$ is from
being a ball, the larger the reference configuration has to be.
If $\delta=1$ nothing has to be imposed;
if $\delta<1$, we must have that
\begin{align*}
\omega_n (R_2^n - R_1^n) \geq C (v_1+v_2)(1-\delta)
\end{align*}
for some constant $C$ (see Lemma \ref{le:nc}).
It turns out that the above necessary condition 
is also sufficient, as we show in the following theorem:

\begin{theorem} \label{th:ub2}
Suppose that $\vec a_1$, $\vec a_2 \in \R^n$ and
$d:=|\vec a_1-\vec a_2|>\ep_1+\ep_2$.
Let $\delta \in [0,1]$, $v_1\geq v_2\geq 0$,
 \begin{align} \label{eq:R1R2}
 V_{\delta}:=2^{2n+1}n(v_1+v_2)(1-\delta), \qquad R_1\geq  \max \left \{ \sqrt[n]{\frac{V_{\delta}}{\omega_n} }, 2d \right \},
\qquad R_2:= \sqrt[n]{R_1^n + \frac{V_{\delta}}{\omega_n} }.
 \end{align}
Then there exists $\vec{a^*}$
in the segment joining $\vec a_1$ and $\vec a_2$
and
a piecewise smooth homeomorphism $\vec u\in W^{1,\infty}(\R^n\setminus \{\vec a_1, \vec a_2\},\R^n)$
 such that
$\Det D\vec u = \mathcal L^n + v_1 \delta_{\vec a_1} + v_2 \delta_{\vec a_2}$ in $\R^n$,
$\vec u|_{\R^n \setminus B(\vec{a^*}, R_2)}$ is radially symmetric,
and for all $R\geq R_1$
\begin{multline*}
 \frac{1}{n}
\int_{B(\vec a^*, R) \setminus (B_{\varepsilon_1}(\vec a_1)\cup B_{\varepsilon_2}(\vec a_2))}
\left | \frac{D\vec u}{\sqrt{n-1}} \right |^n \dd\vec x
 \leq
C_1(v_1+v_2+ \omega_n R^n)
+v_1\log \frac{R}{\ep_1} + v_2\log \frac{R}{\ep_2}
\\
 + C_2(v_1+v_2)\left ( (1-\delta) \left (\log \sqrt[n]{\frac{V_{\delta}}{\omega_n d^n}}\right )_+ +
 \delta
\left(\sqrt[n]{\frac{v_2}{v_1}}  \log \frac{d}{\ep_1}
+ \sqrt[2n]{\frac{v_2}{v_1}} \log \frac{d}{\ep_2}
\right ) \right).
\end{multline*}
\end{theorem}

The main differences with respect to Theorem \ref{th:UB} are
that $\vec u$ is now radially symmetric in $\R^n \setminus B(\vec a^*, R_2)$
and that $\log \frac{R}{d}$ has been replaced with $\log \sqrt[n]{\frac{V_\delta}{\omega_n d^n}}
= C+ \log \sqrt[n]{\frac{(v_1+v_2)(1-\delta)}{\omega_n d^n}}$ in the interaction term.
The proof is presented in Section \ref{se:ubD}.
As a consequence
 we finally obtain
 \begin{corollary} \label{co:Dirich}
 Let $\Omega$ be a ball of radius $R\geq 2d$, with $d>\ep_1+\ep_2>0$.
Then, for every $v_1\geq v_2\geq 0$ there exist
$\vec a_1$, $\vec a_2 \in \Omega$ with
$|\vec a_1 - \vec a_2|= d$, and
 a Lipschitz homeomorphism
$\vec u: \Omega \setminus \{\vec a_1, \vec a_2\} \to \R^n$,
such that
$\Det D\vec u = \mathcal L^n + v_1 \delta_{\vec a_1} + v_2 \delta_{\vec a_2}$ in $\Omega$,
$\vec u|_{\partial \Omega}\equiv \lambda \id$ (with $\lambda^n-1:= \frac{v_1+v_2}{|\Omega|}$),
and
\begin{multline*}
 \frac{1}{n}
\int_{\Omega \setminus (B_{\varepsilon_1}(\vec a_1)\cup B_{\varepsilon_2}(\vec a_2))}
\left | \frac{D\vec u}{\sqrt{n-1}} \right |^n \dd\vec x
 \leq
C_1(v_1+v_2+ \omega_n R^n)
+v_1\log \frac{R}{\ep_1} + v_2\log \frac{R}{\ep_2}
\\
 + C_2(v_1+v_2)
\min_{\delta \in [\delta_0, 1]}\left ( (1-\delta) \left ( \log \frac{(v_1+v_2)(1-\delta)}{\omega_n d^n}\right )_+ +
 \delta
\left(\sqrt[n]{\frac{v_2}{v_1}}  \log \frac{d}{\ep_1}
+ \sqrt[2n]{\frac{v_2}{v_1}} \log \frac{d}{\ep_2}
\right ) \right)
\end{multline*}
with $\delta_0:=
\max \left \{0, 1- \frac{|\Omega|- 2^n \omega_n d^n}{4^{n+1}n\omega_n d^n}\right \}$.
\end{corollary}

The value of $\delta_0$ is such that $\delta\geq \delta_0$ if and only if $\omega_n R^n \geq \omega_n R_1^n + V_{\delta}$,
with $\omega_n R_1^n:= V_{\delta} + \omega_n (2d)^n$;
the idea is to be able to use Theorem \ref{th:ub2} and obtain a final energy estimate depending only
on $v_1$, $v_2$, $d$, $\ep_1$, $\ep_2$ and the size $|\Omega|$ of the domain.

\subsection{Convergence results}
Once we have upper and lower bounds, we are able to show that for ``almost-minimizers"
one of the three scenarii described after Theorem \ref{th:LB} holds
in the limit $\ep \to 0$.
\begin{theorem} \label{th:limits}
 Let $\Omega$ be an open and bounded set in $\R^n$, $n\geq 2$. Let $\ep_j \to 0$ be a sequence, that we will denote in the sequel simply by $\ep$.
Let $\{\Omega_\ep\}_{\ep}$ be a corresponding sequence of domains of the form
$\Omega_\ep=\Omega \setminus \bigcup_{i=1}^m \overline B_{\ep}(\vec a_{1,\ep})$,
with $m\in \N$, $\vec a_{1,\ep}, \ldots, \vec a_{m,\ep} \in \Omega$ and $\ep$
such that the balls $B_{\ep}(\vec a_{1,\ep}), \ldots, B_{\ep}(\vec a_{m,\ep})$ are disjoint.
Assume that for each $i=1,\ldots, m$ the sequence $\{\vec a_{i,\ep}\}_{\ep}$ is compactly contained in $\Omega$.
Suppose, further, that
 there exists $\vec u_\ep \in W^{1,n}(\Omega_\ep, \R^n)$ satisfying
condition INV,
$\Det D\vec u_\ep = \mathcal L^n$ in $\Omega_\ep$, $\sup_{\ep} \|\vec u_\ep\|_{L^\infty(\Omega_\ep)} < \infty$
and
\begin{align} \label{eq:hypUB}
 \frac{1}{n}\int_{\Omega_\ep}
   &\left |\frac{D\vec u_\ep(\vec x)}{\sqrt{n-1}}\right |^n\dd \vec x
\leq
  \left ( \sum_{i=1}^m v_{i,\ep}\right ) \log \frac{\diam \Omega}{\ep}
+ C\left ( |\Omega| +  \sum_{i=1}^m v_{i,\ep} \right),
\end{align}
where\footnote{Now we write $E(\vec a_{i, \ep}, \ep; \vec u_{\ep})$, and not just $E(\vec a_{i, \ep}, \ep)$, to
highlight the dependence on $\vec u_\ep$. It corresponds to the cavity opened
by $\vec u_\ep$ at $\vec a_{i, \ep}$ (compare with \eqref{isoper} and \eqref{df:Exr}).}
 $v_{i, \ep}:=|E(\vec a_{i,\ep}, \ep; \vec u_\ep)|-\omega_n \ep^n$
and $C$ is a universal constant.

Then (extracting a subsequence)
the limits $\vec a_i= \lim_{\ep\to0} \vec a_{i, \ep}$
and
$v_i=\lim_{\ep\to0}{v_{i, \ep}}$, $i=1, \ldots, m$
are well defined,
and there exists
$\vec u \in \cap_{1\le p<n} W^{1,p}(\Omega, \R^n)
\cap W^{1,n}_{\loc}\big (\Omega \setminus \{\vec a_1,\ldots, \vec a_{m}\}, \R^n \big)$
 such that
\begin{itemize}
 \item 
$\vec u_\ep \weakc \vec u$ in $W^{1,n}_{\loc}\big (\Omega \setminus \{\vec a_1, \ldots, \vec a_{m}\}, \R^n \big)$
 \item $\Det D\vec u_\ep \weakcs \Det D\vec u$
  in $\Omega \setminus \{\vec a_1, \ldots, \vec a_{m}\}$ locally in the sense of measures
 \item $\Det D\vec u
=
\sum_{i=1}^{m} v_i \delta_{\vec a_i}+ \mathcal{L}^n$
 in $\Omega$.
\end{itemize}

When $m=2$,
 one of the following holds:
\begin{enumerate}[i)]
 \item if $\vec a_1 \neq \vec a_2$ and $v_1,v_2 >0$ (assume without l.o.g.\ $v_1\geq v_2$), then
    \begin{itemize}
     \item the cavities $\imT(\vec u, \vec a_1)$ and $\imT (\vec u, \vec a_2)$
(as defined in \eqref{df:Exr})
are balls of volume $v_1$, $v_2$
      \item $|E(\vec a_{i, \ep}, \ep; \vec u_\ep)\triangle \imT(\vec u, \vec a_i) |\to 0$ as $\ep \to 0$ for $i=1,2$
      \item under the additional assumption that
$v_{1}+v_{2}  < 2^{n}\omega_n (\dist(\frac{\vec a_1+\vec a_2}{2}, \partial \Omega))^n$,
\begin{align*}
\frac{\omega_n |\vec a_2 - \vec a_1|^n}{v_1+v_2} \geq C_1 \exp \left(
-C_2 \left (
  1 + \frac{|\Omega|}{v_1+v_2} + \log \frac{\omega_n (\diam \Omega)^n}{v_1+v_2}
\right ) \left / \left (
 \frac{v_2}{v_1+v_2} \right)^{\frac{n}{n-1}} \right .\right )
\end{align*}
for some positive constants $C_1$ and $C_2$ depending only on $n$;
    \end{itemize}
 \item if \label{it:oneCavity} $\min \{v_1, v_2\}=0$ (say $v_2=0$),
then $im_T(\vec u, \vec a_1)$ (the only cavity opened by $\vec u$) is spherical;
 \item \label{it:coalesce} if $\vec a_1=\vec a_2$ and $v_1,v_2>0$ (assume $v_1\geq v_2$),
then
    \begin{itemize}
     \item $\imT(\vec u, \vec a_1)$ is a ball
of volume $v_1 + v_2$
      \item $|\vec a_{2,\ep}- \vec a_{1, \ep}|=O(\ep)$ as $\ep \to 0$
      \item the cavities must be distorted in the following sense ($C_n$ being as in Proposition \ref{pr:distortions}):
\begin{align} \label{eq:minDist}
 \liminf_{\ep \to 0}
\frac{v_1  D \big ( E(\vec a_{1,\ep}, \ep; \vec u_\ep )\big )^\frac{n}{n-1}
+ v_2 D \big (
  E(\vec a_{2,\ep}, \ep; \vec u_\ep )\big )^\frac{n}{n-1}}{v_1+v_2} 
 > C_n  \left( \frac{v_2}{v_1+v_2} \right)^{\frac{n}{n-1}}.
\end{align}
    \end{itemize}
\end{enumerate}
\end{theorem}

In the situation of two cavities, the three cases above correspond to the three scenarii of the end of Section  \ref{sec:1.2} in the same order.

The main ingredients for the proof are the comparison of the  upper bound \eqref{eq:hypUB} with  the lower
bounds Proposition \ref{pro1} and Theorem \ref{th:LB}, standard compactness arguments,
and an argument introduced by Struwe \cite{Struwe94}
in the context of Ginzburg-Landau
which allows to deduce from the energy bounds sufficient compactness of $\vec u_{\ep}$.

\subsection{Additional comments and remarks}

We note first that our analysis works provided that the distance of the
cavitation points to the boundary does not get small (thus the domain cannot be too thin either).
It is an interesting question to better understand what happens when they do get close to
the boundary, as well as the effect of the boundary conditions.

Second, it follows from our work that it is always necessary to compare quantities
in the reference configuration with quantities in the deformed configuration,
due to the scale-invariance in elasticity.
For example, we have shown that a large price needs to be paid (in terms of elastic energy)
in order to open spherical cavities
whenever the distance between the cavitation points
is small \emph{compared
to the final size of the cavities} ($\omega_n d^n \ll v_1+v_2$).
If we only know that the cavitation points are becoming closer and closer to each other,
from this alone we cannot conclude that the cavities
will interact and that the total elastic energy will go to infinity,
as the following argument shows.
Suppose that $\vec u$ is an incompressible map defined on the unit cube $Q\subset \R^n$,
opening a cavity,
and satisfying affine boundary conditions of the form $\vec u(\vec x)\equiv \vec A \vec x$ on $\partial Q$,
$\vec A\in \R^{n\times n}$.
Then,
by rescaling $\vec u$ and reproducing it periodically,
it is possible to construct a
sequence of incompressible maps creating an increasingly large number of cavities,
at cavitation points that are closer and closer to each other,
in such a way that all the deformations in the sequence
have exactly the same elastic energy
(cf.\ Ball \& Murat \cite{BaMu84}; see also \cite{Reina10t,LoIdNa11,LoNaId11}).
This is possible because the cavities themselves are also becoming increasingly smaller, with radii decaying at the
same rate as the distance between neighbouring cavitation points.
This example also shows that the strategy of filling the material with an arbitrarily large number
of small cavities is, in a sense, equivalent to forming a single big cavity
(there is no interaction between the singularities).
Here we complement that result by showing
that if it is not possible to create an infinite number of cavities, then the interaction effects
in the energy do become noticeable, and under some circumstances can even be quantified.

Third, we mention that the idea of
partitioning the domain and
using angle-preserving maps
inside the resulting subdomains (as described in Section \ref{sec:upperbd})
can be used to produce test maps that
are incompressible and open any prescribed number of cavities
(for example by dividing the initial domain in angular sectors).

Finally, we discuss the case $p\neq n$.
It is not clear how to extend the analysis to this case,
the main reason being that the energy is no longer conformally invariant while the ``ball-construction method" is
only suited for such cases.
To see this in a simple way, let us consider the case of two cavities, assuming incompressibility,
letting $\varepsilon_1=\varepsilon_2 \to 0$, and  let us try to reproduce the steps  \eqref{csgl} and \eqref{cs2} with   \eqref{isoperbis}.
The $p$-equivalent of \eqref{isoperbis} obtained by H\"older's inequality
(and by relating $|D\vec u|^{n-1}$ to the area element $|(\cof D\vec u)\vecg \nu|$, see Lemma \ref{le:Basic}) is
\begin{align*}
  \int_{\partial B(\vec a, r)} \left |\frac{D\vec u(\vec x)}{\sqrt{n-1}}\right|^p \dd\mathcal H^{n-1}(\vec x)
\geq \frac{\Per(E(\vec a, r))^{\frac{p}{n-1}}}{(n\omega_n r^{n-1})^{\frac{p}{n-1}-1}}
\geq n\omega_n^{\frac{n-p}{n}} \frac{|E(\vec a, r)|^{\frac{p}{n}}}{r^{1-(n-p)}}
\left ( 1 + CD(E(\vec a, r))^\frac{p}{n-1} \right ).
\end{align*}
According to this, when $p\ne n$  we may bound from below the energy in
$B(\vec a_1, \frac{d}{2}) \cup B(\vec a_2, \frac{d}{2}) $
(with $d=|\vec a_2-\vec a_1|$) by
\begin{align*}
  \underset{B(\vec a_1, \frac{d}{2})\cup B(\vec a_2, \frac{d}{2})}{\int}
  \frac{\omega_n^{\frac{p-n}{n}}}{n} \left | \frac{D\vec u(\vec x )}{\sqrt{n-1}}\right |^p
\geq
 \left (v_1^\frac{p}{n} + v_2^\frac{p}{n} \right ) \left (\frac{d}{2} \right )^{n-p}
+ C(v_1+v_2)^\frac{p}{n}\int_0^\frac{d}{2}  \langle D(E(\vec a_i, r))^\frac{p}{n-1} \rangle r^{n-p-1},
\end{align*}
where $\langle D(E(\vec a_i, r)))^\frac{p}{n-1} \rangle$ stands for the average distortion
\begin{align*}
\langle D(E(\vec a_i, r)))^\frac{p}{n-1} \rangle :=
\left ( v_1^\frac{p}{n} D(E(\vec a_1, r))^\frac{p}{n-1}
				+ v_2^\frac{p}{n} D(E(\vec a_2, r))^\frac{p}{n-1}\right ) (v_1+v_2)^{-\frac{p}{n}}.
\end{align*}
Analogously, we can bound the energy in
$B(\vec a, R) \setminus \overline B(\vec a, d)$ (with $\vec a = \frac{\vec a_1+\vec a_2}{2}$) by
\begin{align*}
  \underset{B(\vec a, R)\setminus \overline B(\vec a, d)} {\int}
  \frac{\omega_n^{\frac{p-n}{n}}}{n} \left | \frac{D\vec u(\vec x )}{\sqrt{n-1}}\right |^p
\geq
 (v_1+v_2)^\frac{p}{n} \int_d^R r^{n-p-1}
+ C(v_1+v_2)^\frac{p}{n} \int_d^R D(E(\vec a, r))^\frac{p}{n-1} r^{n-p-1}
\end{align*}
 and obtain:
\begin{multline*}
   \int_{B(\vec a, R)}
  \frac{\omega_n^{\frac{p-n}{n}}}{n} \left | \frac{D\vec u(\vec x )}{\sqrt{n-1}}\right |^p   
\geq
  (v_1+v_2)^\frac{p}{n} \left ( \int_0^\frac{d}{2} + \int_d^R \right) r^{n-p-1}
  + \underbrace{\left (v_1^\frac{p}{n} + v_2^\frac{p}{n} - (v_1+v_2)^\frac{p}{n}\right ) \left (\frac{d}{2} \right )^{n-p}}_{II}
\\  
  + \underbrace{C(v_1+v_2)^\frac{p}{n} \left [
	\int_0^\frac{d}{2}  \langle D(E(\vec a_i, r))^\frac{p}{n-1} \rangle r^{n-p-1}
  + \int_d^R D(E(\vec a, r))^\frac{p}{n-1} r^{n-p-1} \right ]}_{III}.
\end{multline*}

Assume that $v_1+v_2$ is fixed (as is the case in the Dirichlet problem).
Let us first consider the case $p<n$.
Since the limit $\ep\to 0$
is not singular in this case (contrarily to $p=n$),
the problem cannot be analyzed by asymptotic analysis.
If we guide ourselves only by the second and third terms (II and III), when $p<n$ we can say the following.
The factor $v_1^\frac{p}{n} + v_2^\frac{p}{n} - (v_1+v_2)^\frac{p}{n}$ in II
is minimized when $\min \{v_1, v_2\}=0$, hence it
motivates the creation of just one cavity (the same can be said for the problem with $M$ cavities, because
$v_1^\frac{p}{n} + \cdots + v_M^\frac{p}{n}$ is concave and the restriction
$v_1+\ldots + v_M = \text{const.}$\ is linear).
If the above difference has to be positive,
the factor $\left (\frac{d}{2}\right) ^n$ suggests that the two cavitation points would want to be arbitrarily close,
and that the cavities will tend to act as a single cavity.
This is consistent with the prediction for III;
indeed, consider the corresponding estimate for $p=n$:
\begin{multline*}
  \frac{1}{n}\int_{\Omega_{\ep}\cap B(\vec a, R)}
  \left | \frac{D\vec u(\vec x )}{\sqrt{n-1}}\right |^n
 \dd \vec x
 \geq
  (v_1+v_2) \left ( \int_{\varepsilon}^\frac{d}{2} + \int_d^R \right) \frac{\dd r}{r}
\\
  + C(v_1+v_2) \left [
	\int_{\varepsilon}^\frac{d}{2}  \langle D(E(\vec a_i, r))^\frac{n}{n-1} \rangle\frac{\dd r}{r}
  + \int_d^R D(E(\vec a, r))^\frac{p}{n-1} \frac{\dd r}{r} \right ].
\end{multline*}
Under a logarithmic cost, it is much more important to minimize the distortions $D(E(\vec a_i, r))$
of the circles $\vec u(\partial B(\vec a_i, r))$, $i=1,2$, $\ep<r<\frac{d}{2}$
near the cavities, rather than the distortion of the outer circles
$D(E(\vec a, r))$, $r>d$. As was discussed before, this leads either to the case
of well-separated and spherical cavities (scenario (i) in p.~\pageref{it:scn1}),
or to the conclusion that if outer circles are mapped to circles (scenario (iii))
then the distance between cavitation points
must be of order $\ep$ (Theorem \ref{th:limits}\ref{it:coalesce})).
In contrast,
When $p<n$, in the presence of the weight $r^{n-p-1}$,
minimizing the distortions $D(E(\vec a, r))$, $r>d$ gains more relevance compared
to the distortion near the cavities.

For the previous reasons, we believe that the deformations of scenario (i) will not be global minimizers, instead
the body will prefer to open a single cavity. If multiple cavities have to be created, then the cavitation points
will try to be close to each other, and the deformation will try to rapidly become radially symmetric. The cavities
will be distorted and try to act as a single cavity (as in scenario (iii),
which creates a state of strain potentially leading to fracture by coalescence),
at distances between the cavitation points that are of order $1$
(not of order $\ep$). This, in fact, is what has been observed numerically \cite{XuHe11,LiLi11}.

Let us now turn to $p>n$. The lower bound reads
\begin{multline*}
  \int_{\Omega_{\ep}\cap B(\vec a, R)}
  \frac{\omega_n^{\frac{p-n}{n}}}{n}
\left | \frac{D\vec u(\vec x )}{\sqrt{n-1}}\right |^p   \dd \vec x
+ \frac{(v_1+v_2)^\frac{p}{n}}{p-n} R^{n-p}
 \geq
  \underbrace{(v_1^\frac{p}{n} + v_2^\frac{p}{n}) \int_\ep^\frac{d}{2} r^{n-p-1}}_{I}
  + \underbrace{(v_1+v_2)^\frac{p}{n} d^{n-p}}_{II}
\\
  + C(v_1+v_2)^\frac{p}{n} \left [
	\int_\ep^\frac{d}{2}  \langle D(E(\vec a_i, r))^\frac{p}{n-1} \rangle r^{n-p-1}
  + \int_d^R D(E(\vec a, r))^\frac{p}{n-1} r^{n-p-1} \right ].
\end{multline*}
This time the limit $\ep\to 0$ is singular, even more so than for $p=n$.
The factor $v_1^\frac{p}{n} + v_2^\frac{p}{n}$ is now minimized when the cavities have equal volumes. Regarding $d$,
the first term prefers small distances ($d=2\ep$) while the second prefers $d\to \infty$;
since $(v_1+v_2)^\frac{p}{n} > v_1^\frac{p}{n} + v_2^\frac{p}{n}$, it can be said that II
has a stronger influence, hence $d$ large should be
preferred\footnote{although in order to be sure
it would be necessary to compute the energy
in the transition region $B(\vec a, d) \setminus (B(\vec a_1, \frac{d}{2}) \cup B(\vec a_2, \frac{d}{2}))$}.
With respect to the third term, it is now much more vital
to create spherical cavities (so as to minimize the first of the two
integrals) than when $p=n$. This implies that it is scenario (i), rather than (ii) or (iii),
which should be observed.

The case $p<n$, therefore, should favour a single cavity and coalescence, $p>n$ should favour
 many cavities and splitting, and
both situations are possible in the borderline case that we have studied: $p=n$.

\subsection{Plan of the paper}
In Section \ref{se:notation} we describe our notation and recall the notions of perimeter, reduced boundary,
topological image, distributional determinant, and the invertibility condition INV.
In Section \ref{se:LB} we begin by
extending \eqref{isoperbis} to the case of an arbitrary power $p$ and space dimension $n$ (Lemma \ref{le:Basic}).
In Section \ref{ballconstruction} we prove
the lower bound for an arbitrary number of cavities using the ball construction method
(Proposition \ref{pro1}).
In Section \ref{se:mainLB}, we prove the main lower bound (Theorem \ref{th:LB})
 and postpone the proof of our estimate on the distortions (Proposition \ref{pr:distortions}) to Section \ref{se:distortions}.
The energy estimates for the angle-preserving ansatz are presented in Section \ref{se:ub1} and proved in Section \ref{se:auxi}.
In Section \ref{se:ub2} we show how to complete the maps away from the cavitation points so as to fulfil
the boundary conditions, and in Section \ref{se:numerics} we comment briefly on the numerical computations
presented in this paper based on the constructive method of Dacorogna \& Moser \cite{DaMo90}.
Finally, the proof of the main compactness result and of the fact that in the limit
only one of the three scenarii holds
(Theorem \ref{th:limits}) is given in Section \ref{se:limits}.

\subsection{Acknowledgements}
We give special thanks to G.~Francfort for his interest and his involvement in this project.
We also thank J.-F.~Babadjian, J.~Ball, Y.~Brenier, A.~Contreras, R.~Kohn,
R.~Lecaros, G.~Mingione, C.~Mora-Corral, S.~M\"uller, T.~Rivi\`ere, and N.~Rougerie for useful discussions.

\section{Notation and preliminaries} \label{se:notation}

\subsection{General notation}

Let $n$ denote the space dimension.
Vector-valued and matrix-valued quantities will
be written in bold face.
The set of unit vectors in $\R^n$ is denoted by $\S^{n-1}$.
Given a set $E\subset \R^n$, $\lambda\geq 0$ and $\vec h \in \R^n$, we define
$\lambda E :=\{\lambda\vec x:\ \vec x \in E\}$ and $E+\vec h:=\{\vec x +\vec h:\ \vec x \in E\}$.
The interior and the closure of $E$ are denoted by $\Int E$ and $\overline E$,
and the symmetric difference of two sets $E_1$ and $E_2$ by $E_1 \triangle E_2$.
If $E_1$ is compactly contained in $E_2$, we write $E_1 \ssubset E_2$.
The notations $B(\vec x, R)$, $B_R(\vec x)$ are used for the open ball of radius $R$ centred at $\vec x$,
and $\overline B(\vec a, R)$, $\overline B_R(\vec a)$ for the corresponding closed ball.
The distance from a point $\vec x$ to a set $E$ is denoted by $\dist(\vec x, E)$,
the distance between sets by $\dist(E_1, E_2)$, and the diameter of a set by $\diam E$. 

Given $\vec A$ an $n\times n$ matrix, $\vec A^T$ will be its transpose, $\det \vec A$ its determinant,
and $\cof \vec A$ its cofactor matrix (defined by $\vec A^T\cof \vec A=(\det \vec A)\vec 1$,
where $\vec 1$ stands for the $n\times n$ identity matrix).
The adjugate matrix of $\vec A$ is $\adj \vec A=(\cof \vec A)^T$.

The Lebesgue and the $k$-dimensional Hausdorff measure are denoted by $\mathcal L^n$ and $\mathcal H^k$,
respectively. If $E$ is a measurable set, $\mathcal L^n(E)$ is also written $|E|$
(as well as $|I|$ for the length of an interval $I$). The measure of the $k$-dimensional unit ball
is $\omega_k$ (accordingly, $\mathcal H^{n-1}(\partial B(\vec x, r))=n\omega_n r^{n-1}$).
The exterior product of $1\leq k\leq n$ vectors $\vec a_1, \ldots, \vec a_k \in \R^n$
is denoted by $\vec a_1 \wedge \cdots \wedge \vec a_k$ or $\bigwedge_{i=1}^k \vec a_i$.
It is $k$-linear, antisymmetric,
and such that $|\vec a_1 \wedge \cdots \wedge \vec a_k|$ is the $k$-dimensional measure
of the $k$-prism formed by $\vec a_1, \ldots \vec a_k$ (see, e.g., \cite{Federer69,Spivak65,GiMoSo98I,AmFuPa00}).
In particular,
$|\vec x |^2 = |\vec x \cdot \vec e|^2 + |\vec x \wedge \vec e |^2$
 for all $\vec x \in \R^n$ and $\vec e\in \S^{n-1}$.
With a slight abuse of notation, when $k=n$
the expression $\vec a_1 \wedge \cdots \wedge \vec a_n$ is used to denote
the determinant (in the standard basis) of the matrix
with column vectors $\vec a_1, \ldots, \vec a_n \in \R^n$.

The characteristic function of a set $E$ is referred to as $\chi_E$,
and the restriction of $\vec u$ to $E$ as $\vec u|_{E}$.
The sign function
$\sgn:\R\to \{-1,0,1\}$
is given by $\sgn x = x/|x|$ if $x\neq 0$, $\sgn 0=0$.
The notation $\id$ is used for the identity function $\id(\vec x)\equiv\vec x$.
The symbol
$\fint_E f$ stands for the integral average $\frac{1}{|E|}{\int_E f}$.
The support of a function $f$ is represented by $\spt f$.

The space of infinitely differentiable functions with compact support is denoted by
$C_c^\infty(\Omega)$, and the $L^p$ norm of a function $f$ by $\|f\|_{L^p}$.
Sobolev spaces are denoted by $W^{1,p}(\Omega, \R^n)$, as usual.
The Hilbert space $W^{1,2}(\Omega, \R^n)$ is denoted by $H^1(\Omega, \R^n)$.
The weak derivative (the linear transformation) of a map
$\vec u \in W^{1,p}(\Omega, \R^n)$ at a point $\vec x\in \R^n$
is identified with the gradient $D\vec u(\vec x)$ (the matrix of weak partial derivatives).

Use will be made of the coarea formula (see, e.g., \cite{Federer69,EvGa92,AmFuPa00}): if $E\subset \R^n$ is
measurable and $\phi:E \to \R$ is Lipschitz, then for all $f\in L^1(E)$
\begin{align*}
 \int_{E} f(\vec x)|D\phi(\vec x)| \dd\vec x = \int_{-\infty}^\infty
\left ( \int_{\{\vec x\in E: \phi(\vec x)=t\}} f(\vec x) \dd \mathcal H^{n-1}(\vec x )\right ) \dd t .
\end{align*}

\subsection{Perimeter and reduced boundary}

\begin{definition}
The perimeter of a measurable set $E\subset \R^n$ is defined as
\begin{equation*}
 \Per E:= \sup \left \{ \int_E \div \vec g(\vec y)\dd\vec y: \vec g \in C^1_c(\R^n, \R^n), \|\vec g\|_{\infty}\leq 1 \right\}.
\end{equation*}
\end{definition}

\begin{definition}
Given $\vec y_0 \in \R^n$ and a non-zero vector $\vecg\nu\in \R^n$, we define
\begin{equation*}
 H^+(\vec y_0, \vecg\nu):= \{\vec y \in \R^n: (\vec y - \vec y_0)\cdot \vecg\nu \geq 0\}, \quad
 H^-(\vec y_0, \vecg\nu):= \{\vec y \in \R^n: (\vec y - \vec y_0)\cdot \vecg\nu \leq 0\}.
\end{equation*}
The reduced boundary of a measurable set $E\subset \R^n$, denoted by $\partial^* E$, is defined as the set of points $\vec y \in \R^n$ for which there exists a unit vector $\vecg\nu \in \R^n$ such that
\begin{equation*}
 \lim_{r\to 0^+} \frac{|E\cap H^-(\vec y, \vecg\nu) \cap B(\vec y, r)|}{|B(\vec y, r)|}= \frac{1}{2} \quad \text{and} \quad \lim_{r\to 0^+} \frac{|E\cap H^+(\vec y, \vecg\nu)\cap B(\vec y, r)|}{|B(\vec y, r)|} =0.
\end{equation*}
If $\vec y\in \partial^* E$ then $\vecg\nu$ is uniquely determined and is called the unit outward normal to $E$.
\end{definition}

The definition of perimeter
coincides precisely with the $\mathcal H^{n-1}$-measure of the reduced boundary,
as follows from the well-known results of Federer, Fleming and De Giorgi
(see, e.g., \cite{Federer69,Ziemer89,EvGa92,AmFuPa00})\footnote{When
$\Per E=\infty$, the result is true if we consider
the \emph{measure-theoretic} boundary, as defined in \cite[Th.\ 5.11.1]{EvGa92}.
For sets of finite perimeter the two notions of boundary coincide $\mathcal H^{n-1}$-a.e.,
thanks to a result of Federer \cite{Federer69}
(also available in \cite[Th.\ 3.61]{AmFuPa00}, \cite[Lemma 5.8.1]{EvGa92}, or \cite[Sect.\ 5.6]{Ziemer89}).
}.

\subsection{Degree and topological image} \label{sec:degree}
We begin by recalling the notion of topological degree for maps $\vec u$ that are only weakly differentiable
\cite{MuSpTa96,FoGa95,BrNi95,CoDeLe03}.

If $\vec u \in W^{1,p}(\Omega, \R^n)$ and $\vec x \in \R^n$,
then, for a.e.\ $r\in (0,\infty)$ with $\partial B(\vec x, r)\subset \Omega$,
\begin{itemize}
\item[(R1)] $\vec u(\vec z)$ and $D\vec u(\vec z)$ are defined at $\mathcal H^{n-1}$-a.e.\ $\vec z\in \partial B(\vec x, r)$
\item[(R2)] $\vec u|_{\partial B(\vec x, r)} \in W^{1,p}(\partial B(\vec x, r), \R^n)$
\item[(R3)] $D(\vec u|_{\partial B(\vec x, r)})(\vec z)= (D\vec u (\vec z))|_{T_{\vec z}(\partial B(\vec x, r))}$
(the $n$-dimensional and the tangential weak derivatives coincide;
$T_{\vec z}(\partial B(\vec x, r))$ denotes the tangent plane)
for $\mathcal H^{n-1}$-a.e.\ $\vec z \in \partial B(\vec x, r)$
\end{itemize}
(this follows by approximating by $C^\infty$ maps and using the coarea formula).
If, moreover, $p>n-1$, then, by Morrey's inequality,
there exists a unique map $\overline{\vec u} \in C^0(\partial B(\vec x, r))$ that coincides
with $\vec u|_{\partial B(\vec x, r)}$ $\mathcal H^{n-1}$-a.e.
With an abuse of notation we write $\vec u(\partial B(\vec x, r))$ to denote $\overline{\vec u}(\partial B(\vec x, r))$.

If $p>n-1$ and (R2) is satisfied,
for every $\vec y \in \R^n\setminus \vec u(\partial B(\vec x, r))$
we define $\deg(\vec u, \partial B(\vec x, r), \vec y)$
as the classical Brouwer degree \cite{Schwartz69,FoGa95}
of $\vec u|_{\partial B(\vec x, r)}$ with respect to $\vec y$.
The degree $\deg(\vec u, \partial B(\vec x, r), \cdot)$ is the only $L^1(\R^n)$ map \cite{MuSpTa96,BrNi95}
such that
\begin{equation} \label{eq:degree}
 \int_{\R^n} \deg(\vec u, \partial B(\vec x, r), \vec y) \div \vec g (\vec y)\dd\vec y = \int_{\partial B(\vec x, r)} \vec g(\vec u(\vec z))\cdot (\cof D\vec u(\vec z))\vecg \nu(\vec z) \dd\mathcal H^{n-1}(\vec z)
\end{equation}
for every $\vec g \in C^1(\R^n, \R^n)$, $\vecg \nu(\vec z)$ being
the outward unit normal to $\partial B(\vec x, r)$.

For a map $\vec u \in W^{1,p}(\Omega, \R^n)$ that is invertible, orientation-preserving, and regular
except for the creation of a finite number of cavities,
$\deg(\vec u, \partial B(\vec x, r), \vec y)$ is equal to $1$, roughly speaking, only at those points $\vec y$
enclosed by $\vec u(\partial B(\vec x, r))$. Because of this, the degree is useful for the study of cavitation,
since
we can detect a cavity by looking at the set of points where the degree is $1$, but which do not belong
to the image of $\vec u$ (they are not part of the deformed body). This gave rise to
{{\v S}ver{\'a}k's notion of topological image \cite{Sverak88}.
\begin{definition} \label{df:imt}
Let $\vec u \in W^{1,p}(\partial B(\vec x, r), \R^n)$ for some $\vec x \in \R^n$, $r>0$, and $p>n-1$. Then
\begin{equation*}
\imT(\vec u, B(\vec x, r)):=\{\vec y \in \R^n:\ \deg(\vec u, \partial B(\vec x, r), \vec y)\neq 0\}.
\end{equation*}
\end{definition}

It was pointed out by M\"uller-Spector \cite[Sect.\ 11]{MuSp95} that Sobolev maps may create cavities
in some part of the body,
and subsequently fill them with material from somewhere else (even if they are one-to-one a.e.\
\cite{Ball81}). In order to avoid this pathological behaviour,
they defined a stronger invertibility condition, based on the topological
image\footnote{The original
definition of condition INV in \cite[Sect.~3]{MuSp95} required that i) and ii) were satisfied
only for a.e.~$r\in (0, \infty)$ such that $B(\vec x, r)\subset \Omega$.
Here we impose i) and ii) for a.e.~$r\in (0, \infty)$ such that $\partial B(\vec x, r) \subset \Omega$.
As explained in \cite{Henao09}, this modification is necessary when considering perforated domains,
due to Sivaloganathan \& Spector's example of leakage between cavities \cite[Sect.~6]{SiSp06}.}.

\begin{definition} \label{df:INV}
Let $\vec u \in W^{1,p}(\Omega, \R^n)$ with $p>n-1$. We say that $\vec u$ satisfies condition INV if 
\begin{enumerate}[i)]
\item \label{cd:INV1} $\vec u(\vec z) \in \imT(\vec u, B(\vec x, r))$ for a.e.\ $\vec z \in B(\vec x, r)\cap \Omega$
\item \label{cd:INV2} $\vec u(\vec z) \in \R^n \setminus \imT(\vec u, B(\vec x, r))$ for a.e.\ $\vec z \in \Omega \setminus B(\vec x, r)$
\end{enumerate}
for every $\vec x \in \R^n$ and a.e.\ $r\in(0,\infty)$ such that $\vec u|_{\partial B(\vec x, r)} \in W^{1,p}(\partial B(\vec x, r), \R^n)$.
\end{definition}

In the following proposition we summarize some of the main virtues of condition INV.
We add a sketch of the proof to make it easier for the interested reader to compile the different ideas
and conciliate the different notation in
\cite{Sverak88},
\cite[Lemmas 2.5, 3.5 and 7.3]{MuSp95},
\cite[Lemmas 3.8 and 3.10]{CoDeLe03},
\cite[Lemma 2]{HeMo11}, and
\cite[Prop.\ 6 and Lemma 15]{HeMopre}.

\begin{proposition} \label{pr:cof}
 Let $\vec u \in W^{1,p}(\Omega, \R^n)$ with $p>n-1$ satisfy $\det D\vec u>0$ a.e.\ and condition INV.
Then, for every $\vec x \in \R^n$ there exists a
full-$\mathcal L^1$-measure subset $R_{\vec x}$ of
$\{r\in (0, \infty): \partial B(\vec x, r)\subset\Omega\}$
for which (R1)--(R3), conditions \ref{cd:INV1})-\ref{cd:INV2}) of Definition \ref{df:INV},
and the following properties are satisfied:
\begin{enumerate}[i)]
 \item $\deg(\vec u, \partial B(\vec x, r), \vec y) \in \{0,1\}$ for every $\vec y \in \R^n \setminus {\vec u}(\partial B(\vec x, r))$
 \item $\partial^* \imT(\vec u, B(\vec x, r))= {\vec u}(\partial B(\vec x, r))$ up to $\mathcal H^{n-1}$-null sets
 \item \label{cd:perCof} $\displaystyle \Per \big ( \imT(\vec u, B(\vec x, r)) \big )=
  \int_{\partial B(\vec x, r)} |(\cof D\vec u (\vec z))\vecg\nu(\vec z)| \dd\mathcal H^{n-1}(\vec z)$
 \item \label{cd:volImt} $|\imT(\vec u, B(\vec x, r))|= \displaystyle \frac{1}{n} \int_{\partial B(\vec x, r)} \vec u(\vec z) \cdot (\cof D\vec u(\vec z)) \vecg\nu(\vec z) \dd\mathcal H^{n-1}(\vec z)$.
 \end{enumerate}
Moreover, for every $\vec x, \vec x' \in \R^n$ and every $r\in R_{\vec x}$, $r' \in R_{\vec x'}$
 \begin{enumerate}[i)] \setcounter{enumi}{4}
  \item \label{cd:monot} $\imT(\vec u, B(\vec x, r)) \subset \imT(\vec u, B(\vec x', r'))$ if $B(\vec x, r)\subset B(\vec x', r')$
  \item \label{cd:disjoint} $\imT(\vec u, B(\vec x, r)) \cap \imT(\vec u, B(\vec x', r'))=\varnothing$ if $B(\vec x, r)\cap B(\vec x', r')=\varnothing$.
 \end{enumerate}
\end{proposition}

\begin{proof}
 Call $\Omega_0$ the set of $\vec x\in \Omega$ for which there exist $\vec w\in C^1(\R^n, \R^n)$ and a
compact set $K\subset \Omega$ such that
\begin{equation} \label{eq:Omega0}
\lim_{r\to 0^+} \frac{|K\cap B(\vec x, r)|}{|B(\vec x, r)|}=1, \quad \vec u|_K = \vec w|_K,\quad \text{and}\quad D\vec u|_K = D\vec w|_K.
\end{equation}
Since $\vec u\in W^{1,p}(\Omega, \R^n)$, it is possible to
 find (combining Federer's approximation of approximately differentiable maps by Lipschitz functions,
Rademacher's theorem,
and Whitney's extension theorem, see, e.g.,
\cite[Cor.\ 6.6.3.2]{EvGa92}, \cite[Thms.\ 3.1.8 and 3.1.16]{Federer69}, \cite[Prop.\ 2.4]{MuSp95},
 \cite[Lemma 1]{HeMo11}) an increasing sequence of compact sets $\{K_j\}_{j\in \N}$ contained
in $\Omega$, and a sequence $\{\vec w_j\}_{j\in \N}$ of maps in $C^1(\R^n, \R^n)$, such
that $\vec u|_{K_j}=\vec w_j|_{K_j}$, $\nabla \vec u_j|_{K_j}=D\vec w|_{K_j}$,
and $|\Omega \setminus K_j|<\frac{1}{j}$ for each $j\in \N$. By
Lebesgue's differentiation theorem, $|K_j\setminus K_j'|=0$
 where $K_j':=\{\vec x \in K_j: \lim_{r\to 0^+} (r^{-n}|B(\vec x, r)\setminus K|)=0\}$.
 Since $\Omega_0\supset \bigcup_{j\in \N} K_j'$, it follows that $|\Omega \setminus \Omega_0|=0$.

Define $R_{\vec x}$ as the subset of
$\{r\in (0, \infty): \partial B(\vec x, r)\subset\Omega\}$
for which (R1)--(R3), conditions \ref{cd:INV1})-\ref{cd:INV2}) of Definition \ref{df:INV},
and the following properties are satisfied:
\begin{itemize}
\item[(R4)] $\mathcal H^{n-1}(\partial B(\vec x, r) \setminus \Omega_0)=0$
\item[(R5)] $\det D\vec u(\vec z)>0$ for $\mathcal H^{n-1}$-a.e.\ $\vec z\in \partial B(\vec x, r)$.
\end{itemize}
The fact that $|\{r\in (0, \infty): \partial B(\vec x, r)\subset\Omega\}\setminus R_{\vec x}|=0$
is a consequence of the coarea formula and of the discussion before
 Definition \ref{df:imt}.
For this choice of $R_{\vec x}$ we have that the properties listed in the proposition are satisfied
for all (not only for a.e.) $r\in R_{\vec x}$. This follows from \eqref{eq:degree},
the fact that
$\vec u_{|\Omega_0}$ is one to one (by \cite[Lemmas 3.4 and 2.5]{MuSp95}; only minor modifications are required,
see \cite[Lemma 2]{HeMo11} if necessary), and a careful inspection of the proofs of
\cite[Lemmas 2.5, 3.5 and 7.3]{MuSp95}.
\end{proof}

By Proposition \ref{pr:cof}\ref{cd:monot}) the topological image of $B(\vec x, r)$
can be defined for all $\vec x\in \R^n$
and all $r\geq 0$ such that $\{\vec z: r<|\vec z|<r+\delta\} \subset \Omega$ for some $\delta>0$
(not only for radii $r\in R_{\vec x}$).
Indeed, since the sequence $\{\imT(\vec u, B(\vec x, r)): r\in R_{\vec x}\}$ is increasing
for every $\vec x \in \R^n$,
we may define
\begin{equation} \label{df:Exr}
 E(\vec x, r):= \bigcap_{\substack{r'>r\\ r'\in R_{\vec x}}}\imT(\vec u, B(\vec x, r)).
\end{equation}
Whenever explicit mention of $\vec u$ is necessary
(such as in Theorem \ref{th:limits} where sequences of deformations are considered),
we write $E(\vec a, r; \vec u)$.
Finally,
if a point $\vec a \in \R^n$ is such that
$B(\vec a, \delta)\setminus \{\vec a\} \subset \Omega$ for some $\delta>0$,
we define its topological image as $E(\vec a_i, 0)$, and denote it by $\imT(\vec u, \vec a)$.

\subsection{The distributional determinant} \label{se:DetDu}
It is well known that the Jacobian determinant of a $C^2$ vector-valued map
$\vec u:\Omega\subset \R^n \to \R^n$ has a divergence structure. When $n=2$ or $n=3$,
this is
\begin{align*}
 \det D\vec u &= u_{1,1}u_{2,2} - u_{2,1}u_{1,2} = (u_{1}u_{2,2})_{,1} - (u_{1}u_{2,1})_{,2}
\end{align*}
\begin{align*}
 \det D\vec u &=
    u_{1,1}\left | \begin{array}{cc} u_{2,2} & u_{2,3} \\ u_{3,2} & u_{3,3} \end{array} \right |
  + u_{1,2}\left | \begin{array}{cc} u_{2,3} & u_{2,1} \\ u_{3,3} & u_{3,1} \end{array} \right |
  + u_{1,3}\left | \begin{array}{cc} u_{2,1} & u_{2,2} \\ u_{3,1} & u_{3,2} \end{array} \right |
\\ &=
    \left (u_1 \left | \begin{array}{cc} u_{2,2} & u_{2,3} \\ u_{3,2} & u_{3,3} \end{array} \right | \right)_{,1}
  + \left ( u_1\left | \begin{array}{cc} u_{2,3} & u_{2,1} \\ u_{3,3} & u_{3,1} \end{array} \right | \right)_{,2}
  + \left ( u_1\left | \begin{array}{cc} u_{2,1} & u_{2,2} \\ u_{3,1} & u_{3,2} \end{array} \right | \right)_{,3},
\end{align*}
where $u_{i,j}$ denotes the $j$-th partial derivative of the $i$-th component of $\vec u$.
In higher dimensions, we may write $\det D\vec u = \Div ( (\adj D\vec u) \frac{\vec u}{n})$.

One of the main ideas in Ball's theory for nonlinear elasticity \cite{Ball77}
is that if the divergence is taken in the sense of distributions,
the right-hand side of the above expressions is well defined for maps
that are only weakly differentiable. This motivated his definition
of the \emph{distributional determinant} of a map $\vec u\in W^{1,n-1}(\Omega, \R^n) \cap L^\infty_{\loc}(\Omega, \R^n)$
as the distribution $\Det D\vec u \in \mathcal D'(\Omega)$ given by
\begin{equation} \label{eq:defDet}
\langle \Det D\vec u, \phi \rangle :=
  -\frac{1}{n} \int_\Omega \vec u(\vec x)
\cdot
  (\cof D\vec u(\vec x)) D\phi(\vec x) \dd\vec x, \qquad \phi \in C_c^\infty (\Omega)
\end{equation}
(see also \cite{Muller90CRAS,CoLiMeSe93,BrBeHe94,SaSe07,DeLeGh11,BrNg11} and references therein
for subsequent developments and for the role of $\Det D\vec u$
in compensated compactness, homogenization, liquid crystals, and superconductivity).
If a map $\vec u\in W^{1,p}(\Omega, \R^n)$, $p>n-1$, satisfies condition INV,
then $\vec u(\vec z)$ is contained in the region enclosed by $\vec u (\partial B(\vec x, r))$
 for every $\vec x \in \R^n$, a.e.~$\vec z \in \Omega \cap B(\vec x, r)$, and a.e.\ $r>0$ such that $\partial B(\vec x, r)\subset \Omega$. Consequently, $\vec u \in L^\infty_{\loc}(\Omega, \R^n)$, and the distributional determinant is well defined.

\begin{proposition}[cf.~\cite{MuSp95}, Lemma 8.1] \label{pr:Det}
 Let $\vec u \in W^{1,p}(\Omega, \R^n)$, $p>n-1$, satisfy $\det D\vec u > 0$ a.e.\ and condition INV. Then
\begin{enumerate}[i)]
 \item \label{cd:singular} $\Det D\vec u = (\det D\vec u)\mathcal L^n + \mu^s$,
where $\mu^s$ is singular with respect to $\mathcal L^n$
 \item \label{cd:Eimt} $|E(\vec x, r)\setminus \imT(\vec u, B(\vec x, r))|=0$
for every $\vec x\in \R^n$ and $r\in R_{\vec x}$
 \item $|E(\vec x, r_2)\setminus E(\vec x, r_1)| = \Det D\vec u (A_{r_1,r_2})$ for all $r_1\geq 0$ and $r_2>0$ such that the annulus $A_{r_1, r_2}:=\{\vec x\in \R^n: r_1<|\vec x|<r_2\}$ is contained in $\Omega$.
\end{enumerate}
\end{proposition}

\begin{proof}
 Let $\vec x\in \R^n$ and set $S:=\{r\in (0, \infty): \partial B(\vec x, r) \subset \Omega\}$. The map
\begin{equation*}
 \omega(r):= \frac{1}{n} \int_{\partial B(\vec x, r)} \vec u(\vec z) \cdot (\cof D\vec u(\vec z)) \vecg\nu(\vec z) \dd\mathcal H^{n-1}(\vec z), \quad r\in R_{\vec x}
\end{equation*}
belongs to $L^1(S)$. Suppose $[r_1, r_2]\subset S$ for some $r_1, r_2\in R_{\vec x}$. For $\delta>0$ let $\phi_\delta (\vec z):=\psi_\delta(|\vec z-\vec x|)$, where $\psi_\delta\in C_c^\infty ([0,\infty))$ is such that $\psi_\delta=1$ in $(r_1+\delta, r_2 - \delta)$, $\psi_\delta = 0$ in $[0, r_1]\cup [r_2,\infty)$, and $\delta\|\psi_\delta'\|_\infty \leq 2$. It is clear that $\phi_\delta \to \chi_{A_{r_1, r_2}}$ pointwise as $\delta \to 0^+$, and that
\begin{align*}
 \langle \Det D\vec u, \phi_\delta \rangle = \omega(r_2) - \omega(r_1) + \fint_{r_1}^{r_1+\delta} \delta \psi_{\delta}'(r) (\omega(r_1)-\omega(r))  + \fint_{r_2-\delta}^{r_2} \delta \psi_{\delta}'(r) (\omega(r_2)-\omega(r)).	
\end{align*}
The proof follows from \cite[Lemma 8.1]{MuSp95}, Proposition \ref{pr:cof}\ref{cd:volImt})--\ref{cd:monot}), and Lebesgue's differentiation theorem aplied to $\omega$.
\end{proof}

\section{Lower bounds} \label{se:LB}

The following is the basic estimate that allows us to relate the elastic energy to the volume and distortion
of the cavities. It extends \eqref{isoperbis} to an arbitrary exponent $p$ and
dimension $n$.

\begin{lemma}\label{le:Basic}
 Suppose that $\vec u \in W^{1,p}(\Omega, \R^n)$, $p>n-1$, satisfies $\det D\vec u>0$ a.e.\ and condition INV.
Then, for every $\vec x \in \Omega$ and $r\in R_{\vec x}$ (as defined in Proposition \ref{pr:cof}),
\begin{align*}
 \fint_{\partial B(\vec x, r)}
  \left |\frac{D\big (\vec u|_{\partial B(\vec x, r)}\big )(\vec x)}{\sqrt{n-1}}\right|^p\dd\mathcal H^{n-1}(\vec x)
\geq
  \left ( \frac{|E(\vec x, r)|}{|B(\vec x, r)|}\right )^\frac{p}{n}
  (1+CD\big (E(B(\vec x, r))\big )^{\frac{p}{n-1}}.
\end{align*}
Equality is attained only if $\vec u|_{\partial B(\vec x, r)}$ is radially symmetric.
\end{lemma}

\begin{proof}
Given $\vec x\in \R^n$,
$r>0$ and $\vec z\in \partial B(\vec x, r)$
such that $D\vec u(\vec z)$ is well defined,
we have that
\begin{align*}
 |(\cof D\vec u(\vec z))\vecg\nu(\vec z)| &= | (D\vec u(\vec z))\vec e_1 \wedge \cdots \wedge
(D\vec u(\vec z)) \vec e_{n-1}| \leq |(D\vec u)\vec e_1| \cdots |(D\vec u)\vec e_{n-1}| \\
& \leq (n-1)^{\frac{1-n}{2}}\left (|(D\vec u)\vec e_1|^2 + \cdots + |(D\vec u) \vec e_{n-1}|^2 \right )^\frac{n-1}{2},
\end{align*}
$\{\vec e_1, \ldots, \vec e_{n-1}, \vecg\nu(\vec z)\}$ being
an orthonormal basis of $\R^n$
with $\vecg \nu(\vec z):=(\vec z - \vec x)/r$.
Equality holds only if $|(D\vec u)\vec e_i|=|(D\vec u)\vec e_j|$ and $(D\vec u)\vec e_i \perp (D\vec u)\vec e_j$ for $i\ne j$,
as in Sivaloganathan-Spector \cite{SiSp10a,SiSp10b}.
If $r\in R_{\vec x}$,
by Propositions \ref{pr:cof}\ref{cd:perCof}), \ref{pr:Det}\ref{cd:Eimt}), and \ref{pr:iso}, we obtain
\begin{align*}
 \fint_{\partial B(\vec x, r)} \left | \frac{D\big (\vec u|_{\partial B(\vec x, r)}\big )}{\sqrt{n-1}}\right |^{n-1} \dd\mathcal H^{n-1}
& \geq \left ( \frac{ |E(\vec x, r)|}{\omega_nr^n	}\right)^\frac{n-1}{n} (1+CD\big (E(\vec x, r)\big )).
\end{align*}
The conclusion follows by Jensen's inequality.
\end{proof}

\subsection{Ball constructions, the case of multiple cavities}\label{ballconstruction}

In this Section we prove Proposition \ref{pro1} (our first
lower bound, valid for an arbitrary number of cavities).
We start by introducing the necessary notation, and by
recalling the ball construction method
in Ginzburg-Landau theory, following the presentation in \cite{SaSe07}.

Collections of balls will be denoted by expressions with $\mathcal B$.
If $B$ is a ball, $r(B)$ denotes its radius. If $\mathcal B$ is a collection of balls,
then $r(\mathcal B)=\sum_{B\in \mathcal B} r(B)$.
If $\lambda\geq 0$, $\lambda \mathcal B := \{\lambda B: B\in \mathcal B\}$.
We use $\bigcup \mathcal B$ to denote the union $\bigcup_{B\in \mathcal B} B$
of a collection of balls. Given a measurable set $A$ and a collection of balls $\mathcal B$,
we denote $\{B\cap A: B\in \mathcal B\}$ by $A\cap \mathcal B$.
Given $\mathcal F: \R^n\times (0,\infty) \to \R$, we regard $\mathcal F$ as a function defined on the set of
all balls (cf.\ \cite[Def.~4.1]{SaSe07}),
and write $\mathcal F(B)$ for $\mathcal F(\vec x, r)$ if $B=B(\vec x, r)$ (or $\overline B(\vec x, r)$).
Also, we write $\mathcal F(\mathcal B)$ for $\sum_{B\in \mathcal B} \mathcal F(B)$ if $\mathcal B$ is a collection of balls.

\begin{proposition}[cf.~\cite{SaSe07}, Th.~4.2] \label{pr:ballconstruction}
 Let $\mathcal B_0$ be a finite collection of disjoint closed balls
and let $t_0:=r(\mathcal B_0)$.
There exists a family $\{\mathcal B(t): t\geq t_0\}$ of collections of disjoint closed balls such that
$\mathcal B(t_0)=\mathcal B_0$ and
\begin{enumerate}[i)]
 \item For every $s\geq t\geq t_0$, $\bigcup \mathcal B(t) \subset \bigcup \mathcal B(s)$.
 \item \label{it:hmgn} There exists a finite set $T$ such that if $[t_1, t_2]\subset [t_0,\infty) \setminus T$,
then $\mathcal B(t_2) = \frac{t_2}{t_1} \mathcal B(t_1)$.
 \item \label{it:bcnst3} $r(\mathcal B(t))=t$ for every $t\geq t_0$.
\end{enumerate}
\end{proposition}

We point out that we chose a different parametrization from the one in \cite[Th.~4.2]{SaSe07}.
Here $t$ corresponds to $e^t$ there.

\begin{definition}[\cite{SaSe07}, Def.~4.1] \label{df:monotonic}
 We say that a function $\mathcal F: \R^n\times (0,\infty)\to \R$ is monotonic (when regarded as a function defined in the set of balls) if $\mathcal F(\vec x, r)$ is continuous with respect to $r$ and $\mathcal F(\mathcal B)\leq \mathcal F(\mathcal B')$ for any families of disjoint closed balls $\mathcal B, \mathcal B'$ such that $\bigcup \mathcal B \subset \bigcup \mathcal B'$.
\end{definition}

\begin{proposition}[cf.~\cite{SaSe07}, Prop.~4.1] \label{pr:bcEstimate}
 Let $\mathcal F: \R^n \times (0,\infty) \to \R$ be monotonic in the sense of Definition \ref{df:monotonic}.
Let $\mathcal B_0$ and $\{\mathcal B(t): t\geq t_0\}$ satisfy the conditions of Proposition \ref{pr:ballconstruction}.
Then,
\begin{equation}
 \mathcal F(\mathcal B(s)) - \mathcal F(\mathcal B_0) \geq \int_{t_0}^s
\sum_{B(\vec x, r) \in \mathcal B(t)} r\frac{\partial \mathcal F}{\partial r} (\vec x, r) \frac{\dd t}{t}
\end{equation}
for every $s\geq t_0$, and for every $B\in \mathcal B(s)$
\begin{equation}
 \mathcal F(B) - \mathcal F(\mathcal B_0\cap B) \geq \int_{t_0}^s \sum_{B(\vec x, r) \in \mathcal B(t)\cap B} r\frac{\partial \mathcal F}{\partial r} (\vec x, r) \frac{\dd t}{t}.
\end{equation}
\end{proposition}

Lemma \ref{le:Basic} applied to $\mathcal F(\vec x,r)= \int_{B(\vec x, r)} 
  \left ( \left |\frac{D\vec u(\vec x)}{\sqrt{n-1}}\right |^p-1\right )\dd\vec x$
    and Proposition \ref{pr:bcEstimate} immediately imply the following result
(stated without proof).
\begin{proposition} \label{pr:bc}
Suppose that $\vec u \in W^{1,p}(\Omega, \R^n)$ with $p>n-1$ satisfies $\det D\vec u>0$ a.e.\ and condition INV.
Suppose, further, that $\mathcal B_0$ and $\{\mathcal B(t): t\geq t_0\}$
satisfy the conditions of Proposition \ref{pr:ballconstruction}.
Then, for every $s>t_0$ such that $\Omega_s:=\bigcup \mathcal B(s)\setminus \bigcup \mathcal B_0 \subset \Omega$,
\begin{align*}
 \frac{1}{n}\int_{\Omega_s} \left ( \left |\frac{D\vec u(\vec x)}{\sqrt{n-1}}\right |^p-1\right )\dd\vec x
\geq \int_{t_0}^s \sum_{B\in \mathcal B(t)} |B|\left ( \frac{|E_B|^\frac{p}{n}}{|B|^\frac{p}{n}}(1+CD(E_B))^\frac{p}{n-1} -1\right ) \frac{\dd t}{t},
\end{align*}
where $E_B$ denotes $E(\vec x, r)$ for $B=\overline B(\vec x, r)$. Analogously, for every $B\in \mathcal B(s)$
\begin{align*}
 \frac{1}{n}\underset{B\setminus \bigcup \mathcal B_1}{\int} \left (\left |\frac{D\vec u(\vec x)}{\sqrt{n-1}}\right |^p
-1\right )\dd\vec x
\geq \int_{t_0}^s \sum_{B'\in \mathcal B(t)\cap B}
|B'|\left ( \frac{|E_{B'}|^\frac{p}{n}}{|B'|^\frac{p}{n}}(1+CD(E_{B'}))^\frac{p}{n-1} -1\right )
 \frac{\dd t}{t}.
\end{align*}
\end{proposition}

Proposition \ref{pro1} finally follows from Proposition \ref{pr:bc} and the
incompressibility constraint:

\begin{proof}[Proof of Proposition \ref{pro1}]
Let $A:=\{i: B(\vec a_i, R) \subset \Omega \}$,
$t_0:= r(\mathcal B_0)= \sum_{i\in A} \ep_i$, and
$\mathcal B_0:= \bigcup_{i\in A} \overline B_{\ep_i}(\vec a_i)$.
Let
$\{\mathcal B(t): t\geq t_0\}$ be the family obtained by applying Proposition \ref{pr:ballconstruction} to $\mathcal B_0$.
Then, applying Proposition \ref{pr:bc},
if $\bigcup \mathcal B(s) \subset \Omega$,
\begin{align} \label{eq:bcp=n}
\frac{1}{n}\int_{\Omega_\ep \cap \bigcup \mathcal B(s)} \left (\left |\frac{D\vec u(\vec x)}{\sqrt{n-1}}\right |^n-1\right ) \dd\vec x
\geq  & \int_{t_0}^s \sum_{B\in \mathcal B(t)} \left ( (|E_B|-|B|)
+ C |E_B|D(E_B)^\frac{n}{n-1}\right ) \frac{\dd t}{t}.
\end{align}
Proceeding as in the proof of Proposition \ref{pr:Det} and using incompressibility we obtain
\begin{align*}
 \left | E_B \setminus \bigcup_{\vec a_i \in B} E(\vec a_i, \ep_i) \right |
&=\Det D\vec u \left (B\setminus \bigcup_{\vec a_i \in B} {\overline B}_{\ep_i}(\vec a_i) \right )
= |B|-\sum_{\vec a_i\in B} \omega_n \ep_i^n,
\end{align*}
hence, by the definition of $v_i$ in the statement of the proposition,
\begin{align} \label{eq:Eb-B}
 |E_B|-|B| = \left | \bigcup_{\vec a_i \in B} E(\vec a_i, \ep_i) \right | -\sum_{\vec a_i\in B} \omega_n \ep_i^n
= \sum_{\vec a_i  \in B} v_i.
\end{align}
Combining \eqref{eq:bcp=n} and \eqref{eq:Eb-B} we obtain
\begin{align*}
\frac{1}{n}\int_{\Omega_\ep\cap \bigcup \mathcal B(s)} \left (\left |\frac{D\vec u(\vec x)}{\sqrt{n-1}}\right |^n-1\right ) \dd\vec x
\geq  \left (\sum_{i, B(\vec a_i, R)\subset \Omega_\ep} v_i \right ) \log \frac{s}{t_0}
+
 C \int_{t_0}^s \left (\sum_{B\in \mathcal B(t)} |E_B|D(E_B)^\frac{n}{n-1}\right ) \frac{\dd t}{t}.
\end{align*}
Let $s_0:= \sup \{s \in [t_0, R): \bigcup \mathcal B(s) \subset \Omega\}$.
If $s_0=R$, the claim is proved.
Otherwise,
from Proposition \ref{pr:ballconstruction} we deduce that there
exists a ball $B(\vec a, r)\in \mathcal B(s_0)$, of radius $r\leq s_0$,
containing at least one $\vec a_i$, $i\in A$,
  such that
$\overline B(\vec a, r) \cap \partial \Omega \ne \vacio$.
The proof is completed
by observing that
\begin{align*}
 R < \dist (\vec a_i, \partial \Omega)
\leq
|\vec a_i - \vec {a}|
+ \dist(\vec{a}, \partial \Omega)
< 2s_0.
\end{align*}
\end{proof}

\subsection{The case of two cavities: proof of Theorem \ref{th:LB}} \label{se:mainLB}
In this section, we prove Theorem \ref{th:LB} assuming Proposition \ref{pr:distortions}, whose proof is postponed to Section \ref{se:distortions}.

We will need the following lemma.

\begin{lemma}[Modulus of continuity of the distortion] \label{le:mod}
 Let $E, E'\subset \R^n$ be measurable. Then
\begin{enumerate}[i)]
 \item \label{cd:mod1} $||E|D(E)-|E'|D(E')|\leq 2|E\triangle E'|$
 \item \label{cd:mod2} $\left | |E|D(E)^\frac{n}{n-1} - |E'|D(E')^\frac{n}{n-1} \right| \leq 2^\frac{n}{n-1}\frac{n+1}{n-1}|E\triangle E'|$.
\end{enumerate}
\end{lemma}

\begin{proof}
Let $B'$ be a ball such that $|B'|=|E'|$ and $|E'|D(E')=|E'\triangle B'|$. For all measurable sets $B$
\begin{align*}
  |E\triangle B|-|E'|D(E') &= \|\chi_{E}-\chi_{B}\|_{L^1} - \|\chi_{E'}-\chi_{B'}\|_{L^1}
\leq
  \|\chi_{E}-\chi_{E'}\|_{L^1} + \|\chi_{B}-\chi_{B'}\|_{L^1}.
\end{align*}
Testing with concentric balls, and taking the minimum over all balls $B$ with $|B|=|E|$, yields
\begin{align*}
 |E|D(E)-|E'|D(E')  \leq \|\chi_{E}-\chi_{E'}\|_{L^1} + ||E|-|E'||
\end{align*}
($\|\chi_{B}-\chi_{B'}\|_{L^1}=||E|-|E'||$ since $B$ and $B'$ are concentric).
Combining this with the fact that $||E|-|E'||=|\|\chi_{E}\|_{L^1} - \|\chi_{E'}\|_{L^1}|\leq \|\chi_{E} - \chi_{E'}\|_{L^1}$, we obtain \ref{cd:mod1}).

Property \ref{cd:mod2}) follows from \ref{cd:mod1}), the mean value theorem,
and the fact that $D(E)\leq 2$ for all $E$ (a direct consequence of its definition).
To be more precise, suppose that $|E|>|E'|$, then
\begin{align*}
& \left | |E|D(E)^\frac{n}{n-1} - |E'|D(E')^\frac{n}{n-1} \right| \\
&\ =
 \left | |E|^{-\frac{1}{n-1}}(|E|D(E))^\frac{n}{n-1} - |E'|^{-\frac{1}{n-1}}(|E'|D(E'))^\frac{n}{n-1} \right| \\
&\ \leq
 |E|^{-\frac{1}{n-1}} \left | (|E|D(E))^\frac{n}{n-1} - (|E'|D(E'))^\frac{n}{n-1} \right|
   + (|E'|D(E'))^\frac{n}{n-1} \left | |E|^{-\frac{1}{n-1}} - |E'|^{-\frac{1}{n-1}} \right | \\
&\ \leq \frac{2n}{n-1}|E|^{-\frac{1}{n-1}}(\max\{|E|D(E), |E'|D(E')\})^\frac{1}{n-1}|E\triangle E'|
     +\frac{2^\frac{n}{n-1}}{n-1} ||E|-|E'||,
\end{align*}
completing the proof.
\end{proof}
We now proceed to the proof of  Theorem \ref{th:LB}.
As in \eqref{eq:Eb-B},
by Proposition \ref{pr:Det} we have that
 $|E(B)|=|B| + \sum_{i:\, \vec a_i\in B} v_i$
for all balls $B$ with $\partial B \subset \Omega_{\ep}$.
Hence, Lemma \ref{le:Basic} implies that
\begin{align} \label{eq:eq0proof}
  \frac{1}{n} \int_{\partial B(\vec x, r)} \left ( \left |\frac{D\vec u(\vec x)}{\sqrt{n-1}}\right |^n - 1 \right )
\dd \mathcal H^{n-1}(\vec x)
\geq
  \left ( \sum_{i:\, \vec a_i\in B(\vec x, r)} v_i + C|E(\vec x, r)| D\big ( E(\vec x, r) \big )^\frac{n}{n-1} \right ) \frac{1}{r}
\end{align}
for all $\vec x \in \R^n$ and all $r\in R_{\vec x}$.
Given $R>d$ such that $B(\vec a, R)\subset \Omega$, let
\begin{align*}
A_1:=B_{d/2}(\vec a_1) \setminus \overline{B_{\ep_1}(\vec a_1)},
\quad
A_2:= B_{d/2}(\vec a_2) \setminus \overline{B_{\ep_2}(\vec a_2)},
\quad
A_3:= B_{R}(\vec a) \setminus \overline{B_{d}(\vec a)}.
\end{align*}
By considering that
$\Omega_{\ep} \cap B(\vec a, R) \supset A_1\cup A_2\cup A_3$ and integrating successively in each annulus, 
 we obtain
\begin{align} \label{eq:eq1proof}
  \frac{1}{n} \int_{\Omega_{\ep}\cap B(\vec a, R)}
  & \left ( \left |\frac{D\vec u(\vec x)}{\sqrt{n-1}}\right |^n - 1 \right )\dd \vec x
\geq v_1 \log \frac{d}{2\ep_1} + v_2 \log \frac{d}{2\ep_2} + (v_1+v_2) \log \frac{R}{d}
\\ & \nonumber
  + C\int_{\ep_1}^{d/2} |E(\vec a_1, r)| D\big ( E(\vec a_1, r) \big )^\frac{n}{n-1}  \frac{\dd r}{r}
  + C\int_{\ep_2}^{d/2} |E(\vec a_2, r)| D\big ( E(\vec a_2, r) \big )^\frac{n}{n-1} \frac{\dd r}{r}
\\ & \nonumber
  + C\int_d^R |E(\vec a, r)| D\big ( E(\vec a, r) \big )^\frac{n}{n-1} \frac{\dd r}{r}.
\end{align}

Proposition \ref{pr:distortions}
applied to $E_1=E(\vec a_1, \frac{d}{2})$, $E_2=E(\vec a_2, \frac{d}{2})$,
and $E=E(\vec a, r)$, $r\in (d, R)$ gives
\begin{align}
& \nonumber |E(\vec a, r)| D\big ( E(\vec a, r) \big )^\frac{n}{n-1}
\\ & \geq  \label{eq:lbd}
    C(v_1+v_2)
  \left (
    \frac{(|E_1|^\frac{1}{n}+|E_2|^\frac{1}{n})^n-|E(\vec a, r)|}{(|E_1|^\frac{1}{n}+|E_2|^\frac{1}{n})^n - |E_1\cup E_2|}
  \right )^\frac{n(n+1)}{2(n-1)} \left ( \frac{\min\{|E_1|,|E_2|\}}{|E_1|+|E_2|} \right)^\frac{n}{n-1}
\\ & \quad \nonumber
  - |E(\vec a_1, d/2)|D\big (E(\vec a_1, d/2)\big )^\frac{n}{n-1}
  - |E(\vec a_2, d/2)|D\big (E(\vec a_2, d/2)\big )^\frac{n}{n-1}.
\end{align}

Define
$g(\beta_1, \beta_2):= (\beta_1^\frac{1}{n} + \beta_2^\frac{1}{n})^n - (\beta_1+\beta_2)$
(when $n=2$, $g(\beta_1,\beta_2)=2\sqrt{\beta_1 \beta_2}$).
Using that $|E_i|= v_i + \frac{\omega_n d^n}{2^n}$, $i=1,2$ we may write
\begin{align} \label{eq:E1E2g}
 (|E_1|^\frac{1}{n}+|E_2|^\frac{1}{n})^n = g(|E_1|, |E_2|) + (|E_1| + |E_2|) = g(|E_1|,|E_2|) + 2\cdot \frac{\omega_n d^n}{2^n} + v_1+v_2.
\end{align}
Estimate \eqref{eq:lbd} is meaningful if
$|E(\vec a, r)| \leq (|E_1|^\frac{1}{n}+|E_2|^\frac{1}{n})^n$, i.e.\ if
\begin{align} \label{eq:condr}
  \omega_n d^n \leq \omega_n r^n
\leq
  g\left (v_1+\frac{\omega_n d^n}{2^n}, v_2+ \frac{\omega_n d^n}{2^n}\right) + \frac{\omega_n d^n}{2^{n-1}}
\end{align}
(since $g$ is increasing in $\beta_1$ and $\beta_2$
and $g(\beta, \beta)=(2^n-2)\beta$, the inequality holds at least for $r=d$).
Define $\rho$ as the radius
for which $\omega_n r^n$
is in the middle of the two extremes in \eqref{eq:condr},
\begin{align} \label{eq:defrho}
 \omega_n \rho^n := \big (2^{n-1} + 1 \big)\frac{\omega_n d^n}{2^n}
+
  \frac{1}{2}g\left ( v_1+ \frac{\omega_n d^n}{2^n}, v_2+ \frac{\omega_n d^n}{2^n} \right).
\end{align}
For all $r\in (d, \min\{\rho, R\})$ we have that
$E(\vec a, r) \subset E(\vec a, \rho)$, hence
\begin{align} \label{eq:Earg}
 |E(\vec a, r)| < \omega_n \rho^n + v_1+v_2 = \frac{1}{2} g(|E_1|, |E_2|) + (2^{n-1}+1) \frac{\omega_n d^n}{2^n} + v_1+v_2.
\end{align}
Noticing that $g$ is $1$-homogeneous, combining \eqref{eq:E1E2g} and \eqref{eq:Earg} we obtain
\begin{align*}
  \frac{(|E_1|^\frac{1}{n}+|E_2|^\frac{1}{n})^n-|E(\vec a, r)|}{(|E_1|^\frac{1}{n}+|E_2|^\frac{1}{n})^n - |E_1\cup E_2|}
& \geq
\frac{\frac{1}{2}g(|E_1|,|E_2|) - (2^{n-1} + 1 - 2) \frac{\omega_n d^n}{2^n}}{g(|E_1|, |E_2|)}
 & =
  \frac{1}{2} - \frac{2^{n-1} - 1}{g \left ( \frac{2^n|E_1|}{\omega_n d^n}, \frac{2^n|E_2|}{\omega_n d^n} \right )}.
\end{align*}

Without loss of generality, assume that $\omega_n d^n < v_1+v_2$.
Estimate $g\left (\frac{2^n|E_1|}{\omega_n d^n}, \frac{2^n|E_2|}{\omega_n d^n}\right )$ by
\begin{align} \nonumber
  g(1+x, 1+y)
&=
    \sum_{k=1}^{n-1} {n\choose k} \left ((1+x)^k(1+y)^{n-k}\right)^\frac{1}{n}
\geq
  \sum_{k=1}^{n-1} {n\choose k} (1+kx)^\frac{1}{n}(1+(n-k)y)^\frac{1}{n}
\\ &\geq  \sum_{k=1}^{n-1} {n\choose k} (1+x)^\frac{1}{n}(1+y)^\frac{1}{n}
\geq  (2^n -2) (1+x+y)^\frac{1}{n}
\label{eq:estg}
\end{align}
(with $x=\frac{2^n |E_1|}{\omega_n d^n}-1= \frac{2^n v_1}{\omega_n d^n}$ and $y=\frac{2^n v_2}{\omega_n d^n})$ to obtain
\begin{align*}
  \left (
\frac{(|E_1|^\frac{1}{n}+|E_2|^\frac{1}{n})^n-|E(\vec a, r)|}{(|E_1|^\frac{1}{n}+|E_2|^\frac{1}{n})^n - |E_1\cup E_2|}
  \right )^\frac{n(n+1)}{2(n-1)}
& \geq
  \left (
    \frac{1}{2} - \frac{2^{n-1}-1}{(2^n -2) \left ( 1+2^n \frac{v_1+v_2}{\omega_n d^n} \right )^\frac{1}{n}}
  \right )^\frac{n(n+1)}{2(n-1)}
\geq
  4^{-\frac{n(n+1)}{2(n-1)}}.
\end{align*}
On the other hand, $|E_1\cup E_2| < 2(v_1+v_2)$
(because $\omega_n d^n < v_1+v_2$),
and since $|E_1|\geq v_1$ and $|E_2|\geq v_2$,
we can substitute
$\frac{\min\{|E_1|, |E_2|\}}{|E_1|+|E_2|} $
with
$\frac{\min \{v_1,v_2\}}{v_1+v_2}$
in \eqref{eq:lbd}.
Hence, for all $r\in (d, \min\{\rho, R\})$, all $s_1\in (\ep_1, d/2)$ and all $s_2\in (\ep_2, d/2)$,
\begin{align} \label{eq:estDist}
&  |E(\vec a, r)| D\big ( E(\vec a, r) \big )^\frac{n}{n-1}
  +
  |E(\vec a_1, s_1)| D\big ( E(\vec a_1, s_1) \big )^\frac{n}{n-1}
  +
  |E(\vec a_2, s_2)| D\big ( E(\vec a_1, s_1) \big )^\frac{n}{n-1}
\\ & \geq \nonumber
  C(n)(v_1+v_2) \left ( \frac{\min\{v_1,v_2\}}{v_1+v_2}\right )^\frac{n}{n-1}
  -
  \sum_{i=1}^2 \left | |E(\vec a_i, s_i)| D\big ( E(\vec a_i, s_i) \big )^\frac{n}{n-1}
    -
    |E(\vec a_i, {\textstyle \frac{d}{2}})| D\big ( E(\vec a_i, {\textstyle \frac{d}{2}}) \big )^\frac{n}{n-1} \right |.
\end{align}

Denoting $E(\vec a_1, s_1)$, $E(\vec a_2, s_2)$, and $E(\vec a, r)$
by $E_{s_1}$, $E_{s_2}$, and $E_r$, from \eqref{eq:eq1proof} we obtain
\begin{align}
 \label{eq:prevLB}
&  \frac{1}{n} \int_{\Omega_{\ep}\cap B(\vec a, R)}
   \left ( \left |\frac{D\vec u(\vec x)}{\sqrt{n-1}}\right |^n - 1 \right )\dd \vec x
\geq
  v_1 \log \frac{R}{2\ep_1} + v_2 \log \frac{R}{2\ep_2}
\\ \nonumber
&	  + C \inf_{\substack{r\in (d, \min\{\rho,R\})\\ s_i\in (\ep_i, d/2)}}
\left (
  |E_r| D(E_r )^\frac{n}{n-1}
  +
  |E_{s_1}| D( E_{s_1})^\frac{n}{n-1}
  +
  |E_{s_2}| D( E_{s_2})^\frac{n}{n-1}
  \right )
  \log \min \left \{ \frac{\rho}{d}, \frac{R}{d}, \frac{d}{\ep} \right \},
\end{align}
with $\ep=\max\{\ep_1, \ep_2\}$. In order to estimate $\log \frac{\rho}{d}$, from \eqref{eq:defrho} and \eqref{eq:estg} we find that
\begin{eqnarray*}
  \frac{\rho^n}{d^n}
& \geq
  2^{-(n+1)} g \left ( 1 + \frac{2^n v_1}{\omega_n d^n} , 1 + \frac{2^n v_2}{\omega_n d^n} \right )
 \geq
    (2^{-1} - 2^{-n} ) \left ( 1 + 2^n\frac{v_1+v_2}{\omega_n d^n} \right )^\frac{1}{n}
  \geq
      (1-2^{1-n}) \left ( \frac{v_1+v_2}{\omega_n d^n} \right )^\frac{1}{n}.
\end{eqnarray*}
The proof is completed by combining \eqref{eq:estDist} and \eqref{eq:prevLB} with Lemma \ref{le:mod}.

\subsection{Estimate on the distortions} \label{se:distortions}

This section is devoted to the proof of Proposition \ref{pr:distortions}.

\begin{lemma} \label{le:dist1} Let $q>1$ and suppose that $E$, $E_1$, and $E_2$ are sets of positive measure such that $E\supset E_1 \cup E_2$ and $E_1\cap E_2 =\varnothing$. Then
\begin{multline*}
\frac{|E|D(E)^q + |E_1|D(E_1)^q + |E_2|D(E_2)^q}{|E|+|E_1\cup E_2|}
\geq \min_{B, B_1, B_2} \left (\frac{\|\chi_B  - \chi_{B_1}  - \chi_{B_2}\|_{L^1}  - (|B|-|B_1|-|B_2|) }{ |E| + |E_1\cup E_2| } \right)^q,
\end{multline*}
where the minimum is taken over all balls $B$, $B_1$, $B_2$ with $|B|=|E|$, $|B_1|=|E_1|$, $|B_2|=|E_2|$.
\end{lemma}

\begin{proof}
 Let $B$, $B_1$, $B_2$ attain the minimum in the definition of $D(E)$, $D(E_1)$, $D(E_2)$, that is, suppose that $|B|=|E|$, $|B_1|=|E_1|$, $|B_2|=|E_2|$ and
\begin{equation*}
|E|D(E)=|E\triangle B|,\quad  |E_1|D(E_1) = |E_1\triangle B_1|, \quad  |E_2|D(E_2) = |E_2\triangle B_2|.
\end{equation*}
Since $\chi_B - \chi_{B_1} - \chi_{B_2} = (\chi_B-\chi_E) +
(\chi_E-\chi_{E_1}-\chi_{E_2}) + (\chi_{E_1}-\chi_{B_1}) + (\chi_{E_2}-\chi_{B_2})$,
then
\begin{align*}
 \|\chi_B  - \chi_{B_1}  - \chi_{B_2}\|_{L^1}  - \|\chi_E-\chi_{E_1}-\chi_{E_2}\|_{L^1}
 \leq |E|D(E) + |E_1|D(E_1) + |E_2|D(E_2).
\end{align*}
Also, note that $\|\chi_E-\chi_{E_1}-\chi_{E_2}\|_{L^1}=|E|-|E_1|-|E_2|=|B|-|B_1|-|B_2|$ because
$E_1\cap E_2=\varnothing$ and $E_1\cup E_2 \subset E$.
The result follows by Jensen's inequality applied to the map $t \mapsto t^q$.
\end{proof}

\begin{lemma} \label{le:twoways}
 Let $B, B_1, B_2$ be measurable subsets of $\R^n$. Then
\begin{align}
\|\chi_B  - \chi_{B_1}  - \chi_{B_2}\|_{L_1}
 - ( |B|-|B_1|-|B_2|)
\label{eq:max_int} & = 2(|B_1|+|B_2| - |B\cap (B_1\cup B_2)|) \\
\label{eq:min_excess} & = 2( |B_1\setminus B| + |B_2\setminus B| + |B \cap B_1\cap B_2|).
\end{align}
\end{lemma}

\begin{proof}
 Consider, first, the elementary relations
\begin{eqnarray}
 \label{eq:R1} |B_i\setminus B| &=& |B_i| - |B\cap B_i|, \quad i=1,2. \\
 \label{eq:R2} |B\cap(B_1\cup B_2)| &=& |B\cap B_1| + |B\cap B_2| - |B\cap B_1 \cap B_2|. \\
 \label{eq:R3} |B\setminus(B_1\cup B_2)| &=& |B| - |B\cap (B_1 \cup B_2)|.
\end{eqnarray}
From \eqref{eq:R1} and \eqref{eq:R2} we obtain
\begin{equation}
 \label{eq:R4} |B_1\setminus B| + |B_2\setminus B| + |B\cap B_1\cap B_2| = |B_1| + |B_2| - |B\cap (B_1\cup B_2)|.
\end{equation}
From \eqref{eq:R3} and \eqref{eq:R4} we obtain
\begin{equation}
 \label{eq:R5} |B\setminus (B_1\cup B_2)| = |B| - (|B_1|+|B_2|) + (|B_1\setminus B| + |B_2\setminus B| + |B\cap B_1\cap B_2|).
\end{equation}

Decomposing $\R^n$ as $\bigcup_{\alpha, \alpha_1, \alpha_2 \in \{0,1\}} \{\vec y:  (\chi_B,\chi_{B_1},\chi_{B_2})=(\alpha,\alpha_1,\alpha_2)\}$
we find that
\begin{align*}
  \|\chi_B - \chi_{B_1}  - \chi_{B_2}\|_{L_1} = & |B\cap B_1\cap B_2| + |B\setminus (B_1\cup B_2)| \\
& {} + 2|(B_1\cap B_2)\setminus B| + |(B_1\setminus B) \setminus B_2| + |(B_2\setminus B)\setminus B_1|.
\end{align*}
Since $|(B_1\cap B_2)\setminus B|$ can be seen either as $|(B_1\setminus B) \cap B_2|$ or as $|(B_2\setminus B)\cap B_1|$,
\begin{align*}
  \|\chi_B - \chi_{B_1}  - \chi_{B_2}\|_{L_1} & = |B\cap B_1\cap B_2| + |B\setminus (B_1\cup B_2)| + |B_1\setminus B| + |B_2\setminus B|.
\end{align*}
Using \eqref{eq:R4} and \eqref{eq:R3} we obtain \eqref{eq:max_int}; from \eqref{eq:R5} we obtain \eqref{eq:min_excess}.
\end{proof}

From \eqref{eq:max_int} we see that the minimization problem in the conclusion of Lemma \ref{le:dist1} is equivalent to
\begin{equation} \label{eq:max_int2}
 \max\{ |B\cap(B_1\cup B_2)|:\ B, B_1, B_2\ \hbox{balls of radii}\ R, R_1, R_2\},
\end{equation}
where $R$, $R_1$, $R_2$ are such that $|E|=\omega_n R^n$, $|E_1|=\omega_n R_1^n$, $|E_2|=\omega_n R_2^n$.

\begin{lemma} \label{le:auxprblm}
 Suppose $0<R_1,R_2<R<R_1+R_2$. Then \eqref{eq:max_int2} admits a solution, unique up to isometries of the plane,
 characterized by the facts that:
\begin{enumerate}[i)]
\item \label{cd:aligned} the centres of $B$, $B_1$, $B_2$ are aligned
\item \label{cd:nEmptyInt} $\varnothing \neq B_1\cap B_2 \subset B$,
$B_1 \not \subset B$, and $B_2 \not \subset B$
\item \label{cd:heights} $\partial B\cap \partial B_1$, $\partial B_1\cap \partial B_2$, and $\partial B_2\cap \partial B$
are (($n-2$)-dimensional) circles  having the same radius
 (or, if $n=2$, the common chords between $B$ and $B_1$, $B_1$ and $B_2$, and $B_2$ and $B$
all three have the same length, see Figure \ref{fg:circlesLB}a).
\end{enumerate}
In addition, the solution to \eqref{eq:max_int2} is such that
\begin{equation}
 \label{eq:optInt}
|B\cap B_1\cap B_2| \geq \frac{2^{n-1}}{n!}(R_1+R_2-R)^\frac{n+1}{2} \left (\frac{R_1R_2}{R_1+R_2}\right )^\frac{n-1}{2}.
\end{equation}
\end{lemma}

\begin{figure}
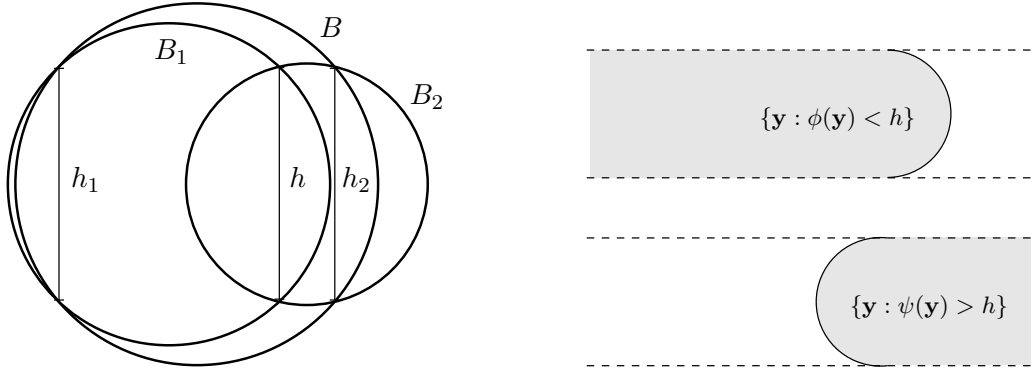

 \centering
  \subfigure{\input{optimal.pstex_t}}
\qquad \qquad \quad
  \subfigure{\input{levelsets.pstex_t}}
 \caption{\label{fg:circlesLB} On the left: optimal choice of $B$, $B_1$ and $B_2$
in \eqref{eq:optInt}, with $h=h_1=h_2$.
On the right: sublevel sets $\{\phi<h\}$ and $\{\psi>h\}$ in the proof of Lemma \ref{le:areadiff}
(as $h$ increases the level sets move along the slab $S$, in the direction of $\vec e$).}
\end{figure}

The proof of Lemma \ref{le:auxprblm} uses the auxiliary Lemmas \ref{le:phiandpsi} and \ref{le:areadiff}.
As mentioned in Section \ref{se:notation}, we write $\vec a \wedge \vec b$ to denote
the exterior product of $\vec a, \vec b \in  \R^n$.
In particular, we use that $|\vec a \wedge \vec b| = |\vec b| \dist (\vec a, \langle \vec b \rangle)$.
The purpose of Lemma \ref{le:phiandpsi} is to show that $B(\vec p + h\vec e, R)$ can
be written as the intersection of the two sets in Figure \ref{fg:circlesLB}b), for all $h\in \R$.
We then write the derivative of the area of the sublevel sets with respect to $h$ as a surface integral
on $\partial B(\vec p+h\vec e, R)$, using the coarea formula (Lemma \ref{le:areadiff}).

\begin{lemma} \label{le:phiandpsi}
Let $R>0$, $\vec p\in \R^n$, $\vec e\in \S^{n-1}$. Define
 \begin{eqnarray*}
  \phi(\vec y) &:=& (\vec y - \vec p)\cdot \vec e - \sqrt{R^2 - |(\vec y - \vec p) \wedge \vec e|^2} \\
  \psi (\vec y) &:=& (\vec y - \vec p)\cdot \vec e + \sqrt{R^2 - |(\vec y - \vec p) \wedge \vec e|^2}\,
 \end{eqnarray*}
in the infinite slab $S:=\{\vec y \in \R^n:\ |(\vec y - \vec p) \wedge \vec e|<R\}$.
Then, for all $h\in \R$,
\begin{equation*}
 B(\vec p+h\vec e, R)=\{\vec y \in S: \phi(\vec y) <h\} \cap \{\vec y \in S: \psi(\vec y)>h\}.
\end{equation*}
\end{lemma}

\begin{proof}
 By Pithagoras's theorem
$
| \vec y -(\vec p + h\vec e) |^2 = |(\vec y - \vec p)\cdot \vec e - h|^2 + |(\vec y -\vec p)\wedge \vec e|^2
$.
Then $\vec y \in B(\vec p +  h \vec e, R)$ if and only if $\vec y \in S$ and $|(\vec y -\vec p)\cdot \vec e-h| < \sqrt{R^2- |(\vec y - \vec p) \wedge \vec e|^2}$, that is, if and only if
\begin{align*}
& \vec y \in S,\quad (\vec y - \vec p) \cdot \vec e \geq h \quad \hbox{and}\quad \phi(\vec y) < h, \\
\hbox{or} \quad & \vec y \in S,\quad (\vec y - \vec p) \cdot \vec e \leq h \quad \hbox{and}\quad \psi(\vec y) > h.
\end{align*}
This proves that $B(\vec p + h\vec e, R) \subset \{ \phi<h\} \cap \{\psi>h\}$,
\begin{eqnarray*}
 &\{\phi<h\} \setminus B(\vec p + h \vec e, R) \subset \{ \vec y \in \R^n:\ (\vec y - \vec p)\cdot \vec e < h \} \\
\text{and} &
\{\psi>h\} \setminus B(\vec p + h\vec e, R) \subset \{ \vec y \in \R^n:\ (\vec y - \vec p) \cdot \vec e > h \}.
\end{eqnarray*}
From this we see that $\{\phi<h\} \cap \{\psi > h \} \subset B(\vec p + h \vec e, R)$, so the conclusion follows.
\end{proof}

\begin{lemma} \label{le:areadiff}
 Let $\vec p\in \R^n$, $R>0$, $E\subset \R^n$ measurable, and suppose that
\begin{equation} \label{eq:diffcondition}
\mathcal H^{n-1} (\partial B(\vec p, R) \cap \partial E)=0.
\end{equation}
Then the map $\vec y \mapsto |B(\vec y, R) \cap E|$ is differentiable at $\vec y=\vec p$ with gradient
\begin{equation*}
 \left . D_{\vec y} \big ( |B(\vec y, R) \cap E| \big ) \right |_{\vec y=\vec p} = \int_{\partial B(\vec p, R) \cap E} \frac{\vec z-\vec p}{R} \dd\mathcal H^{n-1}(\vec z) \,.
\end{equation*}
\end{lemma}

\begin{proof}
Given $\vec e\in \S^{n-1}$ arbitrary, let $\phi$, $\psi$, and $S$ be as in Lemma \ref{le:phiandpsi}. By definition of $\phi$ and $\psi$, we have that $\phi(\vec y) < (\vec y - \vec p) \cdot \vec e < \psi(\vec y)$ for all $\vec y \in S$, hence
\begin{equation*}
 (\vec y - \vec p)\cdot \vec e \leq h \ \Rightarrow\ \phi(\vec y) < h
\quad \hbox{and} \quad
 (\vec y - \vec p)\cdot \vec e \geq h \ \Rightarrow\ \psi(\vec y) > h
\end{equation*}
for all $h\in \R$.
 Thus, $\{\phi < h\} \cup \{\psi > h\} = S$
and is independent of $h$. From the elementary relation $|E \cap S_1\cap S_2| + |E\cap (S_1 \cup S_2)| = |E\cap S_1| + |E\cap S_2|$ we obtain (first for the case $|E\cap S|<\infty$, then for all measurable sets)
\begin{align*}
  &|E\cap B(\vec p + h\vec e, R)| - |E\cap B(\vec p, R)| \\
  &= (|E\cap \{\phi < h \}| + |E\cap\{\psi>h\}| - |E\cap S|) - (|E\cap \{\phi < 0 \}| + |E\cap\{\psi>0\}| - |E\cap S|) \\
&= |E\cap \{0\leq \phi < h\}| - |E\cap \{0< \psi \leq h\} |.
\end{align*}
Writing $\vec y\in S$ as $\vec p + \lambda \vec e + \mu \vec e'$, with $|\vec e'|=1$ and $\vec e \perp \vec e'$,
a direct computation shows that
\begin{align*}
 D\phi(\vec y) = \vec e - \frac{\mu \vec e'}{\sqrt{R^2 - \mu^2}}
\quad \text{and} \quad D\psi(\vec y )= \vec e + \frac{\mu \vec e'}{\sqrt{R^2 -\mu^2}}.
\end{align*}
Hence, by the coarea formula and Pithagoras's theorem,
\begin{align*}
 |E\cap  B(\vec p + h\vec e, R)| - |E\cap B(\vec p, R)|
&= \int_0^h \left ( \int_{\{\phi=\tau\}\cap E} \frac{\dd \mathcal H^{n-1}(\vec y)}{|D\phi(\vec y)|} - \int_{\{\psi=\tau\}\cap E} \frac{\dd \mathcal H^{n-1}(\vec y)}{|D\psi(\vec y)|} \right ) \dd\tau \\
&= \int_0^h \int_{\partial B(\vec p+\tau \vec e, R)\cap E} \sgn(\lambda - \tau) \frac{\sqrt{R^2 - \mu^2}}{R} \dd\mathcal H^{n-1}(\vec y) \dd \tau \\
&= \vec e \cdot \int_0^h \int_{\partial B(\vec p+\tau \vec e, R)\cap E} \frac{\vec y - \vec p-\tau \vec e}{R}\dd\mathcal H^{n-1}(\vec y) \dd\tau.
\end{align*}
Since $h$ and $\vec e$ are arbitrary, the above equation expresses that for all $\vec h \in \R^n$
\begin{equation*}
 |E\cap B(\vec p + \vec h, R)| - |E\cap B(\vec p, R)|
= \vec h \cdot \int_0^{1} \int_{\partial B(\vec p, R)} \frac{\vec z - \vec p}{R} \chi_{E-\tau \vec h} (\vec z)\dd\mathcal H^{n-1}(\vec z) \dd\tau.
\end{equation*}
Denoting $|\{\tau\in(0,1): \vec z + \tau \vec h \in E\}|$ by $\alpha(\vec z, \vec h, E)$, Fubini's theorem gives
\begin{multline*}
\left | |E\cap B(\vec p + \vec h, R)| - |E\cap B(\vec p, R)| - \vec h \cdot  \int_{\partial B(\vec p, R)\cap E} \frac{\vec z - \vec p}{R}\dd\mathcal H^{n-1}(\vec z) \right | \\
\leq |\vec h| \int_{\partial B(\vec p, R)} |\chi_E(\vec z) - \alpha(\vec z, h, E)|\dd\mathcal H^{n-1}(\vec z).
\end{multline*}
Due to the connexity of the line segment joining $\vec z$ and $\vec z+\vec h$,
if $\dist(\vec z, \partial E) \geq |\vec h|$ then either
$\vec z\in \Int E$ and $\alpha(\vec z,\vec h,E)=\chi_E(\vec z)=1$,
or $\vec z\in \R^n\setminus \overline E$ and $\alpha(\vec z,\vec h,E)=\chi_E(\vec z)=0$. Therefore,
\begin{multline*}
 \limsup_{\vec h \to 0} |\vec h|^{-1} \left | |E\cap B(\vec p + \vec h, R)| - |E\cap B(\vec p, R)| - \vec h \cdot  \int_{\partial B(\vec p, R)\cap E} \frac{\vec z - \vec p}{R}\dd\mathcal H^{n-1}(\vec z) \right | \\
 \leq \lim_{\vec h \to 0} \mathcal H^{n-1}(
\{\vec z\in \partial B(\vec p, R): \dist(\vec z, \partial E) < |\vec h| \})
= \mathcal H^{n-1}(\partial B(\vec p, R) \cap \partial E),
\end{multline*}
completing the proof.
\end{proof}

\begin{remark} \label{rk:diffcd}
 The example $\vec p=\vec 0$, $R=1$, $E=(-1,1)^n\setminus B(\vec 0,1)$ shows that $|B(\vec y, R) \cap E|$ is not always differentiable with respect to $\vec y $ if \eqref{eq:diffcondition} is not satisfied. However, this condition holds in the situations to be considered in the sequel, namely, when $E$ is a ball, the union of balls, or the intersection of balls of radii different from $R$.
\end{remark}

\begin{proof}[Proof of Lemma \ref{le:auxprblm}]
The existence of solutions to \eqref{eq:max_int2} can be easily deduced from the continuity of $|B\cap(B_1\cup B_2)|$ with respect to the centres of $B$, $B_1$, and $B_2$. Let $(B, B_1, B_2)$ be one such solution. We divide the proof of \ref{cd:aligned})-\ref{cd:heights}) in the following steps.

\emph{Step 1: one of the following possibilities occur
\begin{equation} \label{eq:B1intB2}
 \dist(B_1 \cap B_2, B)>0,\quad \dist(B_1\cap B_2, \R^n\setminus B) >0, \quad  \textit{or}\quad  B_1\cap B_2=\emptyset.
\end{equation}}
  Suppose, looking for a contradiction, that neither $\overline{B_1\cap B_2} \cap \overline B = \emptyset$ nor $\overline{B_1\cap B_2} \subset B$.
Then, by the connexity of $\overline{B_1\cap B_2}$, there exists $\vec x_0 \in \overline{B_1 \cap B_2} \cap \partial B$.
Let $B=B(\vec p, R)$, $\vec e:= \frac{\vec x_0 - \vec p}{|\vec x_0 - \vec p|}$,
and consider the following parametrization of $\partial B(\vec p, R)$ using spherical coordinates
\begin{equation*}
 \vec f(\theta, \vecg \xi):= \vec p + (R\cos \theta)\vec e + (R\sin \theta)\vecg \xi, \quad \theta \in  [0, \pi], \quad
\vecg \xi \in \mathbb S^{n-2}_{\vec e} := \mathbb S^{n-1} \cap \langle \vec e \rangle ^\perp.
\end{equation*}
Applying Lemma \ref{le:areadiff} to $E=\overline{B_1}\cup \overline{B_2}$ (see Remark \ref{rk:diffcd})
\begin{align*}
 \left .\frac{\dd}{\dd h} \left (|B(\vec p+h\vec e, R) \cap (B_1\cup B_2)| \right )\right |_{h=0}
&= \int_{\partial B \cap (B_1\cup B_2)} \vec e \cdot \frac{\vec z - \vec p}{R} \dd\mathcal H^1(\vec z) \\
 &=
R^{n-1} \int_{\mathbb S^{n-2}_\vec e} \int_{\theta\in (0, \pi):\, \vec f (\vec \theta, \vecg \xi) \in E} \cos \theta  (\sin \theta)^{n-2} \dd\theta \dd\mathcal H^{n-2}(\vecg \xi)
\end{align*}	
We can write the integral with respect to $\theta$ as
\begin{align*}
\int_0^{\pi/2} \cos \theta  (\sin \theta)^{n-2}
\big ( \chi_E(\vec f (\vec \theta, \vecg \xi)) - \chi_E(\vec f (\pi-\theta, \vecg \xi)) \big )\dd\theta.
\end{align*}
If we prove that
\begin{align} \label{eq:reflection}
 \vec f (\pi - \theta, \vecg \xi) \in \overline{B_1}\cup \overline {B_2}
\  \Rightarrow \
\vec f (\theta, \vecg \xi) \in \overline{B_1} \cup \overline{B_2}
\quad\text{for every $\theta \in [0, \pi/2]$}
\end{align}
and that
\begin{align} \label{eq:reflection2}
 \chi_E(\vec f (\vec \theta, \vecg \xi)) - \chi_E(\vec f (\pi-\theta, \vecg \xi)) = 1
\quad
\text{for all $(\theta, \vecg \xi)$ in a set of positive measure,}
\end{align}
we will obtain that $\frac{\dd}{\dd h} \left (|B(\vec p+h\vec e, R) \cap (B_1\cup B_2)| \right )>0$ at $h=0$.
The contradiction will follow
by noting that if $(B,B_1,B_2)$ solves \eqref{eq:max_int2},
 then $D_{\vec x}|B(\vec x, R) \cap (B_1 \cup B_2)|$ must be zero at $\vec x=\vec p$.

Suppose that $\vec f(\pi - \theta_0, \vecg \xi)\in \overline{B_i}$ for some $i=1,2$ and some $\theta_0\in [0, \frac{\pi}{2}]$.
Since $\overline{B_i}\cap \partial B$ is connected and contains $f(0, \vecg \xi)=\vec x_0$,
its projection to the plane $\vec p + \langle \vec e, \vecg \xi\rangle$ must contain
the whole of the arc $\vec f (\theta, \vecg  \xi)$, $\theta\in [0, \pi-\theta_0)$.
This proves \eqref{eq:reflection}.
In order to prove \eqref{eq:reflection2}, define $\theta_1(\vecg \xi):= \sup \{\theta \in [0, \pi]: f(\theta, \vecg \xi) \in \overline{B_1} \cup \overline{B_2}\}$.
Arguing as before, we see that
\begin{align} \label{eq:reflection3}
 |\{\theta\in [0, \pi]: \chi_E(\vec f (\vec \theta, \vecg \xi)) - \chi_E(\vec f (\pi-\theta, \vecg \xi)) = 1\}|>0
\end{align}
unless $\theta_1(\vecg \xi) =0$ or $\theta_1(\vecg \xi) =\pi$ (by continuity, if \eqref{eq:reflection3} holds for at least one $\vecg \xi \in \mathbb S^{n-2}_{\vec e}$,
then \eqref{eq:reflection2} follows).
Since $R_1, R_2 < R$, in fact $\theta_1=\pi$ is not possible (in that case $\vec x_0$ and $\vec x_0-2R\vec e$ would belong to some $\overline{B_i}$,
but $\diam \overline{B_i}=2R_i < 2R$).
It remains to rule out the possibility that $\theta_1(\vecg \xi)=0$ for all $\vecg \xi$, that is, that $\overline B \cap (\overline{B_1}\cup \overline{B_2})=\{\vec x_0\}$.
If that were the case then $B$ and $B_1$ would be tangent, so for all $h<R_1$ we would have that
\begin{align*}
 |B(\vec p + h\vec e, R) \cap (B_1\cup B_2)|\geq |B(\vec p + h\vec e, R) \cap B_1| >0 = |B\cap (B_1\cup B_2)|
\end{align*}
and $(B,B_1,B_2)$ would not be a solution to \eqref{eq:max_int2}. This completes the proof.

\emph{Step 2: the centres of $B$, $B_1$, $B_2$ lie on a same line.} In all the three cases considered in \eqref{eq:B1intB2}, $|B\cap B_1 \cap B_2|=|(B+\vec h) \cap B_1 \cap B_2|$ for every $\vec h$ sufficiently small.
Also, for given $R$, $R_1$, $R_2$, the expression $|B(\vec y_i, R_i) \cap B(\vec y, R)|$ is a decreasing function of $|\vec y - \vec y_i|$, $i=1,2$. If $\vec y$ were not in the line  containing $\vec y_1$ and $\vec y_2$, both $|\vec y-\vec y_1|$ and $|\vec y - \vec y_2|$ could be reduced by displacing  $\vec y$ towards that line. By \eqref{eq:R2}, this would increase $|B\cap (B_1\cup B_2)|$, contradicting the choice of $(B, B_1, B_2)$ as a solution to \eqref{eq:max_int2}.

\emph{Step 3: $(B, B_1, B_2)$ satisfies \ref{cd:nEmptyInt})-\ref{cd:heights}). Moreover, these conditions uniquely determine the distances and relative positions between the centres (that is, the solution to \eqref{eq:max_int2} is unique up to isometries).}

Let $h$, $h_1$, and $h_2$ denote, respectively, the radii of $\partial B_1\cap \partial B_2$, $\partial B\cap \partial B_1$, and $\partial B\cap \partial B_2$ (or the semi-lengths of the common chords between $B_1$ and $B_2$, $B$ and $B_1$, and $B$ and $B_2$ if $n=2$) defining these radii (or lengths) as zero in case of empty intersection. By virtue of \ref{cd:aligned}), both $\vec p_1-\vec p$ and $\vec p_2-\vec p$ are parallel to $\vec e:=\frac{\vec p_2-\vec p_1}{|\vec p_2-\vec p_1|}$, where $\vec p$, $\vec p_1$, $\vec p_2$ are the centres of $B$, $B_1$, $B_2$, respectively. Setting $q_i:= (\vec p_i- \vec p)\cdot \vec e$, $i=1,2$, and using Cartesian coordinates $(y_1,\ldots, y_n)$ with $\vec p$ as the origin and $\vec e$ in the direction of the $y_1$-axis, we have that $B=B\big ( (0,0, \ldots, 0), R \big )$, $B_1 = B( (q_1, 0, \ldots, 0), R_1 \big )$, $B_2 = B\big ( (q_2, 0, \ldots, 0), R_2 \big )$. By \eqref{eq:R2} and\footnote{There is exactly one situation not covered by Lemma \ref{le:areadiff}, namely when $R_1=R_2$ and $B_1=B_2 \ssubset B$, but it is easy to see that this does not give a maximum of $|B\cap(B_1\cup B_2)|$.} Lemma \ref{le:areadiff},
\begin{multline*}
 \frac{\partial}{\partial q_1} |B\cap (B_1 \cup B_2)|
= \frac{\partial}{\partial q_1} |B\cap B_1| - \frac{\partial}{\partial q_1} |(B\cap B_2) \cap B_1| \\
= \int_{\partial B_1 \cap B} \frac{z_1-q_1}{R_1} \dd\mathcal{H}^{n-1}(z_1,\ldots, z_n) - \int_{\partial B_1 \cap (B\cap B_2)} \frac{z_1-q_1}{R_1} \dd\mathcal{H}^{n-1}(z_1,\ldots, z_n).
\end{multline*}
In the first of the possibilities considered in \eqref{eq:B1intB2},
$B$ cannot intersect both $B_1$ and $B_2$, hence $(B, B_1, B_2)$ is not optimal (for example,
it would be better if B contained completely either $B_1$ or $B_2$).
 In the other two cases we have $\partial B_1 \cap (B\cap B_2) = \partial B_1 \cap B_2$.
Parametrize $\partial B_1$ by
\begin{equation*}
 \vec z\in \partial B_1\quad \Leftrightarrow \quad \vec z - \vec p_1 = (R_1\cos \theta )\vec e + (R_1\sin \theta) \vecg\xi,
 \quad \theta\in [0,\pi],\quad \vecg\xi \in \S^{n-2}_{\vec e}:= \S^{n-1}\cap \langle \vec e \rangle^\perp.
\end{equation*}
By definition of $\vec e$, $q_1 < q_2$. Therefore, $\vec z\in \partial B_1\cap B_2$ if and only if $\theta \in [0, \theta_2)$, where $\theta_2$
is one of the two angles in $[0, \pi]$
such that by $h=R_1\sin \theta_2$ (when $h=0$, we choose $\theta_2=0$ or $\theta_2=\pi$ according to whether $B_2\cap B_1=\emptyset$ or $B_2\subset B_1$). Thus,
\begin{equation*}
 \frac{\partial}{\partial q_1} |(B\cap B_2) \cap B_1| =\mathcal H^{n-2} (\S^{n-2}_{\vec e})\int_0^{\theta_2} R^{n-1}  \cos \theta (\sin \theta)^{n-2} \dd\theta = \omega_{n-1} h^{n-1}.
\end{equation*}
As for the integral on $\partial B_1\cap B$, the same argument shows that it equals $-(\sgn q_1)\omega_{n-1} h_1^{n-1}$. After obtaining the corresponding expression for $\frac{\partial}{\partial q_2} |B\cap B_2|$, and by virtue of the optimality of $(B, B_1, B_2)$, we obtain
\begin{equation*}
\sgn(q_1)h_1^{n-1} + h^{n-1} = h^{n-1} - \sgn(q_2)h_2^{n-1}=0.
\end{equation*}
The case $h=h_1=h_2=0$ is not optimal (due to the assumption $R<R_1+R_2$), hence $q_1<0<q_2$ and $h=h_1=h_2>0$. This proves \ref{cd:nEmptyInt})-\ref{cd:heights}).
It remains to show that $q_1$, $q_2$ and $h$ are uniquely determined by these conditions. Denoting the hyperplane containing the intersection of the boundaries of
two (intersecting) balls $B', B''$ by $\Pi(B', B'')$, we have that the hyperplanes $\Pi(B_1, B)$, $\Pi(B_1, B_2)$, and $\Pi(B_2, B)$ are given by $\{y_1=a_1\}$,
$\{y_1=a\}$, and $\{y_1=a_2\}$, for some $a_1$, $a$, $a_2\in \R$. Clearly, the following must be satisfied
\begin{align*}
 (a_1-q_1)^2+ h^2 &= R_1^2  & (a-q_1)^2 + h^2 & = R_1^2 & a_2^2 + h^2 &= R^2\\
 a_1^2 + h^2 & =R^2  & (a-q_2)^2 + h^2 & = R_2^2 & (a_2-q_2)^2 + h^2 &= R_2^2.
\end{align*}
In particular, $|a_1|=|a_2|=\sqrt{R^2-h^2}$, $|a_1-q_1|=|a-q_1|=\sqrt{R_1^2-h^2}$, and $|a-q_2|=|a_2-q_2|=\sqrt{R_2^2-h^2}$.
Conditions \ref{cd:nEmptyInt})-\ref{cd:heights}) imply
 that $a_1<q_1<a<q_2<a_2$ and $a_1<0<a_2$. Therefore
\begin{equation} \label{eq:q1q2}
 q_1 = \sqrt{R_1^2-h^2} - \sqrt{R^2-h^2}, \qquad q_2 = \sqrt{R^2-h^2} - \sqrt{R_2^2-h^2},
\end{equation}
which shows that $q_1$ and $q_2$ are determined by $h$. We also find that
\begin{equation} \label{eq:q1q2a}
 a-q_1=\sqrt{R_1^2-h^2}, \qquad q_2-a =\sqrt{R_2^2-h^2}.
\end{equation}
Adding the equations in \eqref{eq:q1q2a} and subtracting the equations in \eqref{eq:q1q2} yields (see Figure \ref{fg:distance})
\begin{equation}
\label{eq:distance} q_2-q_1=\sqrt{R^2-h^2} = \sqrt{R_1^2-h^2} + \sqrt{R_2^2-h^2}.
\end{equation}
\begin{figure}
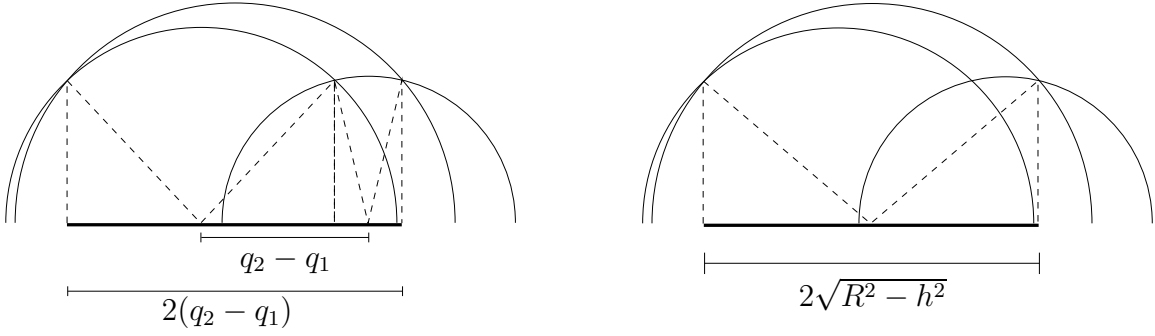

 \centering
 \subfigure{\input{optimal4.pstex_t}} \qquad \qquad
 \subfigure{\input{optimal2.pstex_t}}
 \caption{\label{fg:distance}Relationship between $h$ and the distance between the centres.}
\end{figure}
We may assume, without loss of generality, that $R_2<R_1$. Rewrite \eqref{eq:distance} as
\begin{equation*} \label{eq:h}
 \frac{R^2-R_1^2}{\sqrt{R^2-h^2} + \sqrt{R_1^2-h^2}} - \sqrt{R_2^2-h^2} =0.
\end{equation*}
The expression at the left-hand side is increasing in $h$, and equals $R-(R_1+R_2)<0$ at $h=0$, and $\frac{R^2-R_1^2}{\sqrt{R^2-R_2^2}+\sqrt{R_1^2-R_2^2}}>0$ at $h=R_2$. This shows that $h$ is uniquely determined by  $R, R_1, R_2$, and hence the balls $B_1, B_2$ too.

\emph{Step 4: proof of \eqref{eq:optInt}.} For each $k\in\{2,\ldots, n\}$ denote by $P_k$ the $k$-dimensional polyhedron with vertices (the convex hull of)
\begin{equation*}
 \{(q_2-R_2)\vec e, (q_1+R_1)\vec e\} \cup \{a\vec e \pm h \vec e_i: i=2,\ldots, k\}, \quad \vec e_i := \underbrace{(0, \ldots, 1, \ldots, 0)}_{i\text{-th position}}.
\end{equation*}
It is easy to see that $\mathcal H^2(P_2)=h\gamma$, where $\gamma:= |(q_1+R_1)-(q_2-R_2)|$, and that $\mathcal H^{k}(P_k)=2h\mathcal H^{k-1}(P_{k-1})/k$, for $k\in \{3,\ldots, n\}$. Thus, $|P_n| = (2^{n-1} h^{n-1}\gamma)/n!$.

From the previous analysis, we have that $B_1 \cap B_2$ contains $P_n$. From this we obtain \eqref{eq:optInt}, since, by virtue of  \eqref{eq:distance},
\begin{eqnarray}
 \label{eq:gamma1} &\displaystyle \gamma = R_1+R_2-\sqrt{R^2-h^2} > R_1+R_2-R,\\
 \text{and} \quad &\displaystyle \label{eq:gamma2} \gamma = \frac{h^2}{R_1+ \sqrt{R_1^2-h^2}} + \frac{h^2}{R_2+ \sqrt{R_2^2-h^2}} < \frac{(R_1+R_2) h^2}{R_1R_2}.
\end{eqnarray}
\end{proof}

We finally prove the main result.
\begin{proof}[Proof of Proposition \ref{pr:distortions}]
 We can assume that $|E_1|^\frac{1}{n} + |E_2|^\frac{1}{n} > |E|^\frac{1}{n}$ (otherwise the estimate is trivially true).
 By
\eqref{eq:min_excess} and \eqref{eq:optInt}
we have that
\begin{equation*}
\min ( \|\chi_B - \chi_{B_1} - \chi_{B_2}\|_{L^1} - (|B|-|B_1|-|B_2|)) \geq 
\frac{2^n}{n!}(R_1+R_2-R)^\frac{n+1}{2} \left(\frac{R_1R_2}{R_1+R_2}\right )^\frac{n-1}{2},
\end{equation*}
where the minimum is taken over all balls $B$, $B_1$, $B_2$ with $|B|=|E|$, $|B_1|=|E_1|$, $|B_2|=|E_2|$, and $R, R_1, R_2$
are such that $|E|=\omega_n R^n$, $|E_1|=\omega_n R_1^n$, $|E_2|=\omega_n R_2^n$. Thus, by Lemma \ref{le:dist1},
\begin{align*}
\frac{|E|D(E)^\frac{n}{n-1} + |E_1|D(E_1)^\frac{n}{n-1} + |E_2|D(E_2)^\frac{n}{n-1}}{|E|+|E_1\cup E_2|}
\geq C\frac{(R_1+R_2-R)^{\frac{n+1}{2}\frac{n}{n-1}}}{ (R^n + R_1^n + R_2^n)^\frac{n}{n-1} }  
\left(\frac{R_1R_2}{R_1+R_2}\right )^\frac{n}{2}
\end{align*}
The quantities $R^n + R_1^n + R_2^n$, $R_1^n + R_2^n$, and $(R_1 + R_2)^n$ are comparable,
since we are assuming that $R < R_1+ R_2$ and
by virtue of the identity $a^n + b^n \leq (a+b)^n\leq 2^{n-1}(a^n+b^n)$. Hence
\begin{align*}
(R^n + R_1^n + R_2^n)^\frac{n}{n-1} \leq C(R_1+R_2)^\frac{n^2}{n-1} = C (R_1+R_2)^\frac{n(n+1)}{2(n-1)} (R_1+R_2)^\frac{n}{2},
\end{align*}
which implies that
\begin{eqnarray}
\label{eq:step1Dist}
\frac{|E|D(E)^\frac{n}{n-1} + |E_1|D(E_1)^\frac{n}{n-1} + |E_2|D(E_2)^\frac{n}{n-1}}{|E|+|E_1\cup E_2|}
\geq C \left (\frac{R_1+R_2-R}{R_1 + R_2} \right )^\frac{n(n+1)}{2(n-1)}
\frac{R_1^\frac{n}{2} R_2^\frac{n}{2}}{(R_1+R_2)^n}.
\end{eqnarray}
By the mean value theorem, there exists $\xi$ between $R$ and $R_1+R_2$ such that
\begin{align*}
 R_1+R_2-R = \frac{(R_1+R_2)^n - R_1^n-R_2^n}{n \xi^{n-1}}
\left (\frac{(R_1+R_2)^n - R^n}{(R_1+R_2)^n - R_1^n-R_2^n}\right )
\end{align*}
Since we are assuming that $R < R_1 + R_2$, 
then $\xi \leq R_1+R_2$ and
\begin{align}
 \label{eq:step2Dist}
\frac{R_1+R_2-R}{R_1+R_2} \geq \frac{1}{n} 
\frac{(R_1+R_2)^n - R_1^n-R_2^n}{(R_1+R_2)^n}
\left (\frac{(|E_1|^\frac{1}{n} + |E_2|^\frac{1}{n})^n - |E|}{(|E_1|^\frac{1}{n} + |E_2|^\frac{1}{n})^n - |E_1\cup E_2|}\right ).
\end{align}
Suppose now that $|E_1| \geq |E_2|$, so that $\frac{R_1}{R_1+R_2}\geq \frac{1}{2}$.
By the binomial theorem,
\begin{align}
\label{eq:step3Dist}
\frac{(R_1+R_2)^n - R_1^n-R_2^n}{(R_1+R_2)^n}
= \sum_{k=1}^{n-1} {n\choose k}  
\left ( \frac{R_1}{R_1+R_2}\right )^{n-k}\left ( \frac{R_2}{R_1+R_2}\right )^{k}
\geq \frac{n}{2^{n-1}} \frac{R_2}{R_1+R_2}
\end{align}
(we have considered only the term corresponding to $k=1$).
Combining \eqref{eq:step2Dist} with \eqref{eq:step3Dist} we obtain
\begin{align*}
 \left ( \frac{R_1+R_2-R}{R_1+R_2} \right ) ^\frac{n(n+1)}{2(n-1)}
\geq C \left ( \frac{R_2}{R_1+R_2} \right ) ^\frac{n(n+1)}{2(n-1)} 
\left (\frac{(|E_1|^\frac{1}{n} + |E_2|^\frac{1}{n})^n - |E|}{(|E_1|^\frac{1}{n} + |E_2|^\frac{1}{n})^n - 
|E_1\cup E_2|}\right )^\frac{n(n+1)}{2(n-1)}.
\end{align*}
The conclusion follows from \eqref{eq:step1Dist} and the above equation, considering that $\frac{R_1}{R_1+R_2}\geq \frac{1}{2}$.
\end{proof}

\section{Upper bounds}

As explained in the Introduction, we obtain the upper bounds of Theorem \ref{th:UB} and Corollary \ref{co:Dirich}
by finding suitable test functions opening cavities of different shapes and sizes, the main difficulties being
to satisfy the
incompressibility constraint and the Dirichlet condition at the boundary. We split the problem into two:
in Section \ref{se:ub1} we define a family of incompressible, angle-preserving maps whose energy has the
right singular behaviour as $\ep\to 0$, with leading order $(v_1+v_2)|\log \varepsilon|$, and serves 
to define the test maps close to the singularities. In Section \ref{se:ub2} we extend those maps, using the existence
results of Rivi\`ere \& Ye \cite{RiYe96}, in order to match the boundary conditions.

\subsection{Proof of Theorem \ref{th:UB} } \label{se:ub1} \label{se:ubI}

In order to compute the energy of the test functions, we will need the following auxiliary lemmas, whose proof is postponed to Section \ref{se:auxi}.
\begin{lemma} \label{le:ub1}
 Let $\Omega$ be a domain in $\R^n$, star-shaped with respect to a point $\vec a \in \R^n$, with Lipschitz boundary
parametrized by $\vecg \zeta \mapsto \vec a + q(\vecg \zeta) \vecg \zeta$, $\vecg \zeta \in \S^{n-1}$.
Let $v\geq 0$ and define $\vec u : \R^n \setminus \{\vec a\} \to \R^n$ by
\begin{align} \label{eq:anglepreserving}
 \vec u ( \vec a + r\vecg \zeta ):=  \lambda \vec a + f(r, \vecg \zeta) \vecg \zeta,
\quad
f(r, \vecg \zeta)^n := r^n + (\lambda^n -1) q(\vecg \zeta)^n,
\quad
r\in (0, \infty), \quad \vecg \zeta \in \S^{n-1},
\end{align}
with $\lambda^n := 1 + \frac{v}{|\Omega|}$.
Then $\vec u$
 is a Lipschitz homeomorphism, $\det D\vec u\equiv 1$,
$\vec u(\vec x)=\lambda \vec x$ for all
$\vec x \in \partial \Omega$,
$\vec u(\overline \Omega \setminus \{\vec a\}) = \lambda\overline \Omega\setminus \imT(\vec u, \vec a)$,
$\vec u(\R^n \setminus \Omega)= \R^n \setminus \lambda \Omega$,
$|\imT(\vec u, \vec a)|= v$,  and for all $r$, $\vecg \zeta$,
\begin{align*}
r^{n-1} \left |\frac{D\vec u(\vec a + r\vecg\zeta)}{\sqrt{n-1}}\right |^n \leq
 C\left ( r+ |v|^\frac{1}{n}\frac{\max\{q,|Dq|\}}{|\Omega|^\frac{1}{n}} \right )^{n-1}
+ \left ( \frac{q(\vecg\zeta)^n}{|\Omega|} + C\frac{\max\{q, |Dq|\}^{n-1}|Dq|}{|\Omega|}
\right ) \frac{v}{r},
\end{align*}
$C$ being a constant depending only on $n$.
\end{lemma}

\begin{lemma} \label{le:ub2}
Suppose that $\vec{\tilde a}\in \R^n$, $0\leq d\leq \rho$, and $\vec a= \vec{\tilde a} + d\vec e$
for some $\vec e \in \S^{n-1}$.
Let
$\vecg\zeta \mapsto \vec a + q(\vecg \zeta) \vecg \zeta$, $\vecg \zeta \in \S^{n-1}$
be the polar parametrization of $\partial B(\vec{\tilde a}, \rho)$
taking $\vec a$ as the origin. Then
\begin{enumerate}[i)]
 \item \label{it:b1} for all $\vecg \zeta \in \S^{n-1}$,
$|q(\vecg\zeta)|\leq 2\rho$, $|Dq(\vecg \zeta)|\leq 2d|\vecg\zeta \wedge \vec e|$, and
$\displaystyle |Dq(\vecg\zeta)|
\leq 2d \left |\frac{q(\vecg \zeta)}{\sqrt{\rho (\rho - d)}}\right |^2 |\vecg\zeta \wedge \vec e|$
  \item \label{it:b2}
if $\vecg\zeta \cdot (\vec a - \vec{\tilde a}) <0$ then
$q(\vecg\zeta) \geq \rho \left |\vecg \zeta \cdot \vec e\right |$ and $\displaystyle 1\leq \frac{q(\vecg \zeta)}{d|\vecg \zeta \cdot \vec e|
+ \sqrt{\rho(\rho-d)}}\leq 2$
  \item \label{it:b3}
if $\vecg \zeta \cdot (\vec a -\vec {\tilde a}) >0$ then
$\displaystyle
 \frac{q(\vecg \zeta)}{\sqrt{\rho(\rho - d)}}
\leq \frac{2\sqrt{2}}{1+ \frac{d\vecg \zeta \cdot \vec e}{\sqrt{\rho(\rho - d)}}}$.
\end{enumerate}
\end{lemma}

\begin{lemma}\label{le:areal}
 Let $0\leq d\leq \rho$, $\vec{\tilde a} \in \R^n$, $\vec e \in \S^{n-1}$, and $\Omega:=\{\vec x \in B(\vec{\tilde a}, \rho):
(\vec x -\vec{\tilde a}) \cdot \vec e > \rho-2d\}$. Then
\begin{align*}
 n|\Omega| > \omega_{n-1} d^\frac{n+1}{2} (2\rho-d)^\frac{n-1}{2}.
\end{align*}
\end{lemma}

\begin{proof}[Proof of Theorem \ref{th:ub1}]

{\it - Step 1:}  Construction of the domain.\\
Let $\vec a_1$, $\vec a_2\in\R^n$ and $d:=|\vec a_2 - \vec a_1|>0$, as in the statement of the theorem.
Call $\vec e:= \frac{\vec a_2 - \vec a_1}{|\vec a_2-\vec a_1|}$.
Given $d_1$, $d_2$, $\rho_1$, and $\rho_2$ such that $0< d_1 \leq \rho_1\leq d$, 
$0< d_2 \leq \rho_2\leq d$, and $d_1+d_2=d$, define
\begin{eqnarray}
 & \displaystyle \nonumber
\vec {\tilde a}_1:= \vec a_1 + (\rho_1-d_1)\vec e, \quad
\vec {\tilde a}_2 := \vec a_2 - (\rho_2-d_2)\vec e, \quad
B_1:= B(\vec {\tilde a}_1, \rho_1), \quad
B_2:= B(\vec {\tilde a}_2, \rho_2)
\end{eqnarray}
($\vec {\tilde a}_1$, $\vec {\tilde a}_2$ are chosen such that $B_1\cup B_2$
fits in an infinite slab of width $2d$, as in Figure \ref{fig:upbnd}).
As stated in the Introduction, 
our aim is to show that for every $\delta \in [0,1]$ there are unique $d_1$, $d_2$, $\rho_1$, and $\rho_2$
such that 
the ratio between the width of $B_1\cap B_2$ and that of $B_1\cup B_2$ is exactly $\delta$
(i.e., $\delta := \frac{\rho_1+\rho_2-d}{d}$),
and such that 
$\frac{|\Omega_2|}{|\Omega_1|} = \frac{v_2}{v_1}$, with
\begin{eqnarray}
\label{eq:Om1Om2} & \displaystyle
 \Omega_1:=\{\vec x \in B_1: (\vec x -\vec a_1)\cdot \vec e <d_1\},
\qquad
\Omega_2 := \{\vec x \in B_2: (\vec x - \vec a_2) \cdot \vec e > -d_2\}.
\end{eqnarray}
To this end, we will first consider a simplified but equivalent problem. 
Fix $d>0$ and $\vec e\in \S^{n-1}$, and let $S:= \{ \vec x \in \R^n: |\vec x \cdot \vec e |< d \}$.
Given $\rho_1$ and $\rho_2$ in $(0,d)$ define
\begin{align} \label{eq:B1B2}
 B_1 = B\big ( (-d+\rho_1) \vec e , \rho_1 \big )
\quad \text{and} \quad 
B_2 = B\big ( (d-\rho_2)\vec e , \rho_2)
\end{align}
(the balls of radii $\rho_1$, $\rho_2$ contained in $S$ and tangent to $\partial S$ from the right and from the left).
If the balls intersect, let $\hat a \in (-d,d)$ be such that $\vec x \cdot \vec e = \hat a$ 
for $\vec x \in \overline B_1 \cap \overline B_2$ and define 
\begin{align} \label{eq:defDelta}
  \Omega_1 := \{\vec x \in B_1: \vec x \cdot \vec e < \hat a\}, \qquad 
  \Omega_2 := \{\vec x \in B_2: \vec x \cdot \vec e > \hat a\}, \qquad
  \rho_{\min}:= \frac{v_1^\frac{1}{n}d}{v_1^\frac{1}{n}+v_2^\frac{1}{n}}.
\end{align}
We want to show that
\begin{enumerate}[i)]
  \item \label{it:CD1i} if
$\frac{|\Omega_2|}{|\Omega_1|} = \frac{v_2}{v_1}$
then $\rho_1 \geq \rho_{\min}$
  \item \label{it:CD1ii} for every $\rho_1 \in [\rho_{\min}, d)$ there exists a unique $\rho_2 \in [0, d]$
such that $\overline B_1 \cap \overline B_2 \ne \vacio$ and $\frac{|\Omega_2|}{|\Omega_1|} = \frac{v_2}{v_1}$
  \item \label{it:CD1iii} $\rho_2=\rho_2(\rho_1)$ is such that $\rho_2 \leq \rho_1$ and such that the ratio $\frac{\rho_1+\rho_2-d}{d}$
increases from $0$ to $1$ as $\rho_1$ increases from $\rho_{\min}$ to $d$.
\end{enumerate}
This will imply that for every $\delta \in [0,1]$ there are unique $\rho_1$ and $\rho_2$ such that 
$\frac{\rho_1+\rho_2-d}{d}= \delta$ and $\frac{|\Omega_2|}{|\Omega_1|} = \frac{v_2}{v_1}$.
Let $d_1:= \frac{\hat a+ d}{2}$ and $d_2:= \frac{d-\hat a}{2}$, with $\hat a$ as in 
\eqref{eq:defDelta}
(they are the 
semi-distances from the plane containing $\overline B_1\cap \overline B_2$
to the walls of the slab $S$ containing $\overline B_1\cup \overline B_2$).
Based on the previous reasoning, it can be seen 
that these values of $d_1$, $d_2$, $\rho_1$, and $\rho_2$
constitute a solution to the original problem, and that no other choice is possible.

In order to prove \ref{it:CD1i}), define
$ B_1':= \big ( (-d+\rho_{\min}) \vec e , \rho_{\min} \big )$,
$ B_2':= \big ( \rho_{\min} \vec e, d-\rho_{\min} \big )$
($\rho_{min}$ is such that $B_1'$ and $B_2'$ are tangent
and $\frac{|B_2'|}{|B_1'|} = \frac{v_2}{v_1}$).
If $0<\rho_1<\rho_{\min}$ and $\overline B_1\cap \overline B_2 \ne \vacio$,
then $\frac{|\Omega_2|}{|\Omega_1|} > \frac{v_2}{v_1}$, since
$\Omega_1 \subset B_1'$ and $\Omega_2 \supset B_2'$.
Hence $\frac{|\Omega_2|}{|\Omega_1|}=\frac{v_2}{v_1}\ \Rightarrow\ \rho_1 \geq \rho_{\min}$, as claimed.

Fix $\rho_1 \in [\rho_{\min}, d)$. In order for $B_2$ to intersect $B_1$ we must have that $\rho_2 \geq d-\rho_1$. 
When $\rho_2= d-\rho_1$, $\Omega_1$ and $\Omega_2$ are tangent balls 
with $\frac{|\Omega_2|}{|\Omega_1|}=\frac{(d-\rho_1)^n}{\rho_1^n} \leq \frac{(d-\rho_{\min})^n}{\rho_{\min}^n}=\frac{v_2}{v_1}\leq 1$.
It is clear that $|\Omega_1|$ decreases and $|\Omega_2|$ increases as $\rho_2$ increases
(the intersection plane moves to the left), therefore $\frac{|\Omega_2|}{|\Omega_1|}$ is increasing
in $\rho_2$. When $\rho_2=\rho_1$, the ratio is $1$. This proves \ref{it:CD1ii})
and the first part of \ref{it:CD1iii}).
A similar argument shows that $\frac{\rho_1+\rho_2-d}{d}$
is increasing in $\rho_1$ (it follows 
from the fact that if we fix $\rho_2$ and increase $\rho_1$ then
the intersection plane moves to the right
and $\frac{|\Omega_2|}{|\Omega_1|}$ decreases).

It is clear that if $\rho_1=\rho_{\min}$ then $\rho_2=d-\rho_{\min}$ and $\frac{\rho_1+\rho_2-d}{d}=0$. 
It only remains to prove that as $\rho_1\to d$ also $\rho_2 \to d$.
By \eqref{eq:B1B2}, $|B_2\setminus B_1| \leq |B(\vec 0, d) \Delta B_1| \to 0$ as $\rho_1\to d$, hence
\begin{align*}
\lim_{\rho_1\to d}
\frac{|\Omega_1|}{|B_1|}
= \lim_{\rho_1\to d} \frac{|\Omega_1|}{|\Omega_1 \cup \Omega_2|} \left ( 1 + 
\frac{|(\Omega_1\cup \Omega_2)\setminus B_1|}{|B_1|} \right )= \frac{v_1}{v_1+v_2} \left ( 1 +
\frac{\lim_{\rho_1 \to d} |B_2\setminus B_1|}{\omega_n d^n} \right ) = \frac{v_1}{v_1+v_2}.
\end{align*}
For $\rho_1 < d$, $\partial B_1\cap \partial B_2$ is of the form 
$A(\rho_1):=\{\hat a(\rho_1)\vec e + \sqrt{\rho_1^2- \hat a (\rho_1)^2}\vec e': \vec e'\in \S^{n-1}, \vec e' \perp \vec e\}$.
Since $\hat a(\rho_1)$ is determined by $\frac{|\Omega_1|}{|B_1|}$, it has a well-defined limit as $\rho_1 \to d$.
The sphere $\partial B_2$ can be characterized as the one containing $A(\rho_1)$ and the point $d\vec e$.
Thus, in the limit, $\partial B_2$ will be the sphere containing $d\vec e$ and $A(d)$,
which is none other than $\partial B(\vec 0, d)$. In particular, $\rho_2 \to d$, as desired.
\\

{\it - Step 2:} Definition of the map.\\
We define $\vec u: \R^n \setminus \{\vec a_1, \vec a_2\}$ piecewise, based on Lemma \ref{le:ub1},
in the following manner. Inside $\Omega_1$ we apply Lemma \ref{le:ub1} to $\Omega=\Omega_1$ and $\vec a=\vec a_1$;
inside $\Omega_2$ we apply Lemma \ref{le:ub1} to $\Omega=\Omega_2$ and $\vec a=\vec a_2$. Finally, in order to define $\vec u$
in
$\R^n \setminus \Omega_1\cup \Omega_2$ we define
\begin{align*}
 \vec a^*= \frac{(\vec{\tilde a}_1+\rho_1 \vec e ) + (\vec{\tilde a}_2 - \rho_2 \vec e )}{2}
= \vec{\tilde a}_1 + (d-\rho_2) \vec e = \vec{\tilde a}_2 - (d-\rho_1) \vec e
\end{align*}
(when $\delta=0$, $\vec a^*$ is the intersection point; when $\delta=1$, $\vec a^*$ is the center of the ball)
and use Lemma \ref{le:ub1} with $\Omega=\Omega_1 \cup \Omega_2$,
$\vec a=\vec a^*$.
Let $\vecg \zeta \mapsto \vec a_1 + q_1(\vecg \zeta) \vecg \zeta$,
$\vecg \zeta \mapsto \vec a_2 + q_2(\vecg \zeta) \vecg \zeta$, and
$\vecg \zeta \mapsto \vec a^* + q(\vecg \zeta) \vecg \zeta$ be, respectively, the polar parametrizations of
$\partial \Omega_1$, $\partial \Omega_2$, and $\partial (\overline \Omega_1 \cup \overline \Omega_2)$ (with $\vecg \zeta \in \S^{n-1}$
in all cases).
To be precise,
\begin{align*}
 \vec u (\vec x) :=
\begin{cases}
 \lambda \vec a_1 + \left ( |\vec x - \vec a_1|^n + \frac{v_1}{|\Omega_1|} q_1\left (
\frac{\vec x - \vec a_1}{|\vec x - \vec a_1|} \right ) ^n \right )^\frac{1}{n}
\frac{\vec x - \vec a_1}{|\vec x - \vec a_1|}
&
\vec x \in \overline \Omega_1 \setminus \{\vec a_1\} \\
\lambda \vec a_2 + \left ( |\vec x - \vec a_2|^n + \frac{v_2}{|\Omega_2|} q_2\left (
\frac{\vec x - \vec a_2}{|\vec x - \vec a_2|} \right ) ^n \right )^\frac{1}{n}
\frac{\vec x - \vec a_2}{|\vec x - \vec a_2|}
&
\vec x \in \overline \Omega_2 \setminus \{\vec a_2\} \\
\lambda \vec a^* + \left ( |\vec x - \vec a^*|^n + \frac{v_1+v_2}{|\Omega_1+\Omega_2|} q\left (
\frac{\vec x - \vec a^*}{|\vec x - \vec a^*|} \right ) ^n \right )^\frac{1}{n}
\frac{\vec x - \vec a^*}{|\vec x - \vec a^*|}
&
\vec x \in \R^n \setminus \overline{\Omega_1 \cup \Omega_2},
\end{cases}
\end{align*}
with
\begin{align*}
\lambda^n -1:=\frac{v_1}{|\Omega_1|} = \frac{v_2}{|\Omega_2|} = \frac{v_1+v_2}{|\Omega_1\cup \Omega_2|}.
\end{align*}
Since $\frac{|\Omega_1|}{|\Omega_2|}= \frac{v_1}{v_2}$, the construction is well defined and $\vec u(\vec x)= \lambda \vec x$
for all $\vec x \in \partial \Omega_1 \cup \partial \Omega_2$.
The resulting map is an incompressible homeomorphism, creates cavities at the desired locations with the desired volumes
and is smooth except across $\partial \Omega_1 \cup \partial \Omega_2$ (where it is still continuous). 
It only remains to estimate its elastic energy.
\\

{\it - Step 3 :} Evaluation of the energy in $\R^n \setminus  (\Omega_1 \cup \Omega_2)$.\\
By Lemma \ref{le:ub2}\ref{it:b1}), $\max\{q, |Dq|\}\leq 2d$,
then, by Lemma \ref{le:ub1}
\begin{align} \label{eq:Est3Step1}
 r^{n-1} \left | \frac{D\vec u (r\vecg \zeta)}{\sqrt{n-1}} \right |^n
&\leq C \left ( r + \frac{d(v_1+v_2)^\frac{1}{n}}{|\Omega_1 \cup \Omega_2|^\frac{1}{n}}\right )^{n-1}
+ \left ( \frac{q^n}{|\Omega_1 \cup \Omega_2|} + \frac{Cd^{n-1}|Dq|}{|\Omega_1 \cup \Omega_2|} \right ) \frac{v_1+v_2}{r}.
\end{align}
Since $\rho_i$, $i=1,2$ increases with $\delta$ and assumes the value
$\frac{v_i^\frac{1}{n} d}{v_1^\frac{1}{n} + v_2^\frac{1}{n}}$ when $\delta=0$, it follows that
\begin{align} \label{eq:bdOm}
 2\omega_n d^n > \omega_n (\rho_1^n +\rho_2^n ) > |\Omega_1 \cup \Omega_2| 
> \frac{1}{2}\omega_n (\rho_1^n+\rho_2^n) > 2^{-n} \omega_n d^n
\end{align}
(since $\Omega_1\cup \Omega_2 \supset B_i$ for each $i=1,2$).
Consequently, for any $R>0$ (using that $\mathcal H^{n-1}(\S^{n-1})=n\omega_n$)
\begin{align*}
  &
\frac{1}{n} \int_{B(\vec a^*, R) \setminus \overline{\Omega_1\cup \Omega_2}} 
\left | \frac{D\vec u (r\vecg \zeta)}{\sqrt{n-1}} \right |^n
\dd\vec x 
= \frac{1}{n} \int_{\S^{n-1}}\!\! \int_{q(\vecg \zeta)}^{\max\{q, R\}} 
r^{n-1} \left | \frac{D\vec u (r\vecg \zeta)}{\sqrt{n-1}} \right |^n \dd r \dd \mathcal H^{n-1}(\vecg \zeta) \\
&\leq \fint_{\S^{n-1}}\!\!
\left [\frac{2^{n-1}C}{n} \left ( \omega_n R^n + 2^n(v_1+v_2)
\right ) 
  + (v_1+v_2)\left ( \frac{\omega_n q^n}{|\Omega_1\cup \Omega_2|} + 2^nC\frac{|Dq|}{d}\right )
 \left ( \log \frac{R}{q} \right )_+
\right ] \dd\mathcal H^{n-1},
\end{align*}
where $(\log x)_+:=\max \{0, \log x\}$.
Note that $\left ( \log \frac{R}{q} \right )_+ \leq \left ( \log \frac{R}{d} \right )_+ + \left ( \log \frac{d}{q} \right )_+$. 
Also,
\begin{align} \label{eq:areaQ}
 |\Omega_1\cup \Omega_2| = \int_{\S^{n-1}}\!\! \int_0^{q(\vecg \zeta)} r^{n-1} \dd r \dd\mathcal H^{n-1}(\vecg \zeta) 
= \fint_{\S^{n-1}} \omega_n q(\vecg \zeta)^n \dd\mathcal H^{n-1}(\vecg \zeta).
\end{align}
Finally, 
since $|\vec a^* - \vec{\tilde a}_1| + |\vec a^*-\vec{\tilde a}_2| = d(1-\delta)$,
Lemma \ref{le:ub2}\ref{it:b1})
implies that $|Dq|\leq 2d(1-\delta)$. Hence,
\begin{align*}
\frac{1}{n} \int_{B(\vec a^*, R) \setminus \overline{\Omega_1\cup \Omega_2}} \left | \frac{D\vec u}{\sqrt{n-1}} \right |^n \dd\vec x
\leq & C(v_1+v_2+ \omega_n R^n)
+(v_1+v_2)\Big ( 1 + C(1-\delta) \Big ) \left ( \log \frac{R}{d} \right )_+ \\
& \quad + C (v_1+v_2)  \fint_{\S^{n-1}} \left ( \frac{q^n}{d^n} + \frac{|Dq|}{d} \right )
\left (\log  \frac{d}{q(\vecg\zeta)} \right )_+ \dd\mathcal H^{n-1}(\vecg \zeta).
\end{align*}

\begin{figure}[bth!p]
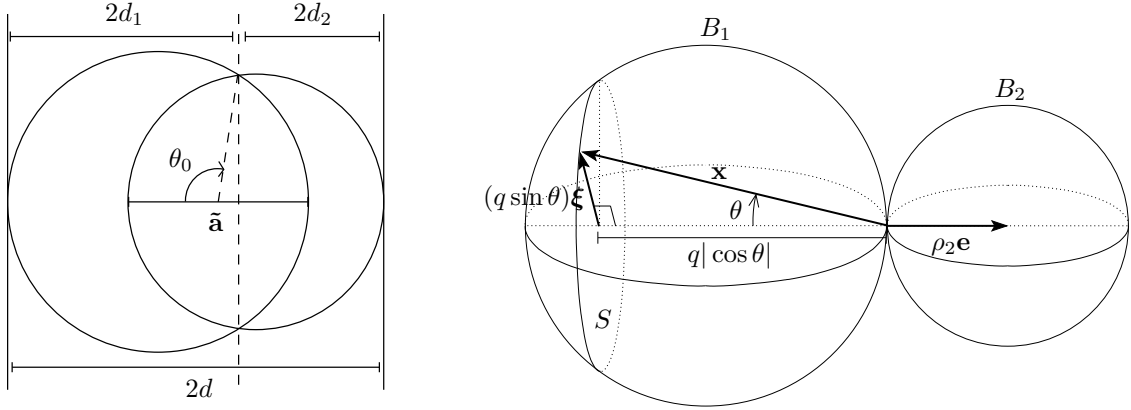

  \centering
  \subfigure{\input{theta0.pstex_t}} \quad \qquad
  \subfigure{\input{spherical2.pstex_t} }
  \caption{\label{fig:sph2}Angle $\theta_0>\frac{\pi}{2}$ and choice of spherical coordinates for $\delta=0$.}
\end{figure}
The main problems at this point are that if $\delta\to 0$
then $\rho_2$ is of the order of $\frac{v_2^\frac{1}{n}d}{v_1^\frac{1}{n}+v_2^\frac{1}{n}}$
(so $\frac{d}{q}\to \infty$ on $\partial B_2 \cap \partial \Omega_2$
if $\frac{v_2}{v_1}\to 0$)
and $q(\vecg \zeta)$ tends to vanish on $\partial B_1 \cap \partial B_2$ (see Figure \ref{fig:sph2}).
Parametrize $\S^{n-1}$ by $\vecg\zeta = -\cos \theta\, \vec e + \sin \theta\, \vecg \xi$ with $\theta \in (0, \pi)$
and $\vecg \xi \in S:= \S^{n-1} \cap \langle \vec e \rangle^\perp$. Since $\frac{q^n}{d^n} |\log \frac{d}{q}|$ is bounded we
only study the term with $|Dq|$, that is, we are to prove that
\begin{align*}
 \mathcal H^{n-2}(S) \left (\int_0^\frac{\pi}{2} + \int_{\frac{\pi}{2}}^\pi \right ) (\sin \theta)^{n-2}\frac{|Dq(\vecg\zeta(\theta, \vecg\xi)|}{d}
\left ( \log \frac{d}{q(\vecg\zeta(\theta, \vecg \xi))} \right )_+ \dd\theta
\end{align*}
is bounded independently of $d$, $\delta$, $v_1$, and $v_2$.
It can be shown that $\vec a^* + q(\theta, \vecg \xi) \vecg \zeta(\theta, \vecg \xi)\in \partial B_1$ for all $\theta \in (0, \frac{\pi}{2})$ (due to the
fact that $\rho_1 \geq \rho_2$, see Figure \ref{fig:sph2}), and clearly $\vecg\zeta \cdot (\vec a^* - \vec{\tilde a}_1)=-\cos \theta (d-\rho_2) <0$.
Lemma \ref{le:ub2}\ref{it:b2}) can thus be used to estimate the first integral by
\begin{align*}
 2\int_0^{\frac{\pi}{2}} \frac{\rho_1}{d} \log \frac{d}{\rho_1 \cos \theta} \dd\theta
\leq 2\left ( \max_{t\in [0,1]} |t\log t|\right )
\int_0^{\frac{\pi}{2}} \left |\log \frac{1}{2}\left (\frac{\pi}{2} - \theta\right)\right |\dd\theta
=  \frac{\pi}{e} \left ( 1 + \log \frac{4}{\pi} \right ).
\end{align*}
As for the second integral we divide $(\frac{\pi}{2}, \pi)$ into $(\frac{\pi}{2}, \theta_0] \cup [\theta_0, \pi)$,
according to whether
$\vec a^* + q(\theta, \vecg \xi) \vecg \zeta(\theta, \vecg \xi)$ belongs to $\partial B_1$ or to $\partial B_2$.
For $\theta > \theta_0$ we can still use Lemma \ref{le:ub2}\ref{it:b2}) (this time with $\vec{\tilde a}=\vec{\tilde a}_2$ and
$\rho=\rho_2$) to obtain exactly the same upper bound as before. For $\theta \in (\frac{\pi}{2}, \theta_0)$,
use parts \ref{it:b1}) and \ref{it:b3}) of Lemma \ref{le:ub2} together with $\rho_1 - |\vec a^*-\vec{\tilde a}_1|=d\delta$ to obtain
\begin{align*}
 \frac{|Dq|}{d} \leq \frac{2(d-\rho_2)}{\delta \rho_1} \frac{q^2}{d^2} \sin \theta
\quad \text{and} \quad
  |Dq| \leq 16(d-\rho_2) \left( 1 + \frac{(d-\rho_2)|\cos \theta|}{\sqrt{\delta \rho_1 d}} \right ) ^{-2}\sin \theta.
\end{align*}
Then, for any $\alpha \in (0,\frac{1}{2})$, using that $t^{2\alpha}|\log t|\leq (2\alpha e)^{-1}$ for every $t\in (0,1)$,
\begin{align*}
 \int_{\frac{\pi}{2}}^{\theta_0} \frac{|Dq|}{d} \left ( \log \frac{d}{q} \right)_+ \dd\theta
&\leq
\int_{\frac{\pi}{2}}^{\theta_0}  \left|\frac{Dq}{d}\right|^{1-\alpha} \left|\frac{Dq}{d}\right|^{\alpha}
\left ( \log \frac{d}{q} \right)_+ \dd\theta \\
&\leq \frac{2^\alpha 16^{1-\alpha}}{2\alpha e} \left (\frac{d-\rho_2}{d}\right)^{1-\alpha}
\left ( \frac{d-\rho_2}{\delta \rho_1} \right )^\alpha
\int_{\frac{\pi}{2}}^\pi
\left ( 1 + \frac{(d-\rho_2)|\cos \theta|}{\sqrt{\delta \rho_1 d}} \right ) ^{2(\alpha-1)} \sin \theta \dd\theta \\
&\leq \frac{8^{1-\alpha}}{\alpha e} \left (\frac{\delta \rho_1}{d} \right )^{\frac{1}{2}-\alpha}
\fint_0^{\frac{d-\rho_2}{\sqrt{\delta \rho_1 d}}}
(1+t)^{2\alpha-2} \dd t
\end{align*}
The last integral can be bounded by means of the relation
\begin{align*}
(1-2\alpha)\int_0^x (1+t)^{2\alpha-2} \dd t =  1- \frac{1}{(1+x)^{1-2\alpha}} < 1-\frac{1}{1+x}=\frac{x}{1+x}.
\end{align*}
Using that $\gamma + \sqrt{1-\gamma}>1$ for all $\gamma\in (0,1)$ (applied to
$\gamma=\frac{d-\rho_2}{\rho_1}=\frac{|\vec a^*-\vec a_1|}{\rho_1}$),
\begin{align*}
 \int_{\frac{\pi}{2}}^{\theta_0} \frac{|Dq|}{d} \left ( \log \frac{d}{q} \right)_+ \dd\theta
&\leq
\frac{8^{1-\alpha}}{\alpha(1-2\alpha)e}
\left (\frac{\delta \rho_1}{d} \right )^{\frac{1}{2}-\alpha}
\frac{d-\rho_2}{\rho_1} \frac{1}{\gamma + \sqrt{1-\gamma}}\\
&\leq
\frac{8^{1-\alpha}}{\alpha(1-2\alpha)e} \delta^{\frac{1}{2}-\alpha} \left (
\frac{d-\rho_2}{d}\right )^{\frac{1}{2}-\alpha} \left (\frac{d-\rho_2}{\rho_1} \right )^{\frac{1}{2}-\alpha}\\
&
\leq \frac{8^{1-\alpha}}{\alpha(1-2\alpha)e}
\delta^{\frac{1}{2}-\alpha}(1-\delta)^{\frac{1}{2}-\alpha}.
\end{align*}
We conclude that for all $R>0$
\begin{multline*}
 \frac{1}{n} \int_{B(\vec a^*, R) \setminus \overline{\Omega_1\cup \Omega_2}} \left | \frac{D\vec u}{\sqrt{n-1}} \right |^n \dd\vec x
 \leq
C(v_1+v_2+ \omega_n R^n)
+(v_1+v_2)\Big ( 1 + C(1-\delta) \Big ) \left ( \log \frac{R}{d} \right )_+.
\end{multline*}
\\

{\it - Step 4: } Estimating the energy in $\Omega_i$.\\ 
Near the cavitation points we still have that
$\fint \omega_n q_i^n \dd\mathcal H^{n-1} = |\Omega_i|$, $i=1,2$, so by Lemma \ref{le:ub1}
\begin{align*}
 \frac{1}{n} \int_{\Omega_i\setminus B_{\ep_i}(\vec a_i)} \left | \frac{D\vec u}{\sqrt{n-1}} \right |^n \dd\vec x
 &\leq
C(v_i+\omega_n \rho_i^n)
+ v_i \left ( \log \frac{2\rho_i}{\ep_i} \right)_+ \\
& \quad +C\frac{v_1+v_2}{|\Omega_1 \cup \Omega_2|} \left ( \int_{\S^{n-1}} \max \{q_i, |Dq_i|\}^{n-1} |Dq_i| \dd\mathcal H^{n-1} \right)
\log \frac{2d}{\ep_i} \\
& \leq C(v_i+\omega_n \rho_i^n)
+ v_i \log \frac{2d}{\ep_i}
+ C(v_1+v_2)\frac{\rho_i^{n-1}}{d^{n-1}} \left (\int_{\S^{n-1}} \frac{|Dq_i|}{d} \right)
\log \frac{2d}{\ep_i}.
\end{align*}
For $\Omega_1$ set $\vecg\zeta =-\cos \theta \vec e +\sin \theta \vecg \xi$.
 If $\theta\in (0, \frac{\pi}{2})$ then, by Lemma \ref{le:ub2},
using that $|\vec a_1-\vec{\tilde a}_1|=\rho_1-d_1$,
\begin{align*}
 \int_0^{\frac{\pi}{2}}|Dq_1|\sin^{n-2}\theta\dd\theta
&\leq 16(\rho_1-d_1)\int_0^{\frac{\pi}{2}}\left (1 + \frac{\rho_1-d_1}{\sqrt{d_1\rho_1}}\cos \theta \right )^{-2} \sin \theta \dd\theta
\\
&=16\sqrt{d_1\rho_1} \int_0^{\frac{\rho_1-d_1}{\sqrt{d_1\rho_1}}} (1+t)^{-2}\dd t
= \sqrt{\frac{d_1}{\rho_1}} \frac{\rho_1-d_1}{\gamma + \sqrt{1-\gamma}},
\end{align*}
with $\gamma=1-\frac{d_1}{\rho_1}$. Since $\gamma +\sqrt{1-\gamma}\leq 1$ for $\gamma\in [0,1]$,
\begin{align*}
 \rho_1^{n-1}\int_0^{\frac{\pi}{2}}|Dq_1|\sin^{n-2}\theta\dd\theta
\leq \rho_1^{n-2} \sqrt{d_1\rho_1} (\rho_1-d_1) .
\end{align*}

Define $\theta_1$ as in Figure \ref{fig:upbnd}. By Lemma \ref{le:ub2},
$|Dq_1|\leq 2(\rho_1-d_1)\sin \theta$ and $q_1\geq \sqrt{d_1\rho_1}$, hence
\begin{align*}
 \rho_1\int_{\frac{\pi}{2}}^{\theta_1}|Dq_1|\sin^{n-2}\theta\dd\theta
&\leq \rho_1 (\rho_1-d_1) |\cos \theta_1|
\leq (\rho_1-d_1) \frac{d_1\rho_1}{q(\theta_1)}
\leq \sqrt{d_1\rho_1} (\rho_1-d_1).
\end{align*}
For $\theta \in (\theta_1, \pi)$, $q_1(\vecg\zeta)$ is given by $q_1\vecg\zeta \cdot \vec e=d_1$ hence
\begin{align*}
q_1(\theta) = \frac{d_1}{\cos(\pi-\theta)}
\quad \text{and} \quad
|Dq_1(\vecg\zeta(\theta, \vecg\xi))| =\left|\frac{q_1(1-\vecg\zeta\otimes\vecg\zeta)\vec e}{-\vecg\zeta\cdot \vec e}\right| =\frac{d_1\sin\theta}{\cos^2(\pi-\theta)}.
\end{align*}
Using that $1-|\cos \theta_1|= \frac{\sin^2 \theta_1}{1+ |\cos \theta_1|}\leq \sin^2 \theta_1$ and that $q(\theta_1)\geq (\rho_1-d_1)\cos \theta +\sqrt{d_1\rho_1}\geq \sqrt{d_1\rho_1}$,
\begin{align*}
\rho_1\int_{\theta_1}^{\pi} |Dq_1|\dd\theta
 \leq d_1\rho_1 \int_{|\cos \theta_1|}^1 \frac{\dd t}{t^2}
 \leq \rho_1\frac{d_1\sin^2 \theta_1}{\cos (\pi - \theta_1)}
= \frac{\rho_1(q_1(\theta_1)\sin \theta_1)^2}{q_1(\theta_1)}
\leq 4\sqrt{d_1\rho_1}(\rho_1-d_1),
\end{align*}
the last equality being due to the fact that $q(\theta_1)\cos\theta_1=d_1$ and $\vec a_1+q(\theta_1)\vecg\zeta(\theta_1, \vecg\xi) \in \partial B(\vec{\tilde a}_1, \rho_1)$.
Now we show that $\max\{q_1,|Dq_1|\}\leq 8 \rho_1$.
The fact that $q(\theta_1)\geq \sqrt{d_1\rho_1}$ implies that
$\rho_1|\cos\theta_1| \leq \sqrt{d_1\rho_1}$. Clearly
$q(\theta)$ is decreasing, therefore
\begin{align*}
q(\theta)\leq q(\theta_1) \leq 2((\rho_1-d_1)|\cos \theta_1| + \sqrt{d_1\rho_1})
\leq 4\sqrt{d_1\rho_1}\leq 4\rho_1.
\end{align*}
As for $|Dq_1|$, we have that $q_1(\theta)\sin \theta$ is decreasing and
$q(\theta_1)\sin\theta_1=2\sqrt{d_1(\rho_1-d_1)}$, then
\begin{align*}
|Dq_1| = \frac{q_1(q_1\sin \theta)}{q_1\cos(\pi-\theta)}
\leq \frac{2q_1(\theta_1)\sqrt{d_1(\rho_1-d_1)}}{d_1}
\leq 8\sqrt{\rho_1(\rho_1-d_1)}|\leq 8\rho_1.
\end{align*}

The study of $\vec u$ in $\Omega_2$ being completely analogous, the conclusion is that  for all $R>0$
\begin{multline*}
 \frac{1}{n} \int_{B(\vec a^*, R) \setminus (B_{\varepsilon_1}(\vec a_1)\cup B_{\varepsilon_2}(\vec a_2))} \left | \frac{D\vec u}{\sqrt{n-1}} \right |^n \dd\vec x
 \leq
C(v_1+v_2+ \omega_n R^n)
+v_1\log \frac{R}{\ep_1} + v_2\log \frac{R}{\ep_2}
\\
+C(v_1+v_2)\left ( (1-\delta) \left ( \log \frac{R}{d} \right )_+
+ \sqrt{\frac{d_1}{d}} \frac{\rho_1-d_1}{d} \log \frac{d}{\ep_1}
+ \sqrt{\frac{d_2}{d}} \frac{\rho_2-d_2}{d} \log \frac{d}{\ep_2}
\right )
\end{multline*}
In the case of $\vec a_1$ it is $\rho_1-d_1$ that has an interesting behaviour, whereas for $\vec a_2$ it is $d_2$. This follows from our final ingredient:
the `height' of $B(\vec a_1, \rho_1)\cap B(\vec a_2, \rho_2)$, whether we measure it from the first ball or from the second, is the same. The corresponding expression is $d_1(\rho_1-d_1)= d_2(\rho_2-d_2)$.
As a consequence,
\begin{align*}
\frac{\rho_1-d_1}{d} = \frac{\delta(\rho_1-d_1)}{(\rho_1-d_1) + (\rho_2-d_2)}
=\frac{\delta d_2}{d_1+d_2}=\delta \frac{d_2}{d}.
\end{align*}
The theorem is thus proved since, by Lemma \ref{le:areal},
\begin{align*}
\left(\frac{d_2}{d}\right)^{\frac{n+1}{2}} \leq C\frac{|\Omega_2|}{\rho_2^{\frac{n-1}{2}}d^{\frac{n+1}{2}}}
\leq C\frac{\frac{v_2|\Omega_1\cup \Omega_2|}{v_1+v_2}}{
	\left(\frac{v_2^\frac{1}{n}}{v_1^\frac{1}{n}+v_2^\frac{1}{n}} d\right)^{\frac{n-1}{2}}d^{\frac{n+1}{2}}}
\leq C\left (\left(\frac{v_2}{v_1+v_2}\right)^\frac{1}{n} \right)^{\frac{n+1}{2}}.
\end{align*}
\end{proof}

\subsection{Transition to radial symmetry} \label{se:ub2} \label{se:ubD}

Our proof of Theorem \ref{th:ub2} is based on the following result
(see \cite{Moser65,DaMo90,Ye94,McMullen98,CuDaKn09,BaDa09} for related work):

\begin{proposition}[Rivi\`ere-Ye, \cite{RiYe96}, Thm. 8] \label{pr:RiviereYe}
 Let $D$ be a smooth domain, $k=0, 1, 2, \ldots$ and suppose that
$g\in C^{k,1}(\overline D)=W^{k+1, \infty}(D)$
with $\inf_D g>0$ and $\fint_D g=1$.
Then, there exists a diffeomorphism $\vecg \phi:\overline D\to \overline D$,
satisfying $\det D\vecg \phi = g$ in $D$ and $\vecg \phi=\id$ on $\partial D$,
such that, for any $\alpha<1$, $\vecg \phi$ is in $C^{k+1,\alpha}(\overline D)$ and
\begin{align*} 
 \|\vecg \phi - \id\|_{C^{k+1, \alpha}(\overline D)} \leq C\|\vec g -1\|_{C^{k,1}(\overline D)}
\end{align*}
for any $0<\delta<1$, where $C$ depends only on $\alpha$, $k$, $D$, $\inf_D g$, $\delta$, and $\|g\|_{0,\delta}$.
\end{proposition}

\begin{lemma} \label{le:nc}
Let $\vecg \zeta \in \S^{n-1} \mapsto \vec a^* + q(\vecg \zeta) \vecg \zeta$ be
the polar parametrization of $\partial (\overline {\Omega_1 \cup \Omega_2})$ and define
\begin{align} \label{eq:fandR3}
\rho(\vecg \zeta)^n &:= R_1^n + (v_1+v_2) \frac{q(\vecg \zeta)^n }{|\Omega_1\cup \Omega_2|},
\qquad
\vecg \zeta\in \S^{n-1},
\end{align}
$R_1$ being fixed and such that $\Omega_1\cup \Omega_2 \subset B(\vec{a^*}, R_1)$.
Suppose that $\vec u$ is a one-to-one incompressible map
from $\{R_1< |\vec x - \vec {a^*}|< R_2\}$ onto $\{r\vecg \zeta: \rho(\vecg \zeta)<r < R_3\}$, for some
$R_2$, $R_3\geq 0$.
Then
\begin{align*}
 \omega_n (R_2^n -R_1^n) > \frac{\frac{\pi}{3} - \frac{1}{2}}{2^{n-2} 3 \pi} (v_1+v_2) (1-\delta).
\end{align*}
\end{lemma}
\begin{proof}
Denote $\max_{\S^{n-1}} q = 2\rho_1- \delta d$ by $q_{\max}$. By incompressibility (using that
$\mathcal H^{n-1}(\S^{n-1})=n\omega_n$),
\begin{align}
  \omega_n R_2^n - \omega_n R_1^n 
&= |\{\vec x: R_1< |\vec x - \vec {a^*}|< R_2\}| 
= \label{eq:R3ub}
 |\{r\vecg \zeta: \rho(\vecg \zeta)<r < R_3\}| 
\\&= \int_{\S^{n-1}}\!\!\int_{\rho(\vecg\zeta)}^{R_3} r^{n-1} \dd r \dd\mathcal H^{n-1}(\vecg \zeta)
= \omega_n R_3^n - \omega_n R_1^n -(v_1+v_2) \frac{\fint_{\S^{n-1}} \omega_n q^n \dd\mathcal H^{n-1}
}{|\Omega_1 \cup \Omega_2|}. \nonumber
\end{align}
Hence,
the requirement that $R_3\geq \rho(\vecg \zeta)$ for all $\vecg \zeta \in \S^{n-1}$
is equivalent to
\begin{align*}
 \omega_n (R_2^n-R_1^n) > (v_1+v_2) \frac{\omega_n \fint_{\S^{n-1}} (q_{\max}^n -q^n ) \dd\mathcal H^{n-1}}{|\Omega_1\cup
\Omega_2|}.
\end{align*}
 Write $\vecg \zeta := -\cos \theta \vec e + \sin \theta \vecg \xi$
with $\theta \in [0, \pi]$, $\vecg \xi \in S:=\S^{n-1}\cap \langle \vec e \rangle ^\perp$.
For all $\theta \in (0, \frac{\pi}{2})$
\begin{align*}
 q_{\max} - q(\theta)
&= 2\rho_1 - \delta d - (\rho_1-\delta d) \cos \theta - \sqrt{\delta d (2\rho_1 - \delta d)
+ (\rho_1 - \delta d)^2 \cos ^2 \theta} \\
&= \frac{ \big (\rho_1 + (\rho_1 - \delta d)(1-\cos \theta)\big )^2 - \big ( \delta d (2\rho_1-
\delta d ) + (\rho_1 - \delta d)^2 \cos^2\theta \big)}{\rho_1 +
(\rho_1-\delta d)(1-\cos \theta) +
\sqrt{\delta d (2\rho_1 - \delta d ) + (\rho_1-\delta d )^2 \cos^2 \theta}} \\
&> \frac{(\rho_1-\delta d)^2 (\sin^2\theta + (1-\cos \theta)^2) + 2\rho_1(\rho_1-\delta d)(1-\cos \theta)}{
(2\rho_1 - \delta d) + (2\rho_1 - \delta d) + \rho_1 - \delta d} \\
&= \frac{2(\rho_1-\delta d)(2\rho_1 - \delta d)(1-\cos \theta)}{5\rho_1 - 3\delta d} > \frac{2}{3}(d-\rho_2)(1-\cos \theta)
> \frac{2d}{3} (1-\delta) (1-\cos \theta),
\end{align*}
where we have used that $\rho_1-d\delta= d -\rho_2$ and $\rho_2\leq d$.
Therefore,
\begin{align}
 \frac{\omega_n \fint_{\S^{n-1}} (q_{\max}^n -q^n ) \dd\mathcal H^{n-1}}{|\Omega_1\cup
\Omega_2|}
&> \frac{\mathcal H^{n-2}(S)}{n\omega_n}
\frac{\int_\frac{\pi}{6}^\frac{\pi}{2} (q_{\max} - q) q_{\max}^{n-1} (\sin \theta)^{n-2} \dd\theta}{2d^n}
> \frac{\frac{\pi}{3} - \frac{1}{2}}{2^{n-2} 3 \pi} (1-\delta).
\label{eq:lbqd}\end{align}
\end{proof}

\begin{proof}[Proof of Theorem \ref{th:ub2}]
We prove the theorem in the following stronger version
(see the remark after the proof of Corollary \ref{co:Dirich}):
``\emph{Let $R_1$, $R_2$ be such that
\begin{align} \label{eq:R1R2b}
  R_1\geq 2d
\qquad \text{and}\qquad
  \omega_n (R_2^n -R_1^n) > 4^nn (v_1+v_2)(1-\delta)
\end{align}
($\delta$, $v_1$, $v_2$, $\vec a_1$, $\vec a_2$, $d$, $\ep_1$, and $\ep_2$ being
as in the original statement).
Then there exists $\vec{a^*}$, $C_1$, $C_2$, and $\vec u: \R^n\setminus \{\vec a_1, \vec a_2\}\to \R^n$
 such that
$\vec u|_{\R^n \setminus B(\vec{a^*}, R_2)}$ is radially symmetric
and for all $R\geq R_1$
\begin{align*}
 \frac{1}{n}
\int_{B(\vec a^*, R) \setminus (B_{\varepsilon_1}(\vec a_1)\cup B_{\varepsilon_2}(\vec a_2))}
&
\left | \frac{D\vec u}{\sqrt{n-1}} \right |^n \dd\vec x
 \leq
C_1(v_1+v_2+ \omega_n R^n)
+v_1\log \frac{R}{\ep_1} + v_2\log \frac{R}{\ep_2}
\\
& + C_2(v_1+v_2)\left ( (1-\delta) \log \frac{R_1}{d}
+ \delta
\left(\sqrt[n]{\frac{v_2}{v_1}}  \log \frac{d}{\ep_1}
+ \sqrt[2n]{\frac{v_2}{v_1}} \log \frac{d}{\ep_2}
\right ) \right)
\\
& + (v_1+v_2+\omega_n R_2^n) \Sigma\left (\frac{(v_1+v_2)(1-\delta)}{\omega_n (R_2^n-R_1^n)}\right) 
\left ( \frac{\min\{R^n, R_2^n\}}{R_1^n} -1\right ) ,
\end{align*}
 the function $\Sigma$ being such that $\Sigma(t)< \infty$ for $t\in [0, \frac{1}{4^nn})$
and $\Sigma(t)=O(t^{n(n-1)})$ as $t\to 0$}''.
The Theorem follows by choosing $R_1$ and $R_2$ as in \eqref{eq:R1R2}.

Since the constant in Proposition \ref{pr:RiviereYe} depends on the reference domain, we
work on the annulus $D:=\{\vec z\in \R^n : 1\leq |\vec z| \leq \sqrt[n]{2}\}$
(we choose $\sqrt[n]{2}$ so that $|D|=\omega_n$).
Our strategy is to define $\vec u: B(\vec{a^*}, R_1) \setminus \{\vec a_1, \vec a_2\} \to \R^n$ as in
Theorem \ref{th:ub1} and to look for a map
\begin{align*}
\vec u:\{\vec x \in \R^n : R_1\leq |\vec x - \vec{a^*}| \leq R_2\}\to \{\vec y =
\lambda \vec{a^*} + r\vecg \zeta:
\rho(\vecg \zeta) \leq r \leq R_3, \vecg \zeta \in \S^{n-1}\}
\end{align*}
(where $\rho$ is defined in \eqref{eq:fandR3}) of the form $\vec u=\vec v \circ \vecg\phi^{-1} \circ \vec w^{-1}$, with
$\vecg \phi:\overline D \to \overline D$ a diffeomorphism and
\begin{align}
\begin{gathered}
\vec w(r\vecg \zeta)
:= \vec a^* + \left ((2-r^n)R_1^n +(r^n -1) R_2^n \right )^\frac{1}{n}\vecg \zeta,
\\
\vec v(r \vecg \zeta)
:= \lambda \vec a^* + \left ( (2-r^n) \rho(\vecg \zeta)^n + (r^n -1) R_3^n \right )^\frac{1}{n} \vecg \zeta.
\end{gathered} \label{eq:vyw}
\end{align}
The maps $\vec w$ and $\vec v$ are
parametrizations of the reference and target domains, and are
defined so that $\det D\vec w$ is constant and
$\vec v\circ \vec w^{-1}$ sends $\partial B(\vec {a^*}, R)$, $R_1\leq R\leq R_2$
onto a curve enclosing
a volume of exactly $v_1+v_2 + \omega_n R^n$
(as can be seen by writing
\begin{align} \label{eq:vw-1}
 \vec v \circ \vec w^{-1} (\vec {a^*}+ R\vecg \zeta) = \lambda \vec{\vec a ^*}
+ \left ( R^n + \frac{v_1+v_2}{\omega_n} \left ( 1 + \frac{R_2^n -R^n}{R_2^n - R_1^n }
\frac{\omega_n \left ( q^n - \fint q^n \right)}{|\Omega_1\cup \Omega_2|} \right ) \right ) ^\frac{1}{n}
\vecg \zeta ,
\end{align}
and by considering that $|\{\lambda \vec a^* + r\vecg\zeta: \vecg \zeta \in \S^{n-1},\ 0<r<\rho(\vecg \zeta)\}|
= \fint_{\S^{n-1}} \omega_n \rho^n\dd\mathcal H^{n-1}$).
The problem for $\vecg\phi$ is $\vecg \phi = \id$ on $\partial D$,
$\det D\vecg\phi= g := \frac{\det D\vec v}{\det D\vec w}$ in $D$.
To use Proposition \ref{pr:RiviereYe} we need to bound
\begin{align} \label{eq:gDg}
 g(r\vecg \zeta) - 1 = \frac{v_1+v_2}{\omega_n (R_2^n -R_1^n)} \left ( 1 - \frac{\omega_nq(\vecg\zeta)^n}{|\Omega_1 \cup \Omega_2|} \right)
\quad \text{and} \quad
Dg(r\vecg \zeta) = -\frac{v_1+v_2}{R_2^n -R_1^n} \frac{nq^{n-1} Dq(\vecg \zeta)}{r|\Omega_1\cup \Omega_2|}
\end{align}
for all $\vecg \zeta \in \S^{n-1}$, $r\in [1, \sqrt[n]{2}]$
(the constant in Proposition \ref{pr:RiviereYe} depends on $\|g\|_{0,\delta}$, so
it is not sufficient to control only $\|g-1\|_{L^\infty}$).
Using \eqref{eq:bdOm} and the fact that $\rho_1(\delta) \leq d$ and
$q(\vecg \zeta) \geq \delta d$ for all $\delta$, $\vecg \zeta$,
\begin{align}
 \frac{\omega_n \fint_{\S^{n-1}} (q_{\max}^n -q^n ) \dd\mathcal H^{n-1}}{|\Omega_1\cup
\Omega_2|}
&\leq n(2d)^{n-1}\frac{(2\rho_1-\delta d) - \delta d}{2^{-n} d^n} \leq 4^n n(1-\delta).
\label{eq:ubqd}
\end{align}
By Lemma \ref{le:ub2}\ref{it:b1}),
\begin{align*}
 \sup |Dg| \leq \frac{(v_1+v_2)}{R_2^n -R_1^n} \frac{2n(2d)^{n-1}(1-\delta)d}{2^{-n}\omega_n d^n}
\leq 4^nn \frac{(v_1+v_2)(1-\delta)}{\omega_n (R_2^n-R_1^n)}.
\end{align*}
This and Proposition \ref{pr:RiviereYe} imply the existence of a (piecewise smooth) solution $\vecg\phi$ such that
\begin{align} \label{eq:RYPHI}
\|\vecg\phi - \id\|_{C^1(\overline D)}
\leq \Sigma\left (  \frac{(v_1+v_2)(1-\delta)}{\omega_n (R_2^n-R_1^n)} \right )
\end{align}
for some function $\Sigma$ satisfying $\Sigma(t)< \infty$ for $t\in [0, \frac{1}{4^nn})$ and $\Sigma(t)=O(t)$ as $t\to 0$.

Define $\vec u=\vec v\circ \vecg \phi^{-1} \circ \vec w$. Writing $\vec x=\vec w(\vecg \phi(\vec z))$
and using \eqref{eq:vyw} and $\det D\vecg \phi = g$ we obtain
\begin{align} \label{eq:bndDu1}
 |D\vec u (\vec x)|^n
&= \left |\frac{D\vec v (\vec z) \adj D\vecg\phi(\vec z) D\vec w^{-1}(\vec x)}{\det D\vecg \phi(\vec z)}\right |^n
\leq C_n \frac{R_3^n}{R_1^n} \left ( \frac{\|D\vecg \phi\|_{L^\infty}^{n-1}}{1-\|g-1\|_{L^\infty}} \right ) ^n.
\end{align}
Combining \eqref{eq:R3ub} and \eqref{eq:areaQ} we obtain that for all $R\leq R_2$
\begin{align} \label{eq:R3R1}
 \int_{B(\vec a^*, R)\setminus B(\vec a^*, R_1)} \frac{R_3^n}{R_1^n} \dd\vec x 
\leq (v_1+v_2+\omega_n R_2^n)\frac{\omega_n R^n - \omega_n R_1^n}{\omega_n R_1^n}.
\end{align}
By \eqref{eq:gDg} and \eqref{eq:ubqd}, $\|g-1\|_{L^\infty} \leq 4^n n t$, with 
$t:=\frac{(v_1+v_2)(1-\delta)}{\omega_n (R_2^n - R_1^n)}$.
Hence, by \eqref{eq:bndDu1}, \eqref{eq:R3R1}, and \eqref{eq:RYPHI},
\begin{align*}
 \int_{B(\vec a^*, R)\setminus B(\vec a^*, R_1)} |D\vec u (\vec x)|^n \dd\vec x 
\leq C (v_1+v_2+\omega_n R_2^n) \tilde \Sigma\left (\frac{(v_1+v_2)(1-\delta)}{\omega_n (R_2^n-R_1^n)}\right) 
\left ( \frac{R^n}{R_1^n} - 1 \right)
\end{align*}
for $R_1 \leq R\leq R_2$,
where $\tilde \Sigma(t):= \frac{\Sigma(t)^{n(n-1)}}{(1- 4^n n t)^n}$, $t\in [0, \frac{1}{4^nn})$
and $\tilde \Sigma(t)=O(t^{n(n-1)})$ as $t\to 0$.

The map $\vec u$ can be extended to $\R^n \setminus B(\vec a^*, R_2)$ by
$\vec u(r\vecg \zeta) := \vec a^* + (r^n +v_1+v_2)^\frac{1}{n}\vecg \zeta$. It satisfies
\begin{align*}
 \frac{1}{n} \int_{B(\vec a^*, R) \setminus B(\vec a^*, R_2)} \left |\frac{D\vec u}{\sqrt{n-1}}\right |^{n} \dd\vec x 
\leq C(v_1+v_2+\omega_n R^n) + (v_1+v_2) \log \frac{R}{R_2}.
\end{align*}
The energy inside $B(\vec a^*, R_1)$ has been estimated in Theorem \ref{th:ub1}. This completes the proof.

\end{proof}

\begin{remark}
 For Dirichlet boundary conditions that are not necessarily radially symmetric,
the above method can still be used provided there is an initial diffeomorphism
$\vec v$, from the reference domain $D=\{\vec z: 1 < |\vec z| < \sqrt[n]{2}\}$ 
onto the desired target domain,
for which $g:= \frac{\det D\vec v}{\det D\vec w}$ is bounded away from zero.
The energy estimate will depend on $\inf_D g$, $\|D\vec v\|_{\infty}\|D\vec w^{-1}\|_{\infty}$,
and $\|g\|_{\infty} + \|Dg\|_{\infty}$.
\end{remark}

\subsection{Proof of the preliminary lemmas}\label{se:auxi}
In this Section, we give the proofs of Lemmas  \ref{le:ub1},  \ref{le:ub2} and \ref{le:areal}.

\begin{proof}[Proof of Lemma \ref{le:ub1}]
First we show that for any
map of the form $\vec u(\vec x):= \lambda \vec a + f(\vec x) \frac{\vec x-\vec a }{|\vec x-\vec a |}$
the incompressibility equation reduces to an ODE of the form
$f^{n-1}(r, \vecg \zeta) \frac{\partial f}{\partial r} (r, \vecg \zeta) = r^{n-1}$.
In order to see this, consider a local parametrization
$(s_1, \ldots, s_{n-1}) \mapsto \vecg \zeta (s_1, \ldots, s_{n-1})$
of $\S^{n-1}$
and
introduce polar coordinates of the form
\begin{align*}
  \vec x = \vec x (r, s_1, \ldots, s_{n-1}) = \vec a + r \vecg \zeta (s_1, \ldots, s_{n-1}), \quad r>0,\quad (s_1, \ldots, s_{n-1})\in D\subset \R^{n-1},
\end{align*}
$D$ being some parameter space. The claim follows by observing that
\begin{align*}
 & \frac{\partial \vec u}{\partial r} \wedge \frac{\partial \vec u}{\partial s_1} \wedge \cdots \wedge
\frac{\partial \vec u}{\partial s_{n-1}}
= \det D\vec u(\vec x) \left (\frac{\partial \vec x}{\partial r} \wedge \frac{\partial \vec x}{\partial s_{1}}
\wedge \cdots \wedge \frac{\partial \vec x}{\partial s_{n-1}}\right )
= \det D\vec u(\vec x) \left ( \vecg \zeta \wedge \bigwedge_{k=1}^{n-1} r\frac{\partial \vecg \zeta}{\partial s_k} \right )
\end{align*}
and
\begin{align*}
\frac{\partial \vec u}{\partial r} \wedge \frac{\partial \vec u}{\partial s_1} \wedge \cdots \wedge
\frac{\partial \vec u}{\partial s_{n-1}}
=\frac{\partial f}{\partial r} \vecg \zeta
\wedge
\bigwedge_{k=1}^{n-1} \left (
    \frac{\partial f}{\partial s_k} \vecg \zeta + f \frac{\partial \vecg \zeta}{\partial s_k} \right )
= f^{n-1} \frac{\partial f}{\partial r}
\left ( \vecg \zeta \wedge \bigwedge_{k=1}^{n-1} \frac{\partial \vecg \zeta}{\partial s_k} \right ).
\end{align*}

From the above we find that $\vec u(\vec x):= \lambda \vec a + f(\vec x) \vecg \zeta$
is incompressible provided $f(r, \vecg \zeta)^n  \equiv r^n + A\left ( \vecg \zeta\right )^n$,
for some $A: \S^{n-1} \to \R$. The definition in \eqref{eq:anglepreserving},
namely, $f^n=r^n + \frac{v}{|\Omega|}q^n$,
 is obtained
by imposing the boundary condition $\vec u(\vec x )=\lambda \vec x$ on $\partial \Omega$.
Differentiating \eqref{eq:anglepreserving} with respect to $\vecg \zeta$ yields
\begin{align*}
 f^{n-1}(r,\vecg \zeta)D_{\vecg \zeta} f(r,\vecg \zeta) = \frac{v}{|\Omega|}q^{n-1}(\vecg \zeta) Dq(\vecg \zeta),
\quad D_{\vecg \zeta} f(r, \vecg \zeta),\ Dq(\vecg \zeta): T_{\vecg \zeta}(\S^{n-1})\to \R,
\end{align*}
$T_{\vecg \zeta}(\S^{n-1})$ being the tangent plane to $\S^{n-1}$ at $\vecg \zeta$.
Identifying, in the usual manner,
\begin{align} \label{eq:DzetaF}
D_{\vecg\zeta} f(r,\vecg\zeta)= \frac{v}{|\Omega|}\frac{q^{n-1}(\vecg \zeta)}{f^{n-1}(r, \vecg \zeta)}
Dq(\vecg\zeta)\in (T_{\vecg\zeta}(\S^{n-1}))^* 
\end{align}
with a vector in $\langle \vecg \zeta \rangle^\perp \subset \R^n$,
 from $f(x)=f(r(\vec x), \vecg \zeta(\vec x))$, $r(\vec x)=|\vec x -\vec a|$, 
$\vecg \zeta(\vec x)= \frac{\vec x-\vec a}{|\vec x -\vec a|}$ we obtain 
\begin{align}\label{starr}
Df(\vec x)=\frac{\partial f}{\partial r} Dr + (D\vecg \zeta)^T D_{\vecg \zeta} f
= \frac{\partial f}{\partial r} \vecg \zeta + \frac{D_{\vecg\zeta} f}{r},
\quad |Df|^2 = \left|\frac{\partial f}{\partial r}\right |^2 + \left (
\frac{v}{|\Omega|}\frac{q^{n-1}}{f^{n-1}}\frac{|Dq|}{r}\right )^2 ,
\end{align}
with
$Dr=\vecg\zeta$ and $D\vecg\zeta = \frac{1-\vecg\zeta \otimes \vecg\zeta}{r}$. 
Since $D\vec u = \vecg \zeta \otimes Df + fD\vecg \zeta$
and $(D\vecg \zeta) \cdot (\vecg\zeta \otimes Df) = \vecg \zeta \cdot ((D\vecg \zeta) Df)=0$,
using that $|D\vecg\zeta|^2= \frac{n-1}{r^2}$ 
and $\frac{\partial f}{\partial r} = \frac{r^{n-1}}{f^{n-1}}<1$ we find 
\begin{align} \label{eq:grad1ub}
|D\vec u|^2= |Df|^2 + f^2 |D\vecg\zeta|^2
= (n-1) \frac{f^2}{r^2} + \left|\frac{\partial f}{\partial r}\right |^2
+ \left|\frac{D_{\vecg\zeta} f}{r}\right |^2
\leq (n-1) \frac{f^2}{r^2} + 1
+ \left|\frac{D_{\vecg\zeta} f}{r}\right |^2.
\end{align}

The leading order term $(v_1+v_2)|\log \ep|$ in the energy estimates will come from $(n-1)\frac{f^2}{r^2}$, hence we need to write
$\left | \frac{D\vec u}{\sqrt{n-1}}\right | ^n$ as $\frac{f^n}{r^n}$ plus a remainder 
(for which we do not require an exact expression, only an upper bound).
To this end we bound $a^n - b^n$, with $a=\left |\frac{D\vec u}{\sqrt{n-1}}\right |$ and $b=\sqrt{\frac{1}{n-1} + \frac{f^2}{r^2}}$, 
by
\begin{eqnarray*}
\lefteqn{\left |\frac{D\vec u}{\sqrt{n-1}}\right|^n -\left ( \frac{1}{n-1} + \frac{f^2}{r^2} \right)^\frac{n}{2}
 \leq (a-b)\left |a^{n-1} + \cdots + b^{n-1}\right |
\leq n \frac{|a^2-b^2|}{a+b} \max\{a,b\}^{n-1} }\\
&& \leq \frac{n}{n-1}\frac{\left |{(D_{\vecg \zeta} f)/r}\right |^2}{a+b} \max\{a,b\}^{n-1}
\leq \frac{n}{n-1}\frac{\left |{(D_{\vecg \zeta} f)/r}\right |^2}{a} \max\{a,b\}^{n-1}.
\end{eqnarray*}
From $f^n = r^n + \frac{v}{|\Omega|}q^n$ and \eqref{eq:DzetaF} we find that
\begin{align} \label{eq:boundDzF}
  \frac{f^n}{r^n} = 1 + \frac{v}{|\Omega|} \frac{q^n}{r^n}, \qquad
  f \geq \frac{v^\frac{1}{n}}{|\Omega|^\frac{1}{n}} q, \qquad \text{and} \qquad
  |D_{\vecg \zeta} f| \leq \frac{v^\frac{1}{n}}{|\Omega|^\frac{1}{n}} |Dq|.
\end{align}
As a consequence of \eqref{eq:grad1ub}, $\sqrt{n-1}a\geq \frac{|D_{\vecg\zeta} f|}{r}$, hence
$ \frac{\left |{(D_{\vecg \zeta} f)/r}\right |^2}{\sqrt{n-1}a} \leq \frac{|D_{\vecg \zeta} f|}{r}$
and
\begin{align} \label{eq:grad1ub2}
\left |\frac{D\vec u}{\sqrt{n-1}}\right|^n -\left ( \frac{1}{n-1} + \frac{f^2}{r^2} \right)^\frac{n}{2}
\leq C(n)\frac{v^\frac{1}{n}}{|\Omega|^\frac{1}{n}} \frac{|Dq|}{r} 
\left ( 1 + \frac{f^n}{r^n} + \frac{|D_{\vecg \zeta} f|^n}{r^n} \right )^\frac{n-1}{n}
\end{align}
(we have used \eqref{eq:boundDzF} to bound $\frac{|D_{\vecg \zeta} f|}{r}$ and \eqref{eq:grad1ub} to bound $\max\{a, b\}$).
Proceeding analogously, writing $c=\frac{f}{r}\geq 1$ and $b^n -c^n \leq n \frac{b^2-c^2}{b+c} b^{n-1}
\leq n (b^2-c^2) b^{n-1}$, we obtain
\begin{align} \label{eq:byc}
 \left ( \frac{1}{n-1} + \frac{f^2}{r^2} \right)^\frac{n}{2} \leq \underbrace{1 + \frac{v}{|\Omega|}\frac{q^n}{r^n}}_{f^n /r^n}
+ C \left ( \frac{1}{(n-1)^\frac{n}{2}}+ \frac{f^n}{r^n}\right )^\frac{n-1}{n}
\leq \frac{v}{|\Omega|}\frac{q^n}{r^n} + C\left ( 1+ \frac{v^\frac{1}{n}}{|\Omega|^\frac{1}{n}} \frac{q}{r} 
\right )^{n-1} .
\end{align}

Writing $a^n = b^n + (a^n -b^n)$, equations \eqref{eq:byc}, \eqref{eq:grad1ub2}, and \eqref{eq:boundDzF} yield
\begin{align*}
r^{n-1} \left | \frac{D\vec u}{\sqrt{n-1}}\right | ^n
&
\leq r^{n-1}\left ( \frac{v}{|\Omega|}\frac{q^n}{r^n} + C\left ( 1+ \frac{|v|^\frac{1}{n}}{|\Omega|^\frac{1}{n}} \frac{q}{r}
\right )^{n-1}\right )
\\ & \hspace{1em} 
+ C\frac{v^\frac{1}{n}}{|\Omega|^\frac{1}{n}} |Dq|r^{n-2}
\left(1 + \frac{v^\frac{n-1}{n}}{|\Omega|^\frac{n-1}{n}} \frac{\max\{q^{n-1}, |Dq|^{n-1}\}}{r^{n-1}}\right )
\\
&
\leq 
C\left ( r+ \frac{|v|^\frac{1}{n}}{|\Omega|^\frac{1}{n}} q\right )^{n-1} 
+ \frac{Cv^\frac{1}{n}}{|\Omega|^\frac{1}{n}} |Dq|r^{n-2}
+
\left(\frac{q^n}{|\Omega|}+ C\frac{\max\{q, |Dq|\}^{n-1}|Dq|}{|\Omega|} \right )\frac{v}{r}.
\end{align*}
To finish the proof substitute both $r$ and $\frac{|v|^\frac{1}{n}}{|\Omega|^\frac{1}{n}}|Dq|$
in $\frac{v^\frac{1}{n}}{|\Omega|^\frac{1}{n}} |Dq|r^{n-2}$
with $r+ |v|^\frac{1}{n}\frac{|Dq|}{|\Omega|^\frac{1}{n}}$.
\end{proof}

\begin{proof}[Proof of Lemma \ref{le:ub2}]
Write $\vecg \zeta = \cos \theta\, \vec e  + \sin \theta\, \vecg \xi$, $\theta \in (0,\pi)$, $\vecg \xi \in \S^{n-1} \cap \langle \vec e \rangle ^\perp$.
 By virtue of $|(\vec a + q(\vecg \zeta) \vecg \zeta) - \vec{\tilde a}|^2 \equiv \rho^2$,
\begin{align} \label{eq:Q}
 q^2 + 2q \vecg\zeta \cdot ( \vec a - \vec{\tilde a} ) = \rho^2 -d^2, \qquad
q(\theta, \vecg\xi) = -d\cos \theta + \sqrt{ (\rho^2-d^2) + d^2 \cos ^2 \theta }.
\end{align}
Extending $q$ to $\R^n$ by $\tilde q(\vec x)= q(\vecg \zeta(\vec x))$, $\vecg \zeta(\vec x):=\frac{\vec x}{|\vec x|}$,
and differentiating with respect to $\vec x$, we obtain
\begin{align*}
 \left (2\tilde q  + 2 \vecg\zeta \cdot ( \vec a - \vec{\tilde a} ) \right )D \tilde q 
=- 2\tilde q(D\vecg \zeta)^T(\vec a -\vec {\tilde a}) = - 2\tilde q\frac{\vec 1 - \vecg \zeta \otimes \vecg \zeta}{|\vec x|}
(d\vec e).
\end{align*}
Our aim is to obtain bounds for $q:\S^{n-1}\to \S^{n-1}$ and $Dq(\vecg \zeta) \in (T_{\vecg \zeta} \S^{n-1})^*$.
We can identify $Dq$ with a vector in $\langle \vecg \zeta\rangle^\perp$ in the usual manner.
From the relation $D\tilde q\cdot \vec e = Dq \cdot ((D\vecg \zeta) \vec e)$ we know that $D\tilde q \perp \vecg \zeta$
and $|Dq(\vec x)|=|D\tilde q(\vec x)|$ for all $\vec x \in \S^{n-1}$.
Thus, using that $q^2 + 2q \vecg\zeta \cdot ( \vec a - \vec{\tilde a} ) = \rho^2 -d^2$,
\begin{align*}
 |Dq(\theta, \vecg \xi)| = \left |\frac{-2d q^2 (1-\vecg\zeta \otimes \vecg\zeta)\vec e }{
q^2 + (q^2 + 2q \vecg\zeta \cdot ( \vec a - \vec{\tilde a} ))}\right | \leq  \frac{ 2 d q^2\sin\theta }{
\max\{q^2, (\rho-d)(\rho+d)\}}
\leq \frac{2 d q^2 \sin\theta }{\max\{q^2, \rho(\rho - d)\}};
\end{align*}
this yields the bounds for $|Dq|$ in \ref{it:b1}).
The fact that $|q(\vecg \zeta)|\leq 2\rho$ for all $\vecg \zeta \in \S^{n-1}$ 
follows from
$q(\vecg\zeta) = \dist (\vec a + q(\vecg \zeta) \vecg \zeta, \vec a) \leq
\diam B(\vec{\tilde a}, \rho)$.
Part \ref{it:b2}) is proved directly from the second equation in \eqref{eq:Q}, considering that
$\sqrt{a+b}\leq \sqrt{a} + \sqrt{b}$, that
$\rho(\rho-d)\leq \rho^2-d^2\leq 2\rho(\rho-d)$ and that $\sqrt{\gamma}\geq \gamma$ for all $\gamma\in (0,1)$.
Indeed, if $\cos \theta <0$ then 
\begin{align*}
  2d|\cos \theta| + \sqrt{2\rho(\rho-d)}
& \geq d|\cos \theta| + \sqrt{\rho^2-d^2} + \sqrt{d^2\cos^2\theta}
\geq q(\theta, \vecg \xi) \\
& \geq d|\cos \theta| + \sqrt{\rho^2-d^2}
 \geq d|\cos \theta| + \sqrt{\rho(\rho-d)}
\geq \rho |\cos \theta| \left ( \frac{d}{\rho} + \sqrt{1 - \frac{d}{\rho}} \right ).
\end{align*}
To prove \ref{it:b3}), suppose that $\vecg \zeta\cdot \vec e =\cos \theta > 0$ and rewrite \eqref{eq:Q} as 
\begin{align*}
\frac{q(\theta, \vecg \xi)}{\sqrt{\rho(\rho-d)}} =
\frac{1 + \frac{d}{\rho}}{
  \sqrt{\left ( 1  + \frac{d}{\rho}\right) + \frac{d^2 \cos^2 \theta}{\rho(\rho-d)}} + \frac{d\cos \theta}{\sqrt{\rho(\rho-d)}}}
\leq
\frac{2}{
  \sqrt{\left ( 1  + \frac{d}{\rho}\right) + \frac{d^2 \cos^2 \theta}{\rho(\rho-d)}}}
\leq
 \frac{2\sqrt{2}}{ \sqrt{ 1  + \frac{d}{\rho}} + \frac{d\cos \theta}{\sqrt{\rho(\rho-d)}}} .
\end{align*}
\end{proof}

\begin{figure}[thbp!]
  \centering
  \input{area.pstex_t} \qquad \qquad \input{ratios.pstex_t}
  \caption{\label{fig:rootd}Cone generated by $S$ and $\vec{\tilde a}+\rho \vec e$ (Lemma \ref{le:areal})}
\end{figure}

\begin{proof}[Proof of Lemma \ref{le:areal}]
 Call $\vec a := \vec{\tilde a} + (\rho - d)\vec e$. Consider the $(n-2)$-sphere
$S:= \{\vec x \in \partial B(\vec{\tilde a}, \rho): (\vec x-\vec a)\cdot \vec e=0\}$.
It is clear that $\Omega$ contains the cone generated by $\vec{\tilde a} + \rho \vec e$ (the `right-most' point on
$\partial B(\vec{\tilde a}, \rho)$) and $S$. Since the radius of $S$ (the `height') is given by
$h=\sqrt{d(2\rho-d)}$ (see Figure \ref{fig:rootd}) and the base measures $d$, the volume of the cone is
a constant times
$dh^{n-1}=d^{\frac{n+1}{2}}(2\rho-d)^\frac{n-1}{2}$. The value of the constant is obtained from
\begin{align*}
 |\Omega|\geq \frac{\mathcal H^{n-2}(\S^{n-2})}{n-1} \int_{\rho-d}^\rho \left (\frac{\rho-x_1}{d} \sqrt{\rho^2-(\rho-d)^2}\right )^{n-1}
\dd x_1
= \omega_{n-1} \sqrt{\rho(2\rho - d)}^{n-1} \frac{d}{n}.
\end{align*}
\end{proof}

\subsection{Numerical computations} \label{se:numerics}

The deformations depicted in Figure \ref{fig:transition} are obtained by the alternative method
of Dacorogna-Moser (constructive in nature and easier to implement, \cite[Sect.\ 4]{DaMo90}).
Following the notation in Theorem \ref{th:ub2} (and restricting now to the case $n=2$), let
$\rho(\theta):=\sqrt{R_1^2 + (v_1+v_2)\frac{q(\theta)^2}{|\Omega_1 \cup \Omega_2|}}$,
where $q(\theta)$ denotes the parametrization of $\partial (\overline{\Omega_1 \cup \Omega_2})$
using polar coordinates, taking $\vec a^*$ to be the origin.
Let also $0<R_1<R_2<R_3$ be such that $B(\vec a^*, R_1) \supset \Omega_1 \cup \Omega_2$ and $\pi R_3^2 = v_1+v_2 + \pi R_2^2$.
Given parametrizations $\vec w(s,t)$ and $\vec v(s,t)$, $(s,t)\in D:=[1, \sqrt{2}]\times[0,2\pi]$
 of $\{\vec x : R_1<|\vec x-\vec a^*|< R_2\}$ and of
$\left \{\vec y=\lambda \vec a^* + re^{i\theta}: \rho(\theta) < r < R_3\right \}$, respectively,
the strategy is to find an incompressible homeomorphism
 $\vec u: \vec w(Q)\to \vec v(Q)$ of the form
\begin{align*}
\vec u = \vec v\circ \vecg \phi_2 \circ \vecg \phi_1 \circ \vec w,
\quad \text{with} \quad
 \vecg \phi_1(s,t)=(h(s,t),t), \qquad
\vecg \phi_2(s,t)=(s, t+\eta(s)\beta(t)).
\end{align*}
 Here $\eta:[1, \sqrt{2}] \to \R$ is any function satisfying
\begin{align*}
 \fint_1^{\sqrt{2}} \eta(s) \dd s =1, \quad \eta(0)=\eta(1)=0, \quad 0\leq \eta \leq 1+\varepsilon, \quad
\fint_1^{\sqrt{2}} |1-\eta(s)|\dd s
\leq \varepsilon
\end{align*}
for some $\varepsilon\leq \min \left \{\frac{\min f}{2\max g}, \frac{\min g}{\max g}\right \}$, where
$f(s,t)=\det D\vec w(s,t)$
and $g(s,t)=\det D\vec v(s,t)$. The functions $\beta$ and $h$ are found
by defining $g_1(s_1, t_1):= g(\vecg\phi_2(s_1, t_1)) \det D \vecg \phi_2(s_1, t_1)$
and solving
\begin{align*}
 \int_1^{\sqrt{2}}\!\!\!\int_0^{t+\eta(\sigma)\beta(t)}\!g(\sigma, \tau)\dd\tau \dd\sigma =
 \int_1^{\sqrt{2}}\!\!\!\int_0^t\!f(s, \overline t) \dd\overline t  \dd s ,  \quad
\int_1^{h(s, t)}\!\!g_1(s_1, t) \dd s_1 = \int_1^s\!f(\overline s, t) \dd \overline s
\end{align*}
for every fixed $t\in [0,2\pi]$.
The solution is unique, and
for $\vec v$ and $\vec w$ as in \eqref{eq:vyw}, it is such that
$\int_{R_1<|\vec x-\vec a^*|< R_2} |D\vec u |^2 \leq C$, where
$C$ is an expression
that might possibly go to infinity only if
the target domain is too narrow, more precisely, if
$\displaystyle \frac{v_1+v_2}{\pi(R_2^2-R_1^2)}\left ( \frac{\pi q_{\max}^2}{|\Omega_1 \cup \Omega_2|} -1 \right )
\nearrow 1$,
(recall that $\frac{\pi q_{\max}^2}{|\Omega_1\cup \Omega_2|} -1$ is of the order of $1-\delta$,
equations \eqref{eq:lbqd} and \eqref{eq:ubqd}).
In our computations we choose $R_1=q_{\max}=2\rho_1-d\delta$
and $R_2$ such that $\pi(R_2^2-R_1^2)= 2(v_1+v_2) \left ( \frac{\pi q_{\max}^2}{|\Omega_1 \cup \Omega_2|} -1 \right )$.

\section{Proof of the convergence result, Theorem \ref{th:limits}} \label{se:limits}

We follow the strategy of Struwe \cite{Struwe94}
to prove that $\sup_\ep \|\vec u_\ep\|_{W^{1,p}(\Omega_{\ep})} < \infty$
for all $p<n$.
Fix $\ep>0$, call $\mathcal B_{0}:= \bigcup_{i=1}^m \overline B_{\ep}(\vec a_{i,\ep})$,
$t_{0}:= r(\mathcal B_{0})= m \ep$, and let
$\{\mathcal B(t): t\geq t_{0}\}$
be the family obtained by applying Proposition \ref{pr:ballconstruction} to $\mathcal B_{0}$.
Define $\rho= \sup \{t\geq t_0: \bigcup \mathcal B(t) \subset \Omega \}$
and write
$\mathcal C_{k}:= \bigcup \mathcal B(r_k) \setminus \bigcup \mathcal B(r_{k+1})$, $r_k:=2^{-k} \rho$.
By using H\"older's inequality, then comparing the lower bound of Proposition \ref{pr:bc}, to the upper bound,
 we find that for every $p< n$
\begin{align*}
 \int_{\mathcal C_{k}} |D\vec u_\ep|^p \dd\vec x
&
  \leq
 C(n,p) \rho^{n-p} 2^{-(n-p)k}
    \left ( \frac{1}{n}\int_{\Omega_{\ep}} \left |\frac{D\vec u_\ep }{\sqrt{n-1}}\right |^n  \dd\vec x
      - \sum_{i=1}^m v_{1, \ep} \log \frac{r_{k+1}}{t_0}\right )^\frac{p}{n}
\\ & \leq
  C\rho ^{n-p} 2^{-(n-p)k} \left (|\Omega| + \sum_{i=1}^m v_{i,\ep}\right)^\frac{p}{n}
    \left ( C + \log \frac{\diam \Omega}{\rho/m} + (k+1)\log 2 \right )^\frac{p}{n}.
\end{align*}
Adding over $k$ we find that
\begin{align*}
 \int_{\Omega_\ep}  |D\vec u_\ep|^p \dd\vec x
&\leq C \rho^{n-p}
  \left (|\Omega| + \sum_{i=1}^m v_{i,\ep}\right)^\frac{p}{n}
\left ( \sum_{k=1}^\infty \frac{(C+k\log 2)^\frac{p}{n}}{2^{(n-p)k}}
    + \frac{\left (\log \frac{\diam \Omega}{\rho/m}\right )^\frac{p}{n}}{2^{n-p}-1} \right ) \\
&\quad + n^\frac{p}{n} (n-1)^\frac{p}{2} |\Omega|^{1-\frac{p}{n}} \left (
\frac{1}{n}\int_{\Omega_{\ep}} \left |\frac{D\vec u_\ep }{\sqrt{n-1}}\right |^n  \dd\vec x
      - \sum_{i=1}^m v_{i, \ep} \log \frac{\rho}{m \ep}\right )^\frac{p}{n}
\\ & \leq
C \left (\rho^{n-p} + |\Omega|^{\frac{n-p}{n}}\right )
\left (|\Omega| + \sum_{i=1}^m v_{i,\ep}\right)^\frac{p}{n}\left ( C+ \log \frac{\diam \Omega}{\rho/m} \right )^{\frac{p}{n}}.
\end{align*}
It can be seen (as in the proof of Proposition \ref{pro1})
that $\rho\geq \frac{1}{2} \dist (\{\vec a_{1,\ep}, \ldots, \vec a_{m,\ep}\}, \partial \Omega)$.
Hence, in order to prove that $\sup_\ep \|D\vec u_\ep\|_{L^p}< \infty$
it only remains to show that $\sum_{i=1}^m v_{i, \ep}$ is uniformly bounded.
Choose $r>\ep$ such that
the balls $\overline B(\vec a_{i, \ep},r)$ are disjoint and
$r \in R_{\vec a_{i, \ep}}$ for all $i=1, \ldots, m$. By
Proposition \ref{pr:cof},
the topological images
$E(\vec a_{i,\ep}, r; \vec u_\ep)$
are disjoint,
contained in $B(\vec 0, \|\vec u_\ep\|_{L^\infty(\Omega_{\ep})})$
(because $E(\vec a_{i,\ep}, r; \vec u_\ep)$ is the region enclosed by $\vec u(\partial B(\vec a_{i, \ep}, r))$),
 and
such that
$E(\vec a_{i,\ep}, \ep; \vec u_\ep) \subset E(\vec a_{i,\ep}, r; \vec u_\ep)$.
Therefore
\begin{align*}
 \sum_{i=1}^m (v_{i,\ep} + \omega_n \ep^n ) =
\sum_{i=1}^m |E(\vec a_{i,\ep}, \ep; \vec u_\ep)|
\leq
\left |\bigcup_{i=1}^m E(\vec a_{i,\ep}, r; \vec u_\ep)\right |
\leq \omega_n \|\vec u_\ep\|_{L^\infty(\Omega_\ep)}^n.
\end{align*}
Since we are assuming that $\sup_\ep \|\vec u_\ep\|_{L^\infty(\Omega_\ep)} < \infty$,
we obtain that
$\sup_\ep \|\vec u_\ep\|_{W^{1,p}(\Omega_{\ep})} < \infty$, as desired.

For the existence of a limit map and for the convergence in
$W^{1,n}_{\loc} (\Omega \setminus \{\vec a_1, \ldots \vec a_m\}, \R^n)$,
let $\delta>0$ be small,
assume that $|\vec a_{i,\ep} - \vec a_i| < \delta/2$ for all $i=1, \ldots, m$,
 and consider the following energy bound,
obtained again by comparing \eqref{eq:hypUB}
with the lower bound of Proposition \ref{pr:bc} (applied to $s=\delta/2$)
\begin{align*}
 \frac{1}{n}\int_{\Omega \setminus \bigcup \mathcal B(\delta/2)}
\left |\frac{D\vec u}{\sqrt{n-1}} \right|^n \dd\vec x
\leq \sum_{i=1}^m v_{i, \ep} \log \frac{\diam \Omega}{\delta/2m} + C\left ( |\Omega| +
\sum_{i=1}^m v_{i, \ep} \right ).
\end{align*}
Since $r(\mathcal B(\delta/2))=\delta/2$, it follows that $\{\vec u_{\ep_j}\}_{j\in \N}$
is bounded in $W^{1,n}(\Omega \setminus \bigcup_{i=1}^m \overline B_{\delta}(\vec a_i), \R^n)$.
From this, and since $\delta>0$ is arbitrary,
the existence of $\vec u$ and of a convergent subsequence
follows by standard arguments
(see, e.g., \cite{SiSp06} or \cite{Henao09}):
inductively take succesive subsequences of $\{\vec u_{\ep_j}\}_{j\in \N}$
(for some sequence $\delta_k\to 0$) converging weakly in
$W^{1,n}(\Omega \setminus \bigcup_{i=1}^m \overline B_{\delta_k}(\vec a_i), \R^n)$.
Choose then a diagonal sequence $\{\vec u_{\ep_k}\}_{k\in \N}$
converging weakly in $W^{1,n}(\Omega \setminus \bigcup_{i=1}^m \overline B_{\delta}(\vec a_i), \R^n)$
for every $\delta>0$, to some $\vec u \in W^{1,n}_{\loc} (\Omega \setminus \{\vec a_1, \ldots \vec a_m\}, \R^n)$.

Since $\sup_\ep \|\vec u_{\ep}\|_{W^{1,p}(\Omega_{\ep})} < \infty$
for all $p<n$,
the maps $\vec u_{\ep}$ can be extended, by multiplying them by suitable cut-off
functions $\psi_\ep$, inside the holes $\overline B(\vec a_{i, \ep}, \ep)$,
in such a way that $\sup_\ep \|\psi_\ep \vec u_{\ep} \|_{W^{1,p}(\Omega)} < \infty$.
It is easy to see that any weakly convergent subsequence of $\{\psi_{\ep_k} \vec u_{\ep_k}\}_{k\in \N}$
must converge
to the limit map $\vec u$ defined above; this proves that $\vec u \in W^{1,p}(\Omega, \R^n)$
for all $p<n$.

By the classical result of Reshetnyak \cite[Thm.~4]{Reshetnyak67a} and
Ball \cite[Cor.~6.2.2]{Ball77}, $\cof D\vec u_{\ep_k} \weakc \cof D\vec u$
in $L^{\frac{n}{n-1}}_{\loc} (\Omega \setminus \{\vec a_1, \ldots, \vec a_m\}, \R^{n\times n})$.
By the definition of $\Det D\vec u$ in \eqref{eq:defDet}, and since $\{\Det D\vec u_{\ep}\}_{\ep>0}$
is bounded as a sequence in the space of measures ($\Det D\vec u_{\ep} = \mathcal L^n \res \Omega_{\ep}$, by hypothesis),
it follows that $\Det D\vec u$ coincides with $\mathcal L^n$
in
$\Omega\setminus \{\vec a_1, \ldots, \vec a_m\}$,
and that
$\Det D\vec u_{\ep} \weakcs \Det D\vec u$
in
$\Omega\setminus \{\vec a_1, \ldots, \vec a_m\}$ in the sense of measures.
Moreover, by \cite[Lemma 3.2]{SiSp00}
(applied to $\Omega \setminus \bigcup_{i=1}^m \overline B(\vec a_i, \delta)$
instead of $\Omega$), we obtain that $\det D\vec u(\vec x)=1$ for a.e. $\vec x \in \Omega \setminus \{\vec a_1, \ldots, \vec a_m\}$.

From Definition \ref{df:INV} and from the proof of \cite[Lemma 4.2]{Henao09}
it follows that the limit map $\vec u$ satisfies condition INV.
Proposition \ref{pr:Det}
 then implies that $\Det D\vec u = \mathcal L^n + \sum_{i=1}^m c_i \delta_{\vec a_i}$
for some coefficients $c_i\in \R$, and the proof of the same proposition also shows that
\begin{align*}
\frac{1}{n} \int_{\partial B(\vec a_i, r)} \vec u_\ep \cdot (\cof D\vec u_\ep)\vecg \nu\dd\mathcal H^{n-1}
&=
\omega_n r^n + \sum_{j: \vec a_{j,\ep} \in B(\vec a_i, r)} v_{j, \ep}
\\
 \frac{1}{n} \int_{\partial B(\vec a_i, r)} \vec u \cdot (\cof D\vec u)\vecg \nu\dd\mathcal H^{n-1}
&=
\omega_n r^n + \sum_{j: \vec a_j \in B(\vec a_i, r)} c_j
\end{align*}
for a.e. $r>0$ such that $\partial B(\vec a_i, r)\subset \Omega$
(note that if $\vec a_i= \vec a_j$ for some $i\ne j$, then the choice of the coefficients $c_i$ is not unique).
By standard arguments, for every $\delta>0$ there exists $r< \delta$ such that
$\vec u_{\ep_k}\to \vec u$ uniformly on $\partial B(\vec a_i, r)$
and $\cof D\vec u_{\ep_k} \weakc \cof D\vec u$ in $L^{\frac{n}{n-1}}(\partial B(\vec a_i, r))$
(passing, if necessary, to a subsequence that may depend on $r$). Taking, first, the limit as
$\ep\to 0$, then the limit as $r\to 0$, we obtain that
$\Det D\vec u = \mathcal L^n + \sum_{i=1}^m v_i \delta_{\vec a_i}$.

Consider now the case of two cavities.
Set $\vec a_\ep:= \frac{\vec a_{1,\ep} + \vec a_{2,\ep}}{2}$, $d_\ep:=|\vec a_{2,\ep} - \vec a_{1,\ep}|$.
\begin{enumerate}[{\it i)}]
  \item \emph{Suppose that $v_1\geq v_2>0$ and $d=|\vec a_2 - \vec a_1|>0$.}
By Lemma \ref{le:mod} we have that for all $r> \ep$
\begin{align*}
\left  | |E(\vec a_{i, \ep}, r; \vec u_{\ep})| D(E(\vec a_{i, \ep}, r; \vec u_{\ep}))^\frac{n}{n-1}
- |E(\vec a_{i, \ep}, \ep; \vec u_{\ep})| D(E(\vec a_{i, \ep}, \ep; \vec u_{\ep}))^\frac{n}{n-1} \right |
\leq 2^\frac{n}{n-1} {\textstyle \frac{n+1}{n-1}} \omega_n r^n,
\end{align*}
hence, by \eqref{eq:eq0proof}, for all $\alpha \in (0, 1)$
and all $R<\min\{\frac{d}{2}, \dist(\{\vec a_1, \vec a_2\}, \partial \Omega)\}$
we have that
\begin{align*}
\frac{ \int_{\Omega_{\ep}} \frac{1}{n}\left |\frac{D\vec u(\vec x)}{\sqrt{n-1}} \right |^n}
{|\log \ep|}
& \geq \frac{ \sum_{i=1}^2 \left ( \int_\ep^{\ep^\alpha} + \int_{\ep^\alpha}^R \right )
\int_{\partial B(\vec a_{i, \ep}, r)}
\frac{1}{n}\left |\frac{D\vec u(\vec x)}{\sqrt{n-1}} \right |^n \dd\mathcal H^{n-1}\dd r }{|\log \ep|}
\\
&\geq
\sum_{i=1}^2\left (v_{i, \ep} \frac{\log (R/\ep)}{|\log \ep|}
+ (1-\alpha)C \left ( |E(\vec a_{i, \ep}, \ep; \vec u_{\ep})| D(E(\vec a_{i, \ep}, \ep; \vec u_{\ep}))^\frac{n}{n-1}
-\ep^{\alpha n} \right ) \right).
\end{align*}
Combining this with \eqref{eq:hypUB} we obtain
\begin{align*}
 \sum_{i=1}^2 v_{i,\ep} D\big ( E(\vec a_{i,\ep}, \ep; \vec u_{\ep} ) \big )^\frac{n}{n-1}
 \leq
\frac{(|\Omega| + v_{1,j} + v_{2,j} )\left
(C_2 + \log \frac {\diam \Omega}{R}\right )}{C_1|\log \ep^{1-\alpha}|} +
C\ep^{\alpha n}.
\end{align*}
Therefore, as $\ep\to 0$,  $D\big(E(\vec a_{i,\ep}, \ep; \vec u_{\ep}) \big )\to 0$
(i.e.,~$\vec u_\ep$ tends to create spherical cavities).

As mentioned before, for every $\delta>0$ there exists $r< \delta$ such that $\vec u_{\ep}|_{\partial B(\vec a_i, r)}$
converges uniformly, for each $i=1,2$, to $\vec u |_{\partial B(\vec a_i, r)}$ (passing to a subsequence, if necessary).
By continuity of the degree, this implies that $\imT(\vec u, \vec a_i)$ is contained in $E(\vec a_i, r; \vec u_{\ep})$
for sufficienty small $\ep$.
In particular, by definition of $v_{i,\ep}$ and Proposition \ref{pr:Det},
\begin{align*}
 |E(\vec a_i, r; \vec u_\ep ) \triangle \imT(\vec u, \vec a_i)|
= |E(\vec a_i, r; \vec u_\ep)| - |\imT(\vec u, \vec a_i)|
= (v_{i, \ep} + \omega_n r^n) - v_i.
\end{align*}
On the other hand,
$B(\vec a_{i, \ep}, \ep) \subset B(\vec a_i, r)$ for sufficiently small $\ep$.
By Proposition \ref{pr:cof} this implies that
$E(\vec a_{i,\ep},  \ep; \vec u_{\ep}) \subset E(\vec a_i, r; \vec u_\ep)$,
so,
proceeding as in the proof of Proposition \ref{pr:Det},
we obtain
\begin{align*}
|E(\vec a_{i,\ep},  \ep; \vec u_{\ep}) \triangle E(\vec a_i, r; \vec u_\ep)|
&= \Det D\vec u ( B(\vec a_i, r) \setminus B(\vec a_{i,\ep}, \ep))
=|B(\vec a_i, r) \setminus B(\vec a_{i,\ep}, \ep)| < \omega_n \delta^n.
\end{align*}
Thus,
\begin{align} \label{eq:convImT}
 &\limsup_{\ep\to 0} |E(\vec a_{i, \ep}, \ep; \vec u_{\ep}) \triangle \imT(\vec u, \vec a_i)|
\\ \nonumber
 &\leq \limsup_{\ep\to 0} ( |E(\vec a_{i,\ep},  \ep; \vec u_{\ep})
\triangle E(\vec a_i, r; \vec u_\ep)| + |E(\vec a_i, r; \vec u_\ep ) \triangle \imT(\vec u, \vec a_i)|)
\leq 2\omega_n \delta^n
\end{align}
for all $\delta>0$, that is, the cavities formed by $\vec u_{\ep}$ converge to the cavities formed by $\vec u$.

It remains to prove the estimate for $|\vec a_2-\vec a_1|$ in terms of $|\Omega|$, $\diam \Omega$ and the cavity volumes,
assuming that $v_1+v_2 < 2^n  \omega_n (\dist(\frac{\vec a_1+\vec a_2}{2}, \partial \Omega ))^n$.
Let $R>0$ be such that 
$v_{1, \ep} + v_{2, \ep} < \omega_n (2R)^n$ and $B(\vec a_\ep, R) \subset \Omega$
for every sufficiently small $\ep$.
Suppose first that 
\begin{align} \label{eq:assumptionsT3i}
\frac{\omega_n d^n}{v_1+v_2} < \frac{1}{2^n} \left ( \frac{v_2}{v_1+v_2} \right )^\frac{n}{n-1}.
\end{align}
Since $\frac{v_2}{v_1+v_2}<1$, this implies, in particular, that
$v_{1, \ep} + v_{2, \ep} > \omega_n (2d_\ep)^n$  
 for every small $\ep$.
As a consequence, $\frac{R}{d_{\ep}}>1$ and 
$\left (
\frac{v_{1, \ep} + v_{2, \ep}}{2^n \omega_n d_{\ep}^n}
\right )^{\frac{1}{n^2}} <
\left (\frac{R}{d_{\ep}} \right )^\frac{1}{n} < \frac{R}{d_{\ep}}$, that is,
the minimum at the end of Theorem \ref{th:LB} is attained at 
$\left (\frac{v_{1, \ep} + v_{2, \ep}}{2^n \omega_n d_{\ep}^n}\right ) ^\frac{1}{n^2}$
(it cannot be attained at $\frac{d_{\ep}}{\ep}$ since $d_{\ep}\to d>0$).
By Theorem  \ref{th:LBn} and \eqref{eq:hypUB},
\begin{align*}
& C_1\left (\left ( \frac{v_{2,\ep}}{v_{1,\ep}+v_{2,\ep}} \right )^\frac{n}{n-1}
- \frac{\omega_n d_{\ep}^n}{v_{1,\ep}+v_{2,\ep}} \right )_+
\log \frac{v_{1,\ep}+v_{2,\ep}}{2^n \omega_n d_\ep^n} 
\leq
\frac{\frac{1}{n} \int_{\Omega_{\ep}} \left |\frac{D\vec u}{\sqrt{n-1}} \right | \dd \vec x - 
(v_{1, \ep} + v_{2, \ep}) \log \frac{R}{2\ep}}{v_{1, \ep} + v_{2, \ep}}\\
&\leq 
C_2 \left ( 1 + \frac{|\Omega|}{v_{1,\ep} + v_{2, \ep}}
+ \log \frac{\omega_n (\diam \Omega)^n}{\omega_n R^n } \right ) 
\leq C_2 \left ( 1 + \frac{|\Omega|}{v_{1,\ep} + v_{2, \ep}}
+ \log \frac{\omega_n (\diam \Omega)^n}{v_{1, \ep} + v_{2, \ep}} \right )
\end{align*}
(in the last step we use that $\omega_n R^n > \frac{v_{1, \ep} + v_{2, \ep}}{2^n}$, by the choice of $R$).
If \eqref{eq:assumptionsT3i} holds then 
the factor in front of the logarithm is positive
for $\ep>0$ small; taking the limit 
we obtain that $\frac{\omega_n d^n}{v_1+v_2} \geq 2^{-n} F(\Omega, v_1, v_2)$, with 
\begin{align}
\label{eq:defF}
F(\Omega, v_1,v_2):= \exp \left (
      - C
\left . \left ( {\displaystyle
1 + \frac{|\Omega|}{v_1 + v_2}
+ \log \frac{\omega_n (\diam \Omega)^n}{v_1+v_2}}
\right ) \right /
      {\displaystyle \left ( \frac{v_{2}}{v_{1}+v_{2}} \right )^\frac{n}{n-1}}\right ).
\end{align}

If \eqref{eq:assumptionsT3i} does not hold, we still have that 
$\frac{\omega_n d^n}{v_1+v_2}\geq C F(\Omega, v_1, v_2)$ for some constant $C(n)$. To see this, 
recall that $v_1+v_2 < 2^n \omega_n \dist(\frac{\vec a_1 + \vec a_2}{2}, \partial \Omega))^n
< \omega_n (2\diam \Omega)^n$ (by hypothesis),
hence
\begin{align*}
 F(\Omega, v_1, v_2) \leq \exp \left (- 
{ 
C(1 + n|\log 2|)}
 \left /
{ \left ( \frac{v_2}{v_1+v_2} \right )^\frac{n}{n-1}} \right .
\right )
\leq \left . \left ( \frac{v_2}{v_1+v_2} \right )^\frac{n}{n-1} \right / C(1+n|\log 2|)
\end{align*}
(we have used that $e^\frac{1}{x}\geq \frac{1}{x}$ for all $x\geq 0$). The proof is completed since 
the above implies that
\begin{align*}
 \frac{\omega_n d^n}{v_1+v_2} \geq 2^{-n} \left ( \frac{v_2}{v_1+v_2} \right )^\frac{n}{n-1}
\quad \Rightarrow \quad 
 \frac{\omega_n d^n}{v_1+v_2} 
\geq 2^{-n} C(1+n|\log 2|) F(\Omega, v_1, v_2).
\end{align*}

  \item \emph{Suppose that $v_1>v_2=0$.}
Applying Proposition \ref{pr:ballconstruction} to $\mathcal B_0:=\{\overline B_\ep(\vec a_{1, \ep}),
\overline B_{\ep}(\vec a_{2, \ep})\}$
we obtain $\mathcal B(t)= \{B(\vec a_{1, \ep}, t/2), B(\vec a_{2, \ep}, t/2)\}$ for $t\in (2\ep, d_{\ep})$, and
$\mathcal B(t)=\{B(\vec a_\ep, t)\}$ for $t\geq d_\ep$.
We claim that if $R < \frac{2}{3}\dist(\{\vec a_{1,\ep}, \vec a_{2, \ep}\}, \partial \Omega)$
then $\bigcup \mathcal B(R)\subset \Omega$. Indeed,
if $R < d_{\ep}$, this holds automatically. If $R\geq d_{\ep}$, then
\begin{align*}
 \frac{3R}{2} < \dist (\vec a_{1, \ep}, \partial \Omega) \leq \frac{d_\ep}{2} + \dist(\vec a_\ep, \partial \Omega)
\leq \frac{R}{2} + \dist(\vec a_\ep, \partial \Omega)\ \Rightarrow\ B(\vec a_\ep, R) \subset \Omega.
\end{align*}
Therefore, by Proposition \ref{pr:bc} and Lemma \ref{le:mod}, for every $\alpha\in (0,1)$
\begin{align*}
& |E(\vec a_{1,\ep}, \ep; \vec u_\ep)| D\big (E(\vec a_{1,\ep}, \ep; \vec u_\ep)\big )^\frac{n}{n-1}
 \log \frac{\ep^\alpha}{2\ep}
\\ & \leq
  \int_{\Omega_\ep}\frac{1}{n}\left|\frac{D\vec u_\ep}{\sqrt{n-1}}\right|^n \dd \vec x -
    (v_{1,\ep} + v_{2,\ep}) \log \frac{R}{2\ep}
  + 2^\frac{n}{n-1} {\textstyle \frac{n+1}{n-1}} (v_{2,\ep} + \omega_n \ep^{\alpha n})
      \log \frac{\ep^\alpha}{2\ep}.
\end{align*}
By virtue of \eqref{eq:hypUB} and again Lemma \ref{le:mod},
\begin{align*}
   v_1 D\big (\imT(\vec u, \vec a_1)\big )^\frac{n}{n-1}
 \leq
   2^\frac{n}{n-1}{\textstyle  \frac{n+1}{n-1}}
\lim_{\ep\to 0} (v_{2,\ep} + \omega_n \ep^{\alpha n}
      + |E(\vec a_{1,\ep}, \ep; \vec u_j ) \triangle \imT(\vec u, \vec a_1)|).
\end{align*}
Proceeding as in \eqref{eq:convImT} we find that
\begin{align*}
\limsup_{\ep\to 0} |E(\vec a_{1,\ep}, \ep; \vec u_\ep ) \triangle \imT(\vec u, \vec a_1)|
\leq 2(v_2 + \omega_n r^n)
\end{align*}
for arbitrarily small values of $r>0$, proving that $\imT(\vec u, \vec a_1)$ is a ball.

  \item \emph{Suppose that $v_1\geq v_2>0$ and $\vec a_1=\vec a_2$.}
Let $R>0$ be such that $B(\vec a_{\ep}, R)\subset \Omega$ for all $j\in \N$.
Since $\lim d_{\ep} = |\vec a_2-\vec a_1|=0$,
\eqref{eq:eq1proof} and \eqref{eq:hypUB} imply that
\begin{align*}
 \limsup_{\ep\to 0} \frac{
\int_{d_\ep}^{R} |E(\vec a_\ep, r; \vec u_\ep)| D\big (E(\vec a_\ep, r; \vec u_\ep)\big )^\frac{n}{n-1} \frac{\dd r}{r}}
{\log d_\ep}
\leq
  C\frac{(|\Omega| + v_{1} + v_{2})\left (1+\log \frac{\diam \Omega}{R/2}\right)}{
\displaystyle \lim_{\ep\to 0} \log d_\ep}=0.
\end{align*}
For $\alpha \in (0,1)$ fixed and $\ep$ small
$B(\vec a_\ep, d_\ep)\subset B(\vec a_\ep, d_\ep^\alpha) \subset \Omega$.
By Lemma \ref{le:mod}, for all $r\in (d_\ep, d_\ep^\alpha)$
\begin{align*}
 \Big | |E(\vec a_\ep, r; \vec u_\ep)| D\big (E(\vec a_\ep, r; \vec u_\ep)\big )^\frac{n}{n-1}
- |E(\vec a_\ep, d_\ep; \vec u_\ep)| D\big (E(\vec a_\ep, d_\ep; \vec u_\ep)\big )^\frac{n}{n-1} \Big | \leq
2^\frac{n}{n-1} \frac{n+1}{n-1} \omega_n d_\ep^{\alpha n}.
\end{align*}
Dividing $\int_{d_\ep}^{d_\ep^\alpha} |E(\vec a_\ep, d_\ep; \vec u_\ep)|
D\big (E(\vec a_\ep, d_\ep; \vec u_\ep)\big )^\frac{n}{n-1}
\frac{\dd r}{r}$ by $\log d_\ep^{\alpha-1}$ we obtain
\begin{align} \label{eq:D=0out}
 \limsup_{\ep\to 0} |E(\vec a_\ep, d_\ep; \vec u_\ep)|
D\big (E(\vec a_\ep, d_\ep; \vec u_\ep)\big )^\frac{n}{n-1}
\leq
 \limsup_{\ep\to 0}
    2^\frac{n}{n-1} \frac{n+1}{n-1} \omega_n d_\ep^{\alpha n} =0.
\end{align}
Proceeding as in \eqref{eq:convImT}, it can be proved that
\begin{align} \label{eq:c-c3iii}
 \limsup_{\ep \to 0} \big |\imT(\vec u,\vec a_1)\triangle E(\vec a_{\ep}, d_\ep; \vec u_\ep)|
\leq \limsup_{\ep \to 0} (v_{1, \ep} + v_{2, \ep}) - |\imT(\vec u, \vec a_1)|.
\end{align}
Because of the continuity of the distributional determinant,
$|\imT(\vec u,\vec a_1)|=v_1+v_2$, hence
$D\big (\imT(\vec u,\vec a_1)\big )=0$
(by \eqref{eq:c-c3iii}, Lemma \ref{le:mod}\ref{cd:mod2}), and \eqref{eq:D=0out}).

In order to prove that at least one of the limit cavities must be distorted, we proceed as in the proof of
Theorem \ref{th:LBn} by applying Proposition \ref{pr:distortions} to $E_{1}= E(\vec a_{1,\ep}, \ep; \vec u_\ep)$,
$E_{2} = E(\vec a_{2,\ep}, \ep; \vec u_\ep)$, and $E=E(\vec a_\ep, d_\ep; \vec u_\ep)$. Again we define
$g(\beta_1, \beta_2):= (\beta_1^\frac{1}{n} + \beta_2^\frac{1}{n})^n - (\beta_1 + \beta_2)$ and note that it is
increasing in its two variables. It is easy to see that
\begin{align*}
 \frac{(|E_1|^\frac{1}{n}+|E_2|^\frac{1}{n})^n -|E|}{(|E_1|^\frac{1}{n}+|E_2|^\frac{1}{n})^n - |E_1\cup E_2|}
\geq 1 - \frac{\omega_n d_\ep^n}{g(v_{1,\ep}, v_{2,\ep})} \overset{\ep\to0}{\longrightarrow} 1.
\end{align*}
Therefore,
\begin{align*}
 \liminf_{\ep\to 0}
\frac{|E|D(E)^\frac{n}{n-1} + |E_{1}|D(E_{1})^\frac{n}{n-1} + |E_{2}|D(E_{2})^\frac{n}{n-1}}
{|E| + |E_{1} \cup E_{2}|} \geq C \left (\frac{v_2}{v_1+v_2}\right )^\frac{n}{n-1}.
\end{align*}
Property \eqref{eq:minDist} follows from \eqref{eq:D=0out}.
On the other hand, \eqref{eq:eq1proof}, \eqref{eq:hypUB}, and Lemma \ref{le:mod} imply that
\begin{multline*}
 \sum_{i=1}^2 \int_{\ep}^{\min\{\frac{d_\ep}{2}, \ep^\alpha\}}
    C\left ( v_{i,\ep} D\big (E_{i}\big )^\frac{n}{n-1} - 2^\frac{n}{n-1}
{\textstyle\frac{n+1}{n-1}} \omega_n
	      \min\{\frac{d_\ep^n}{2^n}, \ep^{\alpha n}\} \right ) \frac{\dd r}{r}
\\ \leq	
  (v_{1,\ep} + v_{2,\ep})\log \frac{\diam \Omega}{R/2} + C(v_{1,\ep} + v_{2,\ep}+|\Omega|) .
\end{multline*}
for every fixed $\alpha \in (0,1)$. Hence,
\begin{multline*}
  \limsup_{\ep\to 0} \left ( \min \left \{ \log \frac{d_\ep}{2\ep},
 \log \ep^{\alpha-1}\right \} \right )
  \leq \frac{ C\displaystyle \left ( \log \frac{\diam \Omega}{R/2} + 1 + \frac{|\Omega|}{v_1+v_2}\right )}
{ \displaystyle \liminf_{\ep\to 0}
\left ( \frac{v_{1,\ep} D(E_{1})^\frac{n}{n-1}+ v_{2,\ep} D(E_{2})^\frac{n}{n-1}}{v_{1,\ep}+v_{2,\ep}}
    - \ep^{\alpha n} \right )
}.
\end{multline*}
By virtue of \eqref{eq:minDist}, and since $|\log \ep|\to \infty$, we conclude that
$\displaystyle \limsup_{\ep\to 0} d_\ep/\ep$ is finite.

\end{enumerate}

\bibliography{Biblio} \bibliographystyle{acm}

\noindent
{\sc Sylvia Serfaty}\\
UPMC Univ. Paris 06, UMR 7598 Laboratoire Jacques-Louis Lions,\\
 Paris, F-75005 France;\\
 CNRS, UMR 7598 LJLL, Paris, F-75005 France \\
 \&  Courant Institute, New York University\\
251 Mercer St., New York, NY  10012, USA\\
{\tt serfaty@ann.jussieu.fr}
\vskip 1cm

\noindent
{\sc Duvan Henao}\\
UPMC Univ. Paris 06, UMR 7598 Laboratoire Jacques-Louis Lions,\\
 Paris, F-75005 France;\\
 CNRS, UMR 7598 LJLL, Paris, F-75005 France \\
{\tt henao@ann.jussieu.fr}

\end{document}